\documentclass[12pt,letterpaper,reqno]{amsart}
\usepackage[foot]{amsaddr}
\makeatletter
\renewcommand{\email}[2][]{%
  \ifx\emails\@empty\relax\else{\g@addto@macro\emails{,\space}}\fi%
  \@ifnotempty{#1}{\g@addto@macro\emails{\textrm{(#1)}\space}}%
  \g@addto@macro\emails{#2}%
}
\makeatother
\usepackage{epsfig}
\usepackage{amsmath}
\usepackage{amssymb}
\usepackage{amsthm}
\usepackage{indentfirst}
\usepackage{xspace}
\usepackage{xcolor}
\usepackage{setspace}
\usepackage{verbatim}
\usepackage[letterpaper,margin=1in,headheight=15pt]{geometry}
\usepackage{mathpazo}
\usepackage{booktabs}
\usepackage{float} 
\definecolor{shadecolor}{rgb}{0.85,0.85,0.85}
\usepackage{bibentry}
\usepackage{bm}
\usepackage{hyperref}
\definecolor{darkred}{rgb}{0.5,0.15,0.15}
\hypersetup{colorlinks=true,urlcolor=darkred,linkcolor=darkred,citecolor=darkred}

\allowdisplaybreaks

\newtheorem{thm}{Theorem}
\newtheorem{cor}[thm]{Corollary}
\newtheorem{conj}[thm]{Conjecture}
\newtheorem{lem}[thm]{Lemma}
\newtheorem{prop}[thm]{Proposition}

\theoremstyle{definition}

\newtheorem{rem}[thm]{Remark}

\numberwithin{thm}{section}
\numberwithin{equation}{section}
\numberwithin{figure}{section}

\newcommand{\fsl}{{\mathfrak sl}}

\newcommand{\cN}{\ensuremath{\mathcal N}}
\newcommand{\cC}{\ensuremath{\mathcal C}}
\newcommand{\cD}{\ensuremath{\mathcal D}}
\newcommand{\cG}{\ensuremath{\mathcal G}}
\newcommand{\cB}{\ensuremath{\mathcal B}}
\newcommand{\cL}{\ensuremath{\mathcal L}}

\newcommand{\cF}{\ensuremath{\mathcal F}}

\newcommand{\cM}{\ensuremath{\mathcal M}}
\newcommand{\cO}{\ensuremath{\mathcal O}}

\newcommand{\cI}{\ensuremath{\mathcal I}}

\newcommand{\cV}{\ensuremath{\mathcal V}}

\newcommand{\R}{\ensuremath{\mathbb R}}
\newcommand{\C}{\ensuremath{\mathbb C}}
\newcommand{\CP}{\ensuremath{\mathbb {CP}}}

\newcommand{\Z}{\ensuremath{\mathbb Z}}

\newcommand{\calC}{\mathcal C}
\newcommand{\calF}{\mathcal F}

\newcommand{\ellt}{\ell_t}
\newcommand{\mt}{m_t}

\newcommand{\cE}{{\mathcal E}}

\newcommand{\delbar}{\ensuremath{\overline{\partial}}}
\newcommand{\zbar}{\ensuremath{\overline{z}}}

\newcommand{\Id}{\ensuremath{\mathrm{Id}}}

\newcommand{\semif}{\ensuremath{\mathrm{sf}}}
\newcommand{\app}{\ensuremath{\mathrm{app}}}
\newcommand{\model}{\ensuremath{\mathrm{model}}}
\newcommand{\odd}{\ensuremath{\mathrm{odd}}}

\newcommand{\Det}{\ensuremath{\mathrm{Det}}}

\def\D{\mathbb{D}}

\newcommand{\CC}{\mathbb C}

\newcommand{\smallvee}{\mathrel{\text{\raisebox{0.25ex}{\scalebox{0.8}{$\vee$}}}}}
\newcommand{\I}{{\mathrm i}}
\newcommand{\e}{{\mathrm e}}
\newcommand{\de}{\mathrm{d}}

\newcommand{\norm}[1]{\left\lVert#1\right\rVert}
\newcommand{\IP}[1]{\langle#1\rangle}

\newcommand{\eps}{\epsilon}
\newcommand{\har}{{\hat{r}}}

\newcommand{\del}{{\partial}}

\newcommand{\Fr}{\mathrm{Fr}}

\DeclareMathOperator{\End}{End}

\DeclareMathOperator{\charp}{char}
\DeclareMathOperator{\ParEnd}{ParEnd}

\DeclareMathOperator{\GL}{GL}
\DeclareMathOperator{\SL}{SL}
\DeclareMathOperator{\PSL}{PSL}

\DeclareMathOperator{\UU}{U}
\DeclareMathOperator{\SU}{SU}
\DeclareMathOperator{\Sp}{Sp}

\newcommand{\fixme}[1]{{\color{blue}{[#1]}}}

\begin{document}
\onehalfspacing

\title{Asymptotic Geometry of the Moduli Space of Parabolic $\SL(2,\C)$-Higgs Bundles}
\author{Laura Fredrickson$^1$} 
\address{$^1$Department of Mathematics, University of Oregon, Eugene, OR 97403-1222, USA}
\email{$^1$lfredric@uoregon.edu}
\author{Rafe Mazzeo$^2$}
\address{$^2$Department of Mathematics, Stanford University, Stanford, CA 94305-2125, USA}
\email{$^2$rmazzeo@stanford.edu}
\author{Jan Swoboda$^3$}
\address{$^3$Mathematisches Institut, Ruprecht-Karls-Universit\"at Heidelberg , Im Neuenheimer Feld 205, 69120 Heidelberg, Germany}
\email{$^3$swoboda@mathi.uni-heidelberg.de}
\author{Hartmut Weiss$^4$}
\address{$^4$Mathematisches Seminar, Christian-Albrechts-Universit\"at Kiel, Heinrich-Hecht-Platz 6, 24118 Kiel, Germany}
\email{$^4$weiss@math.uni-kiel.de}
%\date{\today}

\begin{abstract}
Given a generic stable strongly parabolic $\SL(2,\C)$-Higgs bundle $(\cE, \varphi)$, we describe the family of harmonic metrics $h_t$ for
the ray of Higgs bundles $(\cE, t \varphi)$ for $t\gg0$ by perturbing from an explicitly constructed family of approximate solutions
$h_t^{\app}$.  We then  describe the natural hyperK\"ahler metric on $\cM$ by comparing it to a simpler ``semi-flat'' hyperK\"ahler
metric.  We prove that $g_{L^2} - g_{\semif} = O(\e^{-\gamma t})$ along a generic ray, proving a version of Gaiotto-Moore-Neitzke's 
conjecture.  Our results extend to weakly parabolic $\SL(2,\C)$-Higgs bundles as well. 

A centerpiece of this paper is our explicit description of the moduli space and its $L^2$ metric for the case of the four-punctured sphere. 
We prove that the hyperK\"ahler metric in this case is ALG and that its rate of exponential decay to the semiflat metric is the 
conjectured optimal one, $\gamma=4L$, where $L$ is the length of the shortest geodesic on the base curve measured in the singular 
flat metric $|\det \varphi|$.
\end{abstract}

\maketitle

\setcounter{page}{1}

\date{\today}
%\noindent \fixme{Printed \today.}
\medskip

%\noindent {{{\tiny \color{gray} \tt \gitAuthorIsoDate}} \hfill
%{{\tiny \color{gray} \tt \gitAbbrevHash}}}
\bigskip

%\tableofcontents

\section{Introduction}\label{sec:introS1}
Various conjectures and results over the past decade have illuminated the large-scale asymptotic structure of the $\SU(N)$ Hitchin 
moduli space associated to a Riemann surface $C$, and in particular, its natural hyperK\"ahler metric $g_{L^2}$.  A conjectural picture 
due to Gaiotto, Moore and Neitzke \cite{GMNhitchin} has provided stimulus for much of this work, and some part of that has 
now  been established rigorously through the sharp asymptotic results obtained by one of us \cite{Fredricksonasygeo}, 
cf.\ also \cite{DumasNeitzke},  following closely related work by the other three authors together with Witt \cite{MSWW14, MSWWgeometry}.  
The goal of the present paper is to extend these results to the parabolic setting, and to closely analyze the simplest special case:
the moduli space of rank $2$ bundles over the four-punctured sphere, sometimes called the `toy model', which we 
show is a four-dimensional ALG gravitational instanton.

Using the formalism of spectral networks, Gaiotto, Moore and Neitzke \cite{GMNhitchin} introduced a new hyperK\"ahler metric
$g_{GMN}$ which they conjectured is the same as the Hitchin metric $g_{L^2}$.  There is a simpler `semiflat' metric $g_{\mathrm{sf}}$ 
on the part of the Hitchin moduli space above the complement of the discriminant locus in the Hitchin base, and they predicted
that $g_{GMN} - g_{\mathrm{sf}}$ decays exponentially (with respect to the distance from a fixed compact set and in conical 
sectors away from the preimage of the discriminant locus). Their conjecture includes a far more detailed description of the 
asymptotic development of this difference of metrics as a series of terms with increasing exponential rates given by geodesic lengths 
of an underlying family of flat conic metrics on the curve $C$ and with coefficients given in terms of associated BPS states, 
which are Donaldson-Thomas invariants in this setting. (See \textsc{Figure} \ref{fig:curve}. )

While the full scope of this conjectural picture remains out of reach, there has been substantial progress. As a first step, the 
paper \cite{MSWW14} uses gluing methods to construct `large' elements of the $\SU(2)$ Hitchin 
moduli space away from the discriminant locus; this was extended to the $\SU(N)$ case in \cite{FredricksonSLn}.  Unlike
Hitchin's original proof of existence of solutions, this new construction gives a precise description, up to exponentially small 
errors, of the actual fields $(\delbar_E, \varphi, h)$  which solve the Hitchin equations, and hence a parametrization of the ends of 
the Hitchin moduli space. This was then used in \cite{MSWWgeometry} to show that the difference $g_{L^2} - g_{\mathrm{sf}}$ has an 
asymptotic development in a series of {\it polynomially decaying} terms. As a parametrization of the moduli space and its 
tangent bundle, these polynomial terms seem to be actually present, but a miraculous cancellation, first observed by Dumas 
and Neitzke \cite{DumasNeitzke} in a special case and later proved by one of us in full generality \cite{Fredricksonasygeo}, 
shows that this difference does in fact decay exponentially. 

All of this was carried out for connections and Higgs fields without singularities. However, in the applications to 
mathematical physics it is necessary to consider fields admitting simple (or higher order) poles. Our first goal in the 
present work is to extend the results of \cite{MSWW14, MSWWgeometry, Fredricksonasygeo} to the setting of parabolic Higgs bundles, 
i.e., for fields with simple poles.  The analysis proceeds in broad outline much as in the smooth case, but several new technical 
challenges must be faced.   However, having carried this out, we are able to treat, in particular, a special case where the Hitchin moduli 
space is only four-dimensional, and where the discriminant locus lies in a compact set of the Hitchin base. This is the case 
where $C$ is the Riemann sphere $\mathbb C \mathbb P^1$ with four punctures. It is possible here to write out elements
of the Hitchin base explicitly, for example. The Hitchin moduli space in this case has been already examined in detail in 
\cite{hauseldiss, konno93, NakajimaHK, blaavand} since so much can be done explicitly here. (This has also sometimes
been called the `toy model'.)  It was conjectured in these papers, and independently also by Sergey Cherkis, that the 
moduli space in this setting is an ALG gravitational instanton.  We learned of this conjecture in a lecture of Nigel Hitchin at the 
Newton Institute in August 2015.  We prove this here, and also discuss how the family of ALG metrics obtained through this 
construction by varying the parabolic data fit into the recent classification of ALG metrics by Chen-Chen \cite{ChenChenIII}. 
In particular, using an alternate description available in this four-dimensional case, we show that for strongly parabolic data,
the exponential rate  $g_{L^2} - g_{\mathrm{sf}}$ equals the rate predicted in \cite{GMNhitchin}. 

\S 2 reviews general background material about parabolic Higgs bundles.   We introduce the two main building blocks, namely 
limiting configurations and fiducial solutions, in \S 3.  For each of these, we review the construction near simple zeros and
then describe the generalization to strongly and weakly parabolic data successively. This leads immediately to the family of 
approximate solutions. Analysis of the linearized Hitchin equations is carried out first locally in \S 4, then globally in \S 5.
The key new feature here, over what was done in \cite{MSWW14}, is the incorporation of `curvature bubbling' at the parabolic points.
(In addition, in distinction to \cite{MSWW14}, the analysis here is all done at the level of Hermitian metrics.)   The deformation
to exact solutions is in \S 6 and the proof of exponential decay of $g_{L^2} - g_{\mathrm{sf}}$ is in \S 7.  Finally, 
\S 8 contains the detailed and explicit analysis in the case of the sphere punctured at four points. 
The precise statements of results will be given \emph{inter alia}. 

\bigskip

\noindent {\bf Acknowledgments:}  We received substantial advice throughout this project and we wish to thank in particular 
Sergey Cherkis, Lorenzo Foscolo, Andy Neitzke (BPS predictions), Nigel Hitchin, Andr\'e Oliveira (Hitchin section) and 
Claudio Meneses (chamber structure).

Each of the four of us were supported at various times by the GEAR network NSF grant DMS 1107452, 1107263, 1107367  ``RNMS: Geometric Structures and Representation Varieties'';  RM was supported by
NSF DMS-1608223, and JS and HW were supported by DFG SPP 2026. JS was supported   by a Heisenberg grant  of the  DFG  and  within the  DFG RTG 2229 ``Asymptotic invariants and limits of groups and spaces''. This work  is supported by   DFG under  Germany's Excellence Strategy EXC-2181/1 -- 390900948 (the Heidelberg  STRUCTURES Cluster of Excellence).
This paper is based upon work supported by the National Science Foundation under Grant No. DMS-1440140 while we were all in residence at the Mathematical Sciences Research Institute in Berkeley, California, during the Fall 2019 semester.

\begin{figure}[h] 
\begin{centering} 
\includegraphics[height=1in]{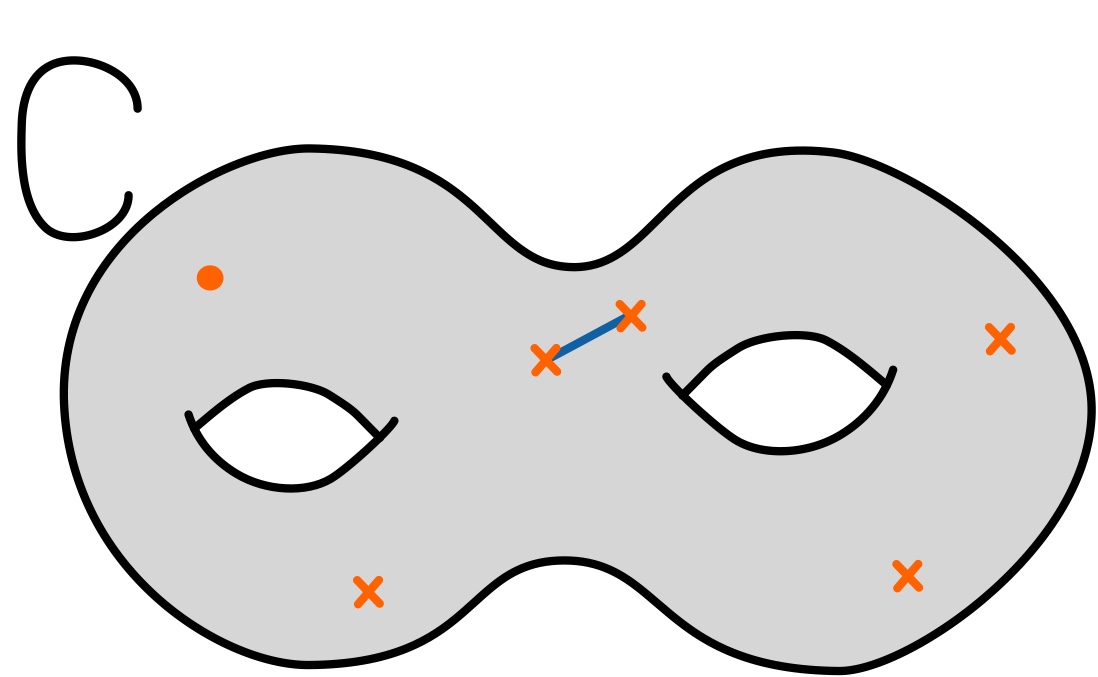}
\caption{\label{fig:curve} The simple zeros and simple poles of $\det \varphi$ are marked respectively by ``$\times$'' and ``$\bullet$''.  The shortest length geodesic determines the constant of exponential decay of $g_{GMN} - g_{\semif}$.
}
\end{centering}
\end{figure}

\section{Background} \label{sec:background}
Fix a compact Riemann surface $C$ of $\mathrm{genus} \geq 2$, with metric $g_C$, complex structure $I$,
and K\"ahler form $\omega_C$. Let $K_C$ be the canonical line bundle. Fix also a complex vector bundle $E \rightarrow C$
of rank $N$ and degree $d$. Its determinant line bundle is denoted $\Det \,E$. Let $\SL(E)$ be the bundle of automorphisms of $E$ inducing the identity map on $\Det\, E$ and $\fsl(E)=\End_0E$ the bundle of tracefree endomorphisms. Then $\cG_\C=\Gamma(\SL(E))$ is the group of complex gauge transformations with Lie algebra $\Gamma(\fsl(E))$.

We also choose on the complex line bundle $\Det\, E$ a holomorphic structure $\delbar_{\Det \,E}$ and denote by $h_{\Det \, E}$
the associated Hermitian-Einstein metric, i.e., 
 \begin{equation*}
 F_{D} = -\sqrt{-1} \deg E 
  \frac{2\pi \omega_C}{\mathrm{vol}_{g_C}(C)} \Id_{\Det \, E},
\end{equation*}
where $D=D(\delbar_{\Det \,E}, h_{\Det \, E})$ is the Chern connection on this holomorphic Hermitian line bundle.  

We now recall two different representations of the moduli space $\cM$ of $\SL(N,\C)$-Higgs bundles associated to this fixed data.
The central objects are stable Higgs pairs. 

 A Higgs pair (or more simply, a Higgs bundle) $(\delbar_E, \varphi)$
consists of a holomorphic structure on $E$ and a section $\varphi \in \Omega^{1,0}(C, \End_0 E)$ satisfying $\delbar_E \varphi =0$, where $\delbar_E$
induces $\delbar_{\Det \, E}$; this pair 
is called stable if any proper holomorphic subbundle $E' \subset E$ which satisfies $\varphi(E') \subset E' \otimes K_C$
satisfies $\mu(E') < \mu(E)$, where the slope $\mu(F)$ of any holomorphic bundle $F$ is defined to
equal the quotient $\deg F/\mathrm{rk}\, F$. A pair is called polystable if it is a direct sum of stable pairs of
lower rank and equal slope. The complex gauge group acts on such pairs by 
\begin{equation} \label{eq:complexgaugeequiv}
g \cdot (\delbar_E, \varphi) = ( g^{-1} \circ \delbar_E \circ g, g^{-1} \varphi g)\ \mbox{for any}\ g\in \cG_\C. 
  %\qquad \mbox{where $(g \cdot h)(v,w)=h(g v,g w)$}.
\end{equation}
The Dolbeault representation of the Higgs bundle moduli space consists of the space of (polystable) pairs modulo
the complex gauge action.

We shall be concerned, however, with the alternate representation of this moduli space as consisting of
triples $(\delbar_E, \varphi, h)$ which solve Hitchin's equations up to unitary gauge equivalence. % cf.\ \eqref{eq:complexgaugeequiv}. Here
Here $(\delbar_E, \varphi)$ is as before, and $h$ is a Hermitian metric on $E$ which induces $h_{\Det \,E}$.
We write $D(\delbar_E, h)$ for the Chern connection associated to $\delbar_E$ and $h$ and $\varphi^{*_h} \in \Omega^{0,1}(C, \End_0 E)$ for
the $h$-Hermitian adjoint. In addition, denote by $F_D^\perp$ the trace-free part of the curvature of $D$, 
 \begin{equation*}
  F_{D(\delbar_E, h)}^\perp= F_{D(\delbar_E,h)} +\sqrt{-1} \mu(E) \frac{2\pi \omega_C}{\mathrm{vol}_{g_C}(C)} \mathrm{Id}_E.
\end{equation*}
The triple $(\delbar_E, \varphi, h)$ satisfies Hitchin's equations for $G=\SU(N)$ if 
 \begin{equation*} \label{eq:Hitchin}
F^\perp_{D(\delbar_E,h)} + [\varphi, \varphi^{*_h}]=0, \ \ \mbox{and}\ \ \  \delbar_E \varphi =0.
\end{equation*}
The first of these two equations may be regarded in two ways: either as an equation for $\delbar_E$ (and hence the full connection $D$) 
and $\varphi$ if the metric $h$ is fixed, or else as an equation for the metric $h$ once the Higgs pair $(\delbar_E, \varphi)$ is fixed.
In the latter case, $h$ is called the \emph{harmonic metric} associated to $(\delbar_E, \varphi)$; we shall write out the equation
for $h$ more explicitly below.  One of the first main results in the theory, due in this formulation to Simpson, is that a Higgs bundle
$(\delbar_E, \varphi)$ admits a harmonic metric if and only if $(\delbar_E,\varphi)$ is polystable.

It will also be convenient to refer to the first formulation.  For this we fix a Hermitian metric $h_0$ on $E$ and consider pairs
$(A, \Phi)$ where $A$ is an $h_0$-unitary connection and $\Phi \in \Omega^{1,0}(C, \End_0 E)$.  The Hitchin
equations are now that $\delbar_A \Phi=0$ and $F^\perp_A + [\Phi, \Phi^{*_{h_0}}]=0$. The Hitchin moduli space $\cM$ for $G=\SU(N)$ 
consists of the space of pairs which solve these equations modulo $h_0$-unitary gauge equivalence. We will write $\cG=\Gamma(\SU(E))$ for the unitary gauge group with Lie algebra $\Gamma(\mathfrak{su}(E))$.

%We shall work both in holomorphic and in unitary gauges. A local basis of sections $\{s_1, \cdots, s_n\}$ for $E$ is
%holomorphic if all the $s_i$ are holomorphic with respect to $\delbar_E$, while a local basis of sections $\{e_1, \cdots, e_n\}$
%is unitary if $h_0(e_i, e_j) = \delta_{ij}$.

It is not difficult to pass back and forth between these two formulations. Indeed,  given a triple $(\delbar_E, \varphi, h)$ where $h$
is the harmonic metric for this Higgs bundle, 
there is an $\SL(E)$-valued $h_0$-Hermitian section $H$ such that $h(v,w)=h_0(Hv, w)$.
 Take the complex gauge transformation $g=H^{-1/2}$.
 Observe that in general, $(g \cdot h)(v,w)=h_0((g^{*_{h_0}} H g) v, w)$;
 consequently, for our choice of gauge transformation $g=H^{-1/2}$, indeed $(g \cdot h)=h_0$.
 Then, for the complex gauge action in \eqref{eq:complexgaugeequiv},
 $g \cdot (\delbar_E, \varphi, h)=(H^{1/2} \circ \delbar_E \circ H^{-1/2}, H^{1/2} \varphi H^{-1/2}, h_0)$. Consequently, 
the associated pair $(A, \Phi)$ is defined by $\delbar_A = H^{1/2} \circ \delbar_E \circ H^{-1/2}$ and $\Phi = H^{1/2} \varphi H^{-1/2}$.

%\subsection{Review: Parabolic Higgs bundles, spectral data, and non-abelian Hodge correspondence}\label{sec:Higgs}
\subsection{Parabolic Higgs bundles} \label{sec:parabolicHiggs}
We next recall the salient facts about parabolic Higgs bundles, and in particular the moduli spaces of (weakly or strongly) parabolic
Higgs bundles, following \cite{BodenYokogawa}. For simplicity, we restrict this discussion to the setting of rank $2$ bundles. 

A parabolic Higgs bundle consists of a Higgs bundle where the fields have simple poles at a given divisor $D = p_1 + \ldots + p_n$, 
and where extra algebraic data is specified at each $p_j$.  This extra data consists of a weighted flag, which amounts to
specifying a boundary condition for the harmonic metric at that puncture, and in the weakly parabolic case a prescription of the eigenvalues of the residues. 

We now explain this in more detail. Fix a divisor $D$ and at each $p \in D$, fix also a  weight vector $\vec{\alpha}(p)=
(\alpha_1(p), \alpha_2(p)) \in [0,1)^2$  and in the weakly parabolic case $\sigma(p) \in \C$.
In the weakly parabolic case, we split $D= D_s \sqcup D_w$ where the strongly parabolic divisor is $D_s=\{p \in D: \sigma(p) =0\}$ and the weakly parabolic divisor is $D_w=\{p \in D: \sigma(p) \neq 0\}$. 
Thus a strongly parabolic Higgs bundle is a weakly parabolic Higgs bundle with $D_w = \emptyset$.% multiplicities $m_{i}(p)$ for each weight $\alpha_{i}(p)$. 

A parabolic $\SL(2,\C)$-Higgs bundle over $(C, D)$ is then a rank $2$ complex vector bundle over $C$ and a triple 
$(\delbar_E, \{\mathcal{F}(p)\}_{p \in D}, \varphi)$, where $\delbar_E$ is a holomorphic structure on $E$ inducing a fixed holomorphic structure on $\Det\,E$, $\varphi$ is a holomorphic 
map $\cE \rightarrow \cE \otimes K(D)$ (i.e., $\varphi$ has simple poles at each $p$) which is traceless and for each $p \in D$, $\mathcal{F}(p)=F_\bullet(p)$ 
is a complete  flag in the fiber $E_p$: 
\begin{equation*}
\begin{array}{ccccccccccccc}
E_p  &=&F_1(p) & \supset& F_2(p) \supset 0 \\ %& \supset& \cdots  &\supset& F_{s_p}(p)  &\supset& 0 \\
0  &\leq&  \alpha_1(p) & <&  \alpha_2(p) < 1. %&< & \cdots &<&  \alpha_{s_p}(p)< 1,
\end{array}
\end{equation*}
%with $m_i(p)= \dim F_{p,i} - \dim F_{p, i+1}$.  
This triple is called strongly parabolic at $p$ if $\varphi(p): F_i(p) \to F_{i+1}(p)\otimes K(D)_p$ 
and weakly parabolic if $\varphi(p): F_i(p) \to F_i(p) \otimes K(D)_p$ for each $i$. Thus in the strongly parabolic case,
the residue of the Higgs field is nilpotent with respect to the flag, while in the weakly parabolic case, this residue
preserves the flag. Furtheromore, in the $\SL(2,\C)$-case we require $\alpha_1(p) + \alpha_2(p)=1$ for all $p \in D$ and, additionally, for weakly parabolic points  the residue to have eigenvalues $\sigma(p)$ and $-\sigma(p)$. In either setting we 
will say that it is compatible with the flag. For simplicity of notation, we write $\cE$ for the holomorphic bundle $(E, \delbar_E)$ together with the flag $\{ \mathcal{F}(p)\}_{p \in D}$;
thus a parabolic Higgs bundle is again simply a pair $(\cE, \varphi)$.

%Strictly speaking we have defined above a parabolic $\GL(2,\C)$-Higgs bundle; to reduce to the $\SL(2,\C)$ case, we consider holomorphic structures $\bar\partial_E$ which induce a fixed holomorphic structure on $\Det(E)$ and traceless Higgs fields as in the ordinary case. Furthermore we require  $\alpha_1(p) + \alpha_2(p) \in \Z$ for all $p \in D$.

Stability in this setting depends on the weight vector $\vec{\alpha}(p)$.  Define the parabolic degree of the parabolic bundle 
$\mathcal{E}=(E, \delbar_E,\{\mathcal{F}(p)\}_{p \in D})$ by
\begin{equation*}
 \operatorname{pdeg}_{\vec\alpha}\cE = \deg \cE + \sum_{p \in D} (\alpha_1(p) + \alpha_2(p)),
\end{equation*}
which reduces to $\deg \cE + |D|$ in the $\SL(2,\C)$-case.
Note also that the parabolic structure on $\cE$ induces parabolic structures on its holomorphic subbundles. 
We say that $\cE$ is $\vec\alpha$-stable if 
\begin{equation*}
 \frac{\operatorname{pdeg}_{\vec\alpha} \cE}{\operatorname{rank} \cE} > 
\mathrm{pdeg}_{\vec\alpha} \mathcal{\cL} %}{\mathrm{rank} \mathcal{\cE'}}.
\end{equation*}
for every holomorphic line subbundle $\cL$ preserved by $\varphi$. 
The parabolic weights of $\cL$ at $p$ are defined to be $\alpha_2(p)$ if the fiber  $L_p = F_2(p)$ and $\alpha_1(p)$ otherwise.

We are now in a position to define the moduli space $\cM_{\mathrm{Higgs}}$ as the set of isomorphism classes of $\vec{\alpha}$-stable parabolic $\SL(2,\C)$-Higgs bundles (whose residues have eigenvalues $\pm\sigma(p)$ at $p \in D$ in the weakly parabolic case), where isomorphism in this category means holomorphic bundle isomorphism commuting with the Higgs fields and preserving the flag structure. Notice that the weight data but not the flags are fixed. This only defines $\cM_{\mathrm{Higgs}}$ as a set; a quotient in the algebraic category has been constructed by \cite{yokogawa93}.

There is a differential-geometric definition of $\cM_{\mathrm{Higgs}}$ which clarifies the role of
the weights $\vec{\alpha}$. For this we fix the smooth complex bundle $E$, divisor $D$, and flag 
 $\mathcal{F}(p)$ on each fiber $E_p$, $p \in D$.  Thus here we fix the flag, but not yet the weights. 
%  \begin{align*}
%  E_p  =F_1(p) & \supset F_2(p) \supset \cdots  \supset F_{s_p}(p)  \supset 0.% \\
%  %\alpha_1(p) & <  \alpha_2(p) < \cdots <  \alpha_{s_p}(p)< 1
% \end{align*}
% (Note that now we \emph{do} fix $\mathcal{F}(p)$, but we do not fix the weights 
% $\boldsymbol{\alpha}(p)$ yet.)
Let $\ParEnd(E)$ be the bundle of endomorphisms of $E$ which preserve the flag $\cF(p)$ for each $p \in D$ and $\cG_\C = \Gamma(\SL(E) \cap \ParEnd(E))$ the complex gauge group.

Letting $\mathcal{A}_0$ denote the affine space of all holomorphic structures on $E$ inducing the fixed holomorphic structure $\bar\partial_{\Det \, E}$ on the determinant line bundle, define the space
\begin{multline*}
\mathcal{H}= \Bigl\{ (\delbar_E, \varphi) \in  \mathcal{A}_0 \times \Omega^{1,0}(C,\End_0 E) \;\big|\; 
  %\delbar_E \varphi = \sum_{p \in D} \mathrm{Res}_p(\varphi) \delta_p\ \mbox{and each}\ \mathrm{Res}_p(\varphi)\\
\varphi \, \text {meromorphic with respect to} \, \bar\partial_E \, \text{with} \\ \text{simple pole at each}\, p \in D \, \text{such that} \, \mathrm{Res}_p(\varphi) \, 
\text{is compatible with the flag} \Bigr\}. 
\end{multline*} 
At this point, fix the weight vector $\vec{\alpha}(p)$ at each $p$, and assume that the weights are generic in the
sense that any $\vec\alpha$-semistable bundle is in fact $\vec\alpha$-stable.  Now let $\mathcal{H}_{\vec\alpha} \subset \mathcal{H}$
be the subspace of pairs $(\delbar_E, \varphi)$ which are $\vec\alpha$-stable and define $\mathcal{M}_{\mathrm{Higgs}} =
\mathcal{H}_{\vec\alpha}/\mathcal{G}_\C$, where $\mathcal{G}_\C$ is the complex gauge group as above. We recall that 
\begin{equation*} \label{eq:dim}
 \dim_\C \mathcal{M}_{\mathrm{Higgs}}= 6(g-1) +   2n, %\sum_{p \in D} \left( 4 - \sum_{i=1}^{s_p} m_i(p)^2 \right),
\end{equation*}
where $g=\mathrm{genus}(C)$ and $n$ is the number of parabolic points \cite{BodenYokogawa}.

\subsection{Non-abelian Hodge correspondence}\label{sec:NAHC}
The usual non-abelian Hodge correspondence extends to this parabolic case, and to a stable parabolic Higgs bundles we can associate a harmonic metric adapted to the parabolic structure. 
% Given a stable parabolic Higgs bundles $(\cE, \varphi)$ we say that $h$ is \emph{harmonic} if the triple $(\cE, \varphi, h)$ satisfies Hitchin's equations
% \begin{equation}
%  F_{D(\delbar_E, h)} + [\varphi, \varphi^{*_h}]=0.
% \end{equation}
% It remains to define what it means for a metric to be ``adapted to the parabolic structure.''
To describe what it means for these metrics to be adapted, let us restrict attention for simplicity to the case of rank $2$ where the flags are full.

% First note that $\mathcal F(p)$ yields a filtration of the space of sections of $\cE(*D)$. In general one considers filtrations on the locally free
% $\cO_C(* D)$-module $\cP_\bullet \cE$. Both the parabolic structure and Hermitian structure associated to a metric induce filtrations on the
% space of sections, and we say that a metric $h$ is adapted to the parabolic structure if the filtrations are the same \cite{simpsonnoncompact}.

As described carefully in \cite[\S3.5]{Mochizukitame}, both a parabolic structure and a Hermitian structure determine a filtration of the
sections of $\cE \otimes \cO_C(*D)$ near any $p$, where $\cO_C(*D)$ is the sheaf of algebras of rational functions with poles at $D$.
The filtration associated to the Hermitian metric is by the order of growth of holomorphic sections: $|s(x)|_h\sim \mathrm{dist}(x,p)^{\alpha}$.
The Hermitian metric $h$ is said to be adapted to the parabolic structure if these two filtrations coincide. 

Let us write this down in a model situation, cf.\ \cite[\S2]{BiquardGarciaPrada}. Consider a parabolic vector bundle $(\cE, \{\cF(p)\}_{p \in D},
\{\vec{\alpha}\}_{p \in D})$ of rank $2$ with a complete flag and parabolic weights $\vec{\alpha}(p)=
(\alpha_1(p), \alpha_2(p))$ at every $p \in D$. Choose a local holomorphic coordinate $z$ and a holomorphic splitting of $\cE$ compatible with the filtration near $p$, i.e.\ a
holomorphic basis of sections $e_1, e_2$ with $e_i(p) \in F_i(p)$. Then the metric
\begin{equation*}
 h_{\vec{\alpha}}= \begin{pmatrix} |z|^{2\alpha_1(p)} & 0 \\ 0 & |z|^{2 \alpha_2(p)}  \end{pmatrix}
% & & \\ & &  \ddots & \\ & & & |z|^{2 \alpha_n(p)} \end{pmatrix}.
\end{equation*}
is adapted to the parabolic structure. A unitary frame is provided by $\widetilde e_1 = |z|^{-\alpha_1(p)}e_1$, $\widetilde e_2 = |z|^{-\alpha_2(p)}e_2$ with respect to which the Chern connection of $h_{\vec{\alpha}}$ is given by
\begin{equation*}
d + \begin{pmatrix} i \alpha_1 & 0 \\ 0 & i\alpha_2 \end{pmatrix} \de \theta
\end{equation*}
for $z=re^{i\theta}$. In the $\SL(2,\C)$-case, where $\alpha_1(p) + \alpha_2(p)=1$, we see that the local monodromy around $p$ lies in $\SU(2)$.

For a stable parabolic $\SL(2,\C)$-Higgs bundle $(\cE,\varphi)$ we can uniquely solve Hitchin's equation
\begin{equation*}
F^\perp_{D(\delbar_E,h)} + [\varphi, \varphi^{*_h}]=0 
\end{equation*}
in the class of Hermitian metrics $h$ adapted to the parabolic structure \cite{simpsonnoncompact}. 
Here, 
 \begin{equation}\label{eq:perp}
  F_{D(\delbar_E, h)}^\perp= F_{D(\delbar_E,h)} +\sqrt{-1} \mu(E) \frac{2\pi \omega_C}{\mathrm{vol}_{g_C}(C)} \mathrm{Id}_E,
\end{equation}
where $\mu(E)$ is the \emph{parabolic} slope of $E$.
This gives rise to (part of) the nonabelian Hodge correspondence, namely the diffeomorphism
\[
\mathrm{NAHC}: \cM_{\mathrm{Higgs}} \rightarrow \cM
\]
obtained by mapping $(\cE,\varphi)$ to the triple $(\cE,\varphi,h)$.

Let us finally discuss the other part of the non-abelian Hodge correspondence. If we assume for simplicity that $\operatorname{pdeg}_{\vec\alpha} E = 0$, then for a solution $h$ of Hitchin's equation the $\SL(2,\C)$-connection $D(\delbar_E,h)+ \varphi + \varphi^{*_h}$ is flat and hence gives rise to representation $\rho: \pi_1(C \setminus D) \to \SL(2,\C)$. The local monodromies around points in $D$ are determined up to conjugacy by the parabolic weights and the eigenvalues of the residues of the Higgs field. They lie in $\SL(2,\C)$ rather than $\GL(2,\C)$ precisely because $\alpha_1(p) + \alpha_2(p)=1$.% is an integer.

\begin{rem}\label{rem:DetE}
Note that the complexities involved in considering $\deg E \neq 0$ are similar to the complexities 
from including parabolic weights.  

The simplest case is the case of ordinary Higgs bundles where the fixed data $\Det \, \cE=(\Det\,E, \delbar_{\Det\, E})$ is the trivial holomorphic 
bundle with trivialization $s$. The Hermitian-Einstein metric  $h_{\Det\, E}(s, s)$ is constant, and we can renormalize $s$ so that 
$h_{\Det\, E}(s,s)=1$. Then locally, we can represent $h$ by a matrix of determinant $1$ in a basis of sections $e_1, \cdots, e_n$ 
such that $e_1 \wedge \cdots \wedge e_n =s$.

In the case of ordinary Higgs bundles where the fixed data $\Det \, \cE$ is not the trivial holomorphic bundle, there exists a local holomorphic 
section $s$ of $\Det \, \cE$; $h_{\Det \, E}(s,s)$ need not be constant. We can then represent $h$ by a matrix which has determinant 
$h_{\Det \,E}(s,s)$ in a basis of sections $e_1, \cdots, e_n$ for which $e_1 \wedge \cdots \wedge e_n =s$, i.e.
\begin{equation*}
 h = (h_{\Det \,E}(s,s))^{1/n} h^\diamond,
\end{equation*}
where $h^\diamond$ is represented by a matrix of determinant $1$.  We emphasize that this prefactor $(h_{\Det \,E}(s,s))^{1/n}$ is 
completely determined by $\Det \, \cE$ and the choice of section $s$ and is not a source of any additional freedom in the problem.
Note that in the basis $e_1, \cdots, e_n$, $\varphi$ and $[\varphi, \varphi^{*_h}]$ are still represented by traceless matrices,
but $F_{D(\delbar_E, h)}$ is not represented by a traceless matrix. 

Finally, in the setting of parabolic Higgs bundles, the additional complexity is that the metric $h$ is adapted to the parabolic structure.
This means that $h_{\Det \,E}(s,s))$ (hence $\Det \, h$) has a singularity of the form $|z|^{2\sum_i \alpha_i(p)}$ in a local holomorphic 
coordinate $z$ centered at $p \in D$. 
\color{black}
\end{rem}

\subsection{\texorpdfstring{Hyperk\"ahler}{Hyperkahler} metric}

In this section we review the construction of a hyperK\"ahler metric on the moduli space of stable parabolic $\SL(2,\C)$-Higgs bundles. In the strongly parabolic case this was carried out by \cite {konno93} and later generalized to the weakly parabolic case by \cite{NakajimaHK}.

Fix $h_0$ adapted to the parabolic structure (which for now we may assume to coincide with the model metric $h_{\vec\alpha}$ near $p \in D$) 
and a parabolic Higgs bundle $(\bar\partial_E,\varphi)$. Let $A_0=D(\bar\partial_E, h_0)$ be the Chern connection and $\Phi_0=\varphi$. 
The pair $(A_0,\Phi_0)$ will serve as a basepoint in the following construction. We consider pairs $(A,\Phi)$ on $C\setminus D$ such that
$\bar\partial_A -\bar\partial_{A_0}$ and $\Phi - \Phi_0$ lie in certain function spaces encoding decay properties near the 
punctures.  This decay is measured using weighted Sobolev spaces in \cite {konno93} and \cite{NakajimaHK}, and in weighted 
$b$-H\"older spaces here.

In either approach, we impose that in the unitary frame $\widetilde e_1 = |z|^{-\alpha_1(p)}e_1$, $\widetilde e_2 = |z|^{-\alpha_2(p)}e_2$, 
\begin{equation*}
\bar\partial_A -\bar\partial_{A_0} = \begin{pmatrix} \alpha & \beta \\ \gamma & - \alpha\end{pmatrix} d\bar z
\qquad
\text{and} 
\qquad 
\Phi-\Phi_0 = \begin{pmatrix} a & b \\ c & -a \end{pmatrix} \de z,
\end{equation*}
where $\alpha, a = O(1)$ and $\beta,\gamma,b,c = O(r^\varepsilon)$ for some $\varepsilon>0$ (the decay being measured 
in an $L^2$ sense in the first approach and an $L^\infty$ sense in the second approach). This infinite dimensional 
affine space of fields is acted on by the unitary gauge group $\cG = \{ g \in \cG_\C \,\vert \, g \, \text{$h_0$-unitary} \}$. The 
moment maps for this action are
\begin{align*}
\mu_I(A,\Phi) &= F_A^\perp + [\Phi,\Phi^*]	\\
(\mu_J+\mu_K)(A,\Phi) &= \bar\partial_A\Phi, 
\end{align*}
corresponding to complex structures
\[
I(\dot A^{0,1},\dot\Phi) = (i \dot A^{0,1}, i \dot\Phi) , \quad J(\dot A^{0,1},\dot\Phi) = (i\dot\Phi^{*},-i (\dot A^{0,1})^*), \quad K(\dot A^{0,1},\dot\Phi)= (-\dot\Phi^*, (\dot A^{0,1})^*),
\]
where $* = {*_{h_0}}$. 
The natural $L^2$-metric on the configuration space is hyperK\"ahler and descends to a hyperK\"ahler metric on the moduli space of solutions to Hitchin's equation
\[
\mathcal{M} = \mu_I^{-1}(0) \cap \mu_J^{-1}(0) \cap \mu_K^{-1}(0)/\cG 
\]
via the process of hyperK\"ahler reduction. By the non-abelian Hodge correspondence $\mathrm{NAHC}: \cM_{\mathrm{Higgs}} \rightarrow \cM$ the hyperK\"ahler metric can be thought to live on $\cM_{\mathrm{Higgs}}$. 

\subsection{Spectral data}\label{sec:parspectraldata} 
We next review briefly the spectral data associated to a parabolic Higgs bundles, cf.\ \cite[\S2.3]{LogaresMartens} for
a more thorough treatment. 

The Hitchin fibration is given by the usual map 
\begin{equation*}
 \mathrm{Hit}: (\delbar_E, \cF(p), \varphi) \mapsto \mathrm{char}_\varphi(\lambda)
\end{equation*}
where $\mathrm{char}_\varphi(\lambda)$ is the characteristic polynomial of $\varphi$; this does not depend on the parabolic structure.
The holomorphicity of $ \varphi: \cE \rightarrow \cE \otimes K_C(D)$ identifies the Hitchin base $\mathcal{B}$ with the vector space
of coefficients of $\charp_\varphi(\lambda)$; since $\mathrm{tr\,} \varphi = 0$,  in this rank $2$ case the only remaining coefficient
is the determinant $q = \det \varphi$. For $p \in D_s$,  $\varphi$ has a nilpotent residue at $p$, hence this is a quadratic differential with at most 
simple poles at $D$.
For $p \in D_w$, $q$ is a quadratic differential with at most a double pole like $-\sigma^2 \frac{\de z^2}{z^2}$ near $p$. 

Each point in $\mathcal{B}$ can be regarded as a spectral curve $\Sigma$ in the total space of the bundle $K_C(D) \rightarrow C$. 
The map $\pi: \Sigma \rightarrow C$ is a $2:1$ ramified cover, and each sheet represents a different eigenvalue of the Higgs field. 
The genus of $\Sigma$ is
\begin{equation*} \label{eq:genus}
g(\Sigma) = 4 (g-1) +  n + 1, %\frac{rn(r-1)}{2} + 1, 
\end{equation*}
see \cite[Eq. 6]{LogaresMartens}.

Let $\cB'$ be the subset of points in $\cB$ for which the spectral covers are smooth and $\cM'=\mathrm{Hit}^{-1}(\cB')$ the \emph{regular locus}.
To any $(\cE, \varphi) \in \cM'$ we can associate the spectral data $\cL \rightarrow \Sigma$, where $\cL$
is a line bundle on the spectral cover $\Sigma \subset \mathrm{Tot}(K_C(D))$. Away from ramification points, the fiber of
$\cL \rightarrow \Sigma$ is the corresponding eigenspace of $\phi$ in $\pi^*\cE \rightarrow \Sigma$. The degree of $\cL$
is \cite[p.\ 10]{LogaresMartens} 
\begin{equation} \label{eq:deg}
 \deg \cL = \deg \cE  + 2(g - 1) + n.
\end{equation} 

We can reverse this process and recover the parabolic Higgs bundle $(\cE, \varphi)$ from the spectral data $(\Sigma, \cL)$. Indeed,
as a holomorphic bundle, $\cE=\pi_*\cL$, while the Higgs field is recovered by pushing down multiplication by the tautological
section $\lambda$ of $\pi^*K_C(D) \rightarrow \Sigma$.
We also need to construct the flags at each of the marked points.  In general, there is some finite ambiguity at this step, see 
Logares-Martens \cite[Proposition 2.2]{LogaresMartens}, but in this rank $2$ case, the flags are full and 
there is no ambiguity.

\section{Asymptotic profiles and approximate solutions}\label{sec:hell}
Fix a stable parabolic $\SL(2,\C)$-Higgs bundle $(\cE, \varphi)$ in $\cM'$, so that $\det \varphi = q$ %\in H^0(C, K_C^2)$  
has only simple zeros. Our eventual goal is to construct the harmonic metrics $h_t$ associated to the one-parameter 
family of Higgs bundles $(\cE, t \varphi)$ when $t \gg 1$. Equivalently, we fix $(\cE, \varphi)$ and view $h_t$ as the 
unique solution of the $t$-rescaled Hitchin equations 
\begin{equation}\label{eq:trescaled}
 F^\perp_{D(\delbar_E, h_t)} + t^2[\varphi, \varphi^{*_{h_t}}]=0.
\end{equation}
In this section we introduce the two main ingredients 
used to construct this family: the limiting configurations and the model solutions on $\CC$, also
known as the fiducial solutions.  The metrics $h_t$ are constructed by desingularizing limiting  configurations
using fiducial solutions.  This may also be understood in reverse, by considering the possible limiting
behavior of $h_t$ as $t \to \infty$.  The ansatz made in \cite{MSWW14} is that the fields $A_t$ and $t\Phi_t$ asymptotically
decouple, and in the limit, on compact sets away from the zeros and poles of $q$, converge to solutions of the decoupled
equations 
\begin{equation}\label{eq:Hitchindecoupled}
F^\perp_{D_{A_\infty}}=0, \quad \mbox{and}\quad \left[ \Phi_{h_\infty},\Phi_{h_\infty}^{*_{h_0}} \right]=0. 
\end{equation}
This was later vindicated by Mochizuki \cite{Mochizukiasymptotic}: 
\begin{thm} \cite[Theorem 2.7]{Mochizukiasymptotic}  For any compact set $K \subset C\setminus (Z \cup D)$, there exist 
positive constants $c_0$ and $\varepsilon$ such that the family of solutions of Hitchin's  equations $(\cE, t \varphi, h_t)$ satisfy 
\begin{equation*} 
\left|[\varphi, \varphi^{*_{h_t}} ] \right|_{h_t, g_C} \leq c_0 \exp(-\varepsilon t) 
\end{equation*}
in $K$. 
\end{thm}
\noindent His analysis is local, and thus applies also to the parabolic setting.  

The fiducial solutions, by contrast, are the limits of rescalings of the fields near the zeros and poles; in other words, they are
the `bubbles' in this theory. 

\subsection{Limiting configurations}\label{sec:limitingconfigurations}
We now describe the limiting configurations more carefully, emphasizing the Hermitian metrics rather than the fields
$(A,\Phi)$. In particular, we describe normal forms for these limiting configurations near the zeros and poles
of $\varphi$. 

Let $(\cE, \varphi) \in \cM'$ be a parabolic $\SL(2,\C)$-Higgs bundle. We now construct a singular Hermitian metric $h_\infty$ on $\cE$ 
as the pushforward of a singular metric on the associated spectral data $\cL \rightarrow \Sigma$, and which solves the decoupled 
Hitchin equations. This metric can be chosen to induce a fixed Hermitian metric on $\Det \, E$.

First denote by $Z \subset C$ the set of zeros of $\det \varphi \in H^0(C, K_C^{2})$. Let $R=Z \sqcup D_s \subset C$ be the set of zeros of $\det\varphi \in H^0(C, K_C(D)^{\otimes 2})$. (Recall that the zeros and simple poles of 
$\det\varphi$, regarded as a section of $ H^0(C, K_C^2)$, are zeros of $\det\varphi$ regarded as a section of $K_C(D)^{\otimes 2}$.)
Let $\cL \rightarrow \Sigma$ be the spectral data associated to $\det \varphi$. The projection $\pi: \Sigma \rightarrow C$ is ramified at $R$ but not at $D_w$.
At points $\widetilde{p} \in \widetilde{D}_s = \pi^{-1}(D_s)$ coming from simple poles of $q$, define $\widetilde{\alpha}_{\widetilde{p}}= \frac{1}{2}$; at points $\widetilde{p} \in \widetilde{Z}$, define $\widetilde{\alpha}_{\widetilde{p}}= - \frac{1}{2}$. Over each point in $D_w$ there are exactly two points
 $\widetilde{p} \in \widetilde{D}_w$. At one of these the naturally associated weight is $\alpha_1(\pi(\widetilde{p}))$ and at the other the weight is $\alpha_2(\pi(\widetilde{p}))$. Define $\widetilde{\alpha}_{\widetilde{p}}$ to be this naturally associated value. 

Now equip $\cL$ with a parabolic structure by setting the parabolic weights at $\widetilde{p} \in \widetilde{R} \sqcup \widetilde{D}_w$ to be equal $\alpha_{\widetilde{p}}.$
\begin{rem}\label{rem:weights} \label{rem:problem}
%This choice of weights makes $\mathrm{pdeg}(\cL)= \mathrm{pdeg}(\cE)$. 
The number of ramification points is 
$|R|=\deg K_C(D)^{\otimes 2}=4(g-1)+2|D|$. From \eqref{eq:deg} we see that 
\begin{equation*}
 \deg \cL = \deg \cE + 2(g-1) + |D| = \deg \cE + \frac{1}{2} |R|.
\end{equation*}
The sum of the parabolic weights on $\cL$ is
\begin{equation*}
 \sum_{\widetilde{p} \in \widetilde{R} \cup \widetilde{D}_w} \alpha_{\widetilde{p}} =   
-\frac{1}{2} |R| + |D| %\sum_{p \in D} \alpha_1(p) + \alpha_2(p).
\end{equation*}
Consequently, $\mathrm{pdeg}_{\vec{ \widetilde{\alpha}}} \cL=\mathrm{pdeg}_{\vec\alpha}\cE$.
\end{rem}

By \cite{biquard, simpsonnoncompact}, there exists a Hermitian-Einstein metric $h_{\cL}$ on the parabolic line bundle $\cL$ 
adapted to the parabolic structure. This metric is unique up to a constant factor and solves
% \begin{equation}\label{eq:parabolichermitianeinstein}
$F_{(\delbar_\cL, h_\cL)}^{\perp} = 0$.

Finally, define $h_\infty$ on  $\cE|_{C\setminus R}$ as the (orthogonal) pushforward of $h_{\cL}$. In other words, we let  $h_\infty$ 
be the metric on $\cE$ such that the eigenspaces of $\varphi$ are orthogonal, and which equals the metric induced by
$h_{\cL}$ on each eigenspace. The metric $h_\infty$ solves $F_{(\delbar_E, h_\infty)}^\perp =0$
, where the constant appearing in the trace-free part is defined in \eqref{eq:perp}, precisely because $\mathrm{pdeg}_{\vec{\widetilde{\alpha}}} \cL=\mathrm{pdeg}_{\vec\alpha}\cE$.  Note that $h_\infty$ is adapted to the parabolic structure at points in $D_w$ but not at points in $D_s$. 

\begin{cor}
There exists a Hermitian metric $h_\infty$ which solves the decoupled Hitchin's equations \eqref{eq:Hitchindecoupled} and induces the fixed
Hermitian structure on $\Det\,\cE$.
\end{cor}
\begin{proof}
The metric $h_\infty$ constructed above induces a Hermitian-Einstein metric on $\Det\,\cE$.  Since Hermitian-Einstein metrics are unique
up to scale, we can rescale $h_{\cL}$ so that this metric on $\Det\,E$ agrees with the given one.
\end{proof}

\subsection{Normal forms} \label{sec:localmodellimiting}
We now exhibit the normal form for $(\delbar_E, \varphi, h_\infty)$ around simple zeros, and simple and double poles of $\det\varphi$.
Here and in several places below, we treat these three different settings in sequence; in most of these cases, for simple zeros we
simply record the appropriate results from elsewhere, making the necessary modifications to the $\Det \,\cE \not\simeq \cO$ setting, as described in Remark \ref{rem:DetE}. 

\begin{prop}\cite[Lemma 4.2]{MSWW14} \label{prop:simplezeronormalform} Fix $(\cE, \varphi) \in \cM'$ and suppose that $\det \varphi$ 
has a simple zero at $p$. Then there is a holomorphic coordinate $z$ centered at $p$ such that $-\det\varphi=z\de z^2$. In addition, 
in some holomorphic gauge,  
\begin{equation*}\label{eq:goodgaugehol0}
\delbar_E= \delbar,  \quad
\varphi =
\begin{pmatrix} 0 & 1 \\z & 0 \end{pmatrix} \de z,  \quad 
h_\infty = Q h_\infty^{\diamond}  \qquad \mbox{where } h_\infty^{\diamond}=\begin{pmatrix} |z|^{1/2} & 0 \\ 0  & |z|^{-1/2} \end{pmatrix}.
\end{equation*}
Here, $Q$ is a locally-defined function completely determined by $h_{\Det \,\cE}$ and the choice of holomorphic section of $\Det \, \cE$.
\end{prop}

Turning now to the parabolic setting, first note that if $(\cE, \varphi)$ is a parabolic Higgs bundle with weights 
$0 \leq \alpha_1(p) < \alpha_2(p) <1$, then $\Det\,\cE$ inherits a parabolic structure with weight $\alpha_1(p) + \alpha_2(p)=1$.
Hence in the same local holomorphic coordinate $z$ as above, and using the holomorphic frame $s_1 \wedge s_2$, 
the metric $|z|^{2}=|z|^{2(\alpha_1(p)+\alpha_2(p))}$ is adapted to the induced parabolic structure on $\Det\,\cE$.

\begin{prop}\label{prop:goodgauge} Fix $p \in D_s$. Then for any $(\cE, \varphi) \in \cM'$, $\det \varphi$ has a simple pole at $p$, and there is a holomorphic coordinate $z$ centered
at $p$ such that $-\det\varphi=z^{-1}\de z^2$, and in some holomorphic gauge, 
\begin{equation*}\label{eq:goodgaugeholinflem}
   \delbar_E= \delbar,  \quad 
    \varphi =
   \begin{pmatrix} 0 & 1 \\ \frac{1}{z} & 0 \end{pmatrix} \de z,  \quad 
  h_\infty = Q h_\infty^{\diamond}  \qquad \mbox{where } h_\infty^{\diamond} =|z|\begin{pmatrix} |z|^{-1/2} & 0  \\ 0 & |z|^{1/2} \end{pmatrix}.
\end{equation*}
Here, $Q$ is a locally-defined function completely determined by $h_{\Det \, \cE}$, the choice of holomorphic section of $\Det \, \cE$, the unique coordinate $z$, and the parabolic weights.
\end{prop}
\begin{proof}
The proof is a small modification of \cite[Lemma 4.2]{MSWW14}.  Choose a holomorphic coordinate $z$ around $p$ so that $-\det \varphi
=z^{-1} \de z^2$, and fix some local holomorphic frame.  Since $p$ is a simple pole of $\det \varphi$, $\left.z\varphi\right|_{z=0}$ is nilpotent,
but not identically zero. Applying a constant gauge transformation, we may assume that 
 \begin{equation*}
  z \varphi = \begin{pmatrix} a(z) & b(z) \\ c(z) & -a(z) \end{pmatrix} \de z \qquad \left. z \varphi\right|_{z=0} = \begin{pmatrix} 0 & 0 \\ 1 & 0 \end{pmatrix}.
 \end{equation*}
Since $c(0)=1$, $\sqrt{c(z)}$ is well-defined and holomorphic near $0$, so if we define
 \begin{equation*}
  g(z) = \frac{1}{\sqrt{c(z)}} \begin{pmatrix} 1 & a(z) \\ 0 & c(z) \end{pmatrix},
 \end{equation*}
then
 \[
g^{-1} \varphi g = \begin{pmatrix} 0 & 1 \\ \frac{1}{z} & 0 \end{pmatrix} \de z,
\] 
as needed. The metric $h_\infty$ defined here satisfies $F_{D(\delbar_E, h_\infty)}^\perp=0$, normalizes the Higgs field and induces the correct metric on $\Det\,\cE$. 
\end{proof}
In unitary gauge, 
\begin{align*}\label{model_simple_pole}
d_A&= \de+ \frac{1}{2}\left(\frac{\del Q}{Q} - \frac{\delbar Q}{Q}\right)\begin{pmatrix}1 & 0 \\ 0 & 1 \end{pmatrix} + \frac{1}{2} \begin{pmatrix}i&0\\0& i\end{pmatrix} \de \theta- \frac{1}{4} \begin{pmatrix}i&0\\0& -i\end{pmatrix} \de \theta, \\
\Phi&= \begin{pmatrix} 0 & |z|^{-1/2} \\ |z|^{1/2}/z & 0 \end{pmatrix} \de z.
\end{align*}

\begin{rem}
Note that if $(\cE, \varphi) \in \cM^{\mathrm{sing}} = \cM \setminus \cM'$, then there need not be a frame near $p \in D_s$ where $(\delbar_E, \varphi, h_\infty)$ has local form in Proposition \ref{prop:simplepolemodel}. In the extreme case, in any moduli space of strongly parabolic Higgs bundles, there is a distinguished point $q\equiv 0 \in \cB$; for any Higgs bundle in the nilpotent cone and any choice of $p \in D_s$, there is no holomorphic coordinate $z$ such that $-\det \varphi= z^{-1} \de z^2$ near $p \in D_s$. More generally, as one moves in $\cB$, the zeros of $q$ wander. They can coincide with each other or with a point of $D_s$, but never with $D_w$ because of the fixed non-zero residue $\sigma$.  In either of these two cases of coincidence, $q$ no longer has all simple zeros, viewed  as a section of $K_C(D)^{\otimes 2}$.
% =======
%  Note that if $(\cE, \varphi) \in \cM^{\mathrm{sing}}=\cM-\cM'$, then there need not be a frame near $p \in D_s$ where $(\delbar_E, \varphi, h_\infty)$ has local form in Proposition \ref{prop:simplepolemodel}. In the extreme case, in any moduli space of strongly parabolic Higgs bundles, there is a distinguished point $q\equiv 0 \in \cB$; for any Higgs bundle in the nilpotent cone and any choice of $p \in D_s$, there is no holomorphic coordinate $z$ such that $-\det \varphi= z^{-1} \de z^2$ near $p \in D_s$. More generally, as one moves in $\cB$, the zeros of $q$ wander. They can coincide with each other or with a point of $D_s$, but never with $D_w$ because of the fixed non-zero residue $\sigma$.  In either of these two cases of coincidence, $q$ no longer has all simple zeros, viewed  as a section of $K_C(D)^{\otimes 2}$.
\end{rem}

\begin{prop}\label{prop:goodgaugedouble} 
Fix $p \in D_w$. Then for any $(\cE, \varphi) \in \cM$, $\det \varphi$ has a double pole at $p$, and there is a holomorphic coordinate $z$ centered
centered at $p$ such that $-\det\varphi=\sigma^2z^{-2}\de z^2$ for some $\sigma \in \C^\times$ and a holomorphic gauge such that 
\begin{equation*}\label{eq:goodgaugeholinf}
\delbar_E= \delbar,  \quad 
\varphi = \frac{\sigma}{z}\begin{pmatrix} 1 & 0 \\ 0 & -1  \end{pmatrix} \de z,  \ \ \mbox{and}\quad
h_\infty = Q h_\infty^{\diamond}  \qquad \mbox{where } h_\infty^{\diamond}=\begin{pmatrix} |z|^{2\alpha_1(p)} &  0 \\  0 & |z|^{2\alpha_2(p)} \end{pmatrix}. 
\end{equation*}
Here $Q$ is a locally-defined function completely determined by $h_{\Det \, E}$, the choice of holomorphic section of $\Det \,\cE$, the coordinate $z$, and the parabolic weights.
\end{prop}
The constant $\sigma$ is an invariant since the residue of $\sigma \de z/z$ is independent under holomorphic change of variables. 
\begin{proof}
Choose $z$ so that $-\det\varphi=\sigma^2z^{-2}\de z^2$, so $z\varphi$ is regular at $p$ and $-\det (z \varphi) = \sigma^2 \de z^2$. Hence there
exists a local holomorphic frame such that
\[
z \varphi = \begin{pmatrix} \sigma & 0 \\ 0 & - \sigma \end{pmatrix} \de z, \ \ \mbox{i.e., } \ \ \  \varphi = \frac{\sigma}{z}
 \begin{pmatrix} 1 & 0 \\ 0 & - 1 \end{pmatrix} \de z.
\]
As before, the metric $h_\infty$ satisfies $F_{D(\delbar_E, h_\infty)}^\perp=0$, normalizes the Higgs field and induces the correct metric on $\Det\,\cE$. 
\end{proof}

In unitary gauge,
\begin{align*}\label{model_double_pole}
d_A&= \de+ \frac{1}{2} \left( \frac{\del Q}{Q} - \frac{\delbar Q}{Q} \right) \begin{pmatrix} 1 & 0 \\ 0 & 1 \end{pmatrix} + \begin{pmatrix} i \alpha_1 & 0 \\ 0 & i\alpha_2 \end{pmatrix} \de \theta, \\
\Phi&=\frac{\sigma}{z} \begin{pmatrix} 1 & 0 \\ 0 & - 1 \end{pmatrix} \de z.
\end{align*}

In this case, $h_\infty$ is already adapted to the parabolic structure, hence requires no further desingularization.

\subsection{Fiducial solutions} \label{sec:localmodelht}
We next display the model solutions on the complex plane, known as the fiducial solutions, which we shall use to  `smooth out' the
limiting configurations. These model solutions were used for this purpose in \cite{MSWW14}, but also arose many years earlier in 
the physics literature. %We show here that modifications of these work near simple poles as well.

\begin{prop}\cite{MSWW14}\label{prop:simplezeromodel} 
There exists a family $(\delbar_E, t \varphi, h_t)$ of smooth solution of Hitchin's equations 
%$F_{D(\delbar_E, \varphi)} +[\varphi, \varphi^{*_h}]=0$
on $\C$ with
\begin{equation*}
\delbar_E =\delbar,  \quad \varphi = \begin{pmatrix} 0 & 1 \\ z & 0 \end{pmatrix} \de z,  \quad h_t^{\mathrm{model}}=
\begin{pmatrix} r^{1/2} \e^{\ellt(r)} &0 \\0 & r^{-1/2} \e^{-\ellt(r)} \end{pmatrix},
\end{equation*}
where, here and below, $r = |z|$. The function $\ellt$ is the solution of the Painlev\'e equation
\begin{equation} \label{eq:Painleve}
\left(\frac{\de^2}{\de r} + \frac{1}{r} \frac{\de}{\de r} \right) \ellt=  8 t^2 r \sinh(2 \ellt)
\end{equation}
with asymptotics 
\[
%\begin{tabular}{ l l l }\label{eq:asymptotics}
\ellt(r) \sim  \frac{1}{\pi} K_0\left(\frac{8}{3} t r^{\frac{3}{2}}\right)\ \mbox{as}\ r \rightarrow \infty, \quad 
\ellt(r) \sim - \frac{1}{2} \log r\ \mbox{as}\ r \rightarrow 0.
\]
\end{prop}

In unitary gauge, this solution takes the form
\[
\begin{aligned}
d_A&=
\de+ \frac14 \left(\frac{1}{2} + r \frac{\de \ellt}{\de r} \right) \begin{pmatrix} 1 & 0\\ 0 & -1 \end{pmatrix}
\left( \frac{\de z}{z} - \frac{\de \zbar}{\zbar} \right)\\[0.5ex] 
\Phi&=\begin{pmatrix} 0 & r^{1/2}e^{\ellt(r)} \\
z r^{-1/2}e^{-\ellt(r)} & 0 \end{pmatrix}\de z.
\end{aligned}
\]
Note that $2i \de \theta = \de z/z - \de \zbar/\zbar$.

\medskip

Near a simple pole there is a similar radial solution.
\begin{prop} \label{prop:simplepolemodel} For parabolic weights $0 \leq \alpha_1 < \alpha_2 <1$, there exists a family
$(\delbar_E, t \varphi, h_t)$ of solutions of Hitchin's equations which are smooth on $\C^\times = \C \setminus \{0\}$, with
\begin{equation*}\label{eq:simplepolemodel}
\delbar_E =\delbar, \quad \varphi = \begin{pmatrix} 0 & 1 \\ \frac{1}{z} & 0 \end{pmatrix} \de z,  \quad h_t^{\mathrm{model}}= 
r \begin{pmatrix} r^{-1/2} \e^{\mt(r)} & 0\\ 0& r^{1/2} \e^{-\mt(r)} \end{pmatrix}.
\end{equation*}
Here $\mt$ solves 
\begin{equation}
\left( \frac{\de^2\, }{\de r^2} + \frac{1}{r} \frac{\de}{\de r} \right) \mt = 8 t^2 r^{-1} \sinh (2 \mt)
\label{vt}
\end{equation}
and satisfies
\begin{equation*} \label{eq:asymptoticsinf}
\mt \sim  \frac{1}{\pi} K_0(8tr^{\frac{1}{2}}) \ \mbox{as}\  r\rightarrow \infty, \quad
 \mt \sim  \left( \frac{1}{2}  + \alpha_1-\alpha_2 \right) \log r\ \mbox{as}\ r \rightarrow 0,
\end{equation*}
so that
\begin{equation*}
\label{eq:asymptoticsinf2}
r \cdot r^{-\tfrac12} \e^{\mt(r)} \sim \, r^{2\alpha_1},\qquad \ r\cdot r^{\tfrac12} \e^{-\mt(r)} \sim \, r^{2\alpha_2}.
\end{equation*}
\end{prop}
The corresponding filtration at $z=0$ is 
 \begin{equation*} \label{eq:parabolic}
% \begin{array}{rccccl}
F_1=\C^2 \supset F_2=\left\langle \begin{pmatrix}0 & 1 \end{pmatrix}^T \right\rangle \supset 0
% \end{array}
\end{equation*}
with weights  $0 \leq \alpha_1  < \alpha_2 < 1$. 

In unitary gauge 
\begin{equation} \label{eq:unitarygaugemodel}
\begin{aligned}
\de_{A_t} & = \de + \left(\frac{1}{4}\begin{pmatrix} 1 & 0 \\ 0 & 1 \end{pmatrix} + F_t^p(r)
\begin{pmatrix} 1 & 0 \\ 0 & -1 \end{pmatrix} \right)\left(\frac{\de z}{z} - \frac{\de \zbar}{\zbar} \right)\\ 
\Phi_t & = \begin{pmatrix} 0 & r^{-1/2} \e^{\mt(r)} \\ z^{-1} r^{1/2} \e^{-\mt(r)}&0 \end{pmatrix} \de z,
\end{aligned}
\end{equation}
where
\begin{equation}\label{ftpr}
F_t^p(r)=\frac14\left( - \frac12 + r \mt'(r)\right). 
\end{equation}

\begin{proof}
The existence of the family $h_t^{\mathrm{model}}$ may be understood in two ways. First, by the change of variables
$\rho = 8 t r^{1/2}$, we have
\[
\left( \frac{\de^2\, }{\de \rho^2} + \frac{1}{\rho} \frac{\de\, }{\de \rho} \right) \mt =  \frac{1}{2} \sinh (2 \mt), 
\]
McCoy-Tracy-Wu \cite{McCoy-Tracy-Wu} prove the existence and uniqueness of solutions of this ODE on $(0, \infty)$ with the 
prescribed asymptotic behavior. 

Alternately, we may also follow \cite[\S1.10]{FredricksonNeitzke} where Proposition \ref{prop:simplezeromodel} is derived at simple zeros.
In the terminology of \cite{MochizukiToda}, $(\cE, t \varphi)$ is a family of ``good filtered Higgs bundle'' on 
$\CP^1$ with marked points $0$ and $\infty$.  The singularity at $0$ is regular and the one at $\infty$ irregular. Consequently,
by \cite{BiquardBoalch,Mochizukiwild}, there is an associated family of adapted harmonic metrics $h_t^{\mathrm{model}}$.
These good filtered Higgs bundle are fixed by the $\C^\times$-action that rotates the base and simultaneously rescales the 
Higgs field.  Hence $h_t^{\mathrm{model}}$ has the claimed shape and radial symmetry.
\end{proof}

\subsection{Approximate solutions}\label{subsect:approxsol} % \texorpdfstring{$h_t^{\app}$}{}}
We now assemble the two pieces above to construct a family of adapted Hermitian metrics $h_t^{\app}$ which 
approximately solve the $t$-rescaled Hitchin's equations in \eqref{eq:trescaled}. In fact, we show in Proposition \ref{prop:defofapprox} below 
that $F^\perp_{D(\delbar_E, h_t^{\app})} +t^2[\varphi, \varphi^{*_{h_t^{\app}}}]$ decays exponentially in $t$. 
The further step of perturbing $h_t^{\mathrm{app}}$ to an exact harmonic metric requires a closer study of
the linearized operator, which is carried out in the next section. 

%Given $(\cE, \varphi) \in \cM'$, we now define an approximate harmonic metric on all of $C$ by joining together the limiting solution
%$h_\ell$ and the model solutions described above near poles and zeroes. More specifically,  
Choose a smooth nonnegative cutoff function $\chi$ on $\R^+$ taking values in $[0,1]$, with $\chi(r) = 1$ for $r \leq 1/2$ 
and $\chi(r) = 0$ for $r \geq 1$.  Now define $h_t^\app$ on $\cE$ as follows: 
\begin{itemize}
\item in a holomorphic coordinate and gauge near zeros of $\det \varphi$, where $\delbar_E$ and $\varphi$ are in normal form, set
%$\delbar_E =\delbar$, $\varphi = \begin{pmatrix} 0 & 1 \\ z & 0 \end{pmatrix} \de z$, 
\begin{equation*}\label{eq:htfiducialzero}
h_t^{\app}:= Q \begin{pmatrix} r^{1/2} \e^{\ellt (r) \chi (r)} & 0\\ 0& r^{-1/2} \e^{-\ellt(r)\chi(r)} \end{pmatrix};
\end{equation*}
\item in a holomorphic coordinate and gauge near simple poles of $\det \varphi$ where $\delbar_E$ and $\varphi$ are in normal form, set
% =\delbar$, $\varphi = \begin{pmatrix} 0 & 1\\ \frac{1}{z} & 0 \end{pmatrix} \de z$, 
\begin{equation} \label{eq:htapprox}
h_t^{\app}= Q r \begin{pmatrix} r^{-1/2} \e^{\mt(r) \chi(r)} &0 \\0 & r^{1/2} \e^{-\mt(r)\chi(r)} \end{pmatrix};
\end{equation}
\item elsewhere on $C$, set $h_t^\app = h_\infty$.
\end{itemize}
Here $Q$ is the same locally-defined function appearing in Proposition \ref{prop:simplezeronormalform} and \ref{prop:goodgauge}; consequently $h_t^{\app}$ induces the metric $h_{\Det \,E}$ on $\Det\, E$. Recall that the function $Q$ is completely determined by $h_{\Det \, E}$, the choice of holomorphic section of $\Det\, \cE$, the coordinate $z$, and parabolic weights. For the analysis, it will be important to note that $Q$ is smooth.

\begin{prop}\label{prop:defofapprox} 
For $t_0 \gg 1$, there exist positive constants $c$, $\mu$, such that for any $t>t_0$, 
\begin{equation*}
\Big\| F^\perp_{D(\delbar_E, h_t^{\app})} + t^2 \left [\varphi, \varphi^{*_{h_t^\app}} \right] \Big\|
\leq c \e^{-\mu t}. 
\end{equation*}
The norm here can be taken either in $L^2$ or in a H\"older norm with respect to $h_t^{\app}$ and the fixed 
Riemannian metric on $C$.
\end{prop}
\begin{proof}
Clearly $h_t^\app$ solves Hitchin's equations exactly away from the annuli where $\chi' \neq 0$. The 
exponential decay rates of $\ellt$ and $r \ellt'$, cf.\ \cite[Lemma 3.3]{MSWW14}, as well as those of $\mt$ and $r \mt'$
stated above, then imply  the stated bounds.
\end{proof}

We now convert $h_t^{\app}$ into the corresponding family of connections and Higgs fields.
Fixing any compatible metric $h_0$, then given a triple $(\delbar_E, \varphi, h)$, there is an $\End E$-valued $h_0$-Hermitian
section $H$ such that $h(w_1,w_2)=h_0(Hw_1,w_2)$.  The complex gauge transformation $g=H^{-1/2}$ satisfies
$g \cdot h=h_0$ since in general $(g \cdot h)(w_1,w_1)=h(g w_1,g w_2)$.  The associated pair $(\de_A, \Phi)$ 
equals $\delbar_A = H^{1/2} \circ \delbar_E \circ H^{-1/2}$, $\Phi = H^{1/2} \varphi H^{-1/2}$. Using these formul\ae, we 
record the forms of these fields in unitary gauge:  
\begin{subequations}
\begin{equation*}
\begin{aligned}
A_t&=\de + \frac12 \left( \frac{\del Q}{Q} - \frac{\delbar Q}{Q} \right)\begin{pmatrix} 1 & 0 \\ 0 & 1 \end{pmatrix} +\frac12 \left(\frac12 + r \frac{\de (\ellt \chi)}{\de r} \right)\begin{pmatrix} i & 0 \\ 0 & -i \end{pmatrix} \de \theta \\ 
\Phi_t&= \begin{pmatrix} 0 & r^{1/2} \e^{\ellt \chi} \\ z r^{-1/2} \e^{-\ellt \chi} &0\end{pmatrix} \de z
\end{aligned}
\end{equation*} 
on disks where $-\det \Phi = z \de z^2$, 
\begin{equation*}
\begin{aligned}
A_t&=\de +  \frac12 \left( \frac{\del Q}{Q} - \frac{\delbar Q}{Q} \right)\begin{pmatrix} 1 & 0 \\ 0 & 1 \end{pmatrix}+ \frac{1}{2}\begin{pmatrix} i & 0 \\ 0 & i \end{pmatrix}\de \theta   + \frac12 \left( -\frac12 + r\frac{\de (\mt\chi)}{\de r}\right) \begin{pmatrix} i & 0 \\ 0 & -i \end{pmatrix}\de \theta, \\
%+ \left(- \frac{1}{8} + \frac{|z|}{4} \frac{\de (v_t \chi)}{\de |z|} \right)\begin{pmatrix} 1 & 0 \\ 0 & -1 \end{pmatrix}2 i \de \theta \\
\Phi_t&=\begin{pmatrix} 0 & r^{-1/2} \e^{\mt \chi} \\ z^{-1} r^{1/2} \e^{-\mt \chi}&0 \end{pmatrix} \de z
\end{aligned}
\end{equation*}
on disks where $-\det \Phi = z^{-1} \de z^2$, and 
\begin{equation*}
A= \frac12 \left( \frac{\del Q}{Q} - \frac{\delbar Q}{Q} \right)\begin{pmatrix} 1 & 0 \\ 0 & 1 \end{pmatrix}+ \begin{pmatrix} i \alpha_1 & 0 \\ 0 & i\alpha_2 \end{pmatrix} \de \theta,\qquad \Phi=\frac{\sigma}{z} \begin{pmatrix} 1 & 0 \\ 0 & - 1 \end{pmatrix} \de z
\end{equation*}
on disks where $-\det \Phi = \sigma^2z^{-2}\de z^2$. 
\end{subequations}
 	
\section{The linearization}\label{sec:ht}

Define the nonlinear operator 
\begin{equation*}\label{eq:F1}\mathbf{F}_t(\delbar_E, \varphi, h)
:=H^{1/2} \left(F_{D(\delbar_E, h)}^\perp + t^2[ \varphi ,\varphi^{*_h} ] \right) H^{-1/2},
\end{equation*}
where $H^{1/2}$ is the $\End(E)$-valued $h_0$-Hermitian section satisfying $h(v, w)=h_0(H^{1/2} v, H^{1/2}w)$,  as discussed 
at the end of the introduction of \S\ref{sec:background}. By doing this, the output $\mathbf{F}(\delbar_E, \varphi, h)$ 
is an $h_0$-skew-Hermitian section of $\Omega^{1,1}(C, \End E)$.
Or equivalently in the unitary formulation of Hitchin's equations, this operator $\mathbf{F}_t$ is equal to 
$\mathbf{F}_t(A_t, t\Phi_t) = F^\perp_{A_t} + t^2[\Phi_t, \Phi_t^{*_{h_0}}].$)
The output of the operator $\mathbf{F}_t$ measures the failure of $(\delbar_E, \varphi, h)$ from being a solution 
of the $t$-rescaled Hitchin equations. We fix the underlying Higgs bundle $(\delbar_E, \varphi)$ and regard 
$\mathbf{F}_t$ as an operator acting on Hermitian 
metrics, which we assume are perturbations of the approximate solution $h_t^\app$.  In unitary gauge, the Higgs bundle data 
becomes 
\[
\Phi_t=\left(H_t^{\app}\right)^{1/2} \circ \varphi \circ \left(H_t^{\app}\right)^{-1/2},\quad   A^{0,1}_t=
 \left(H_t^{\app}\right)^{1/2} \circ \delbar_E \circ\left(H_t^{\app}\right)^{-1/2},
\]
where $H_t^{\app}$ is the $\End\; E$-valued $h_0$-Hermitian section such that\begin{equation}
h_t^{\app}(v,w) = h_0((H_t^{\app})^{1/2} v, (H_t^{\app})^{1/2}w).
\end{equation}
Local expressions for these fields near the zeros and poles of $\det \varphi$ are recorded above. Finally, write
$h_t(w_1,w_2)= h_t^\app(\e^{\gamma_t} w_1, \e^{\gamma_t} w_2)$, or equivalently $H_t^{1/2}=(H_t^{\app})^{1/2}\e^{\gamma_t} $,
 and break unitary invariance by assuming that $\gamma_t$  is $h_0$-Hermitian. Thus we 
focus on the operator
\begin{equation*}\label{eq:Fapp}
\calF_t(\gamma)
:= F_{A_t^{\exp(\gamma)}}^\perp + t^2[\e^{-\gamma}\Phi_t \e^{\gamma},\e^{\gamma} \Phi_t^{*_{h_0}} \e^{-\gamma}]. 
\end{equation*}
Note  that this operator is computed relative to the background approximate solution fields.  The error term $\calF_t(0)$, which 
is supported on a union of annuli around the zeros and poles of $\det \varphi$, was estimated in the last section. 

Up to a simple isomorphism, the linearization of $\calF_t$ at $0$ equals 
\begin{equation*} \label{eq:Lt}
\mathcal L_t\gamma:= -i \star D\calF_t(0)[\gamma]= - i \star \frac{\de}{\de\eps}\Big|_{\eps=0} \calF_t(\eps \gamma) 
= \Delta_{A_t} \gamma - i \star t^2 M_{\Phi_t} \gamma,
\end{equation*}
where 
\begin{equation*}
\Delta_{A_t} := d_{A_t}^* d_{A_t} \gamma, \ \mbox{and}\ \  M_{\Phi_t}\gamma := \left[\Phi_t^* \wedge [\Phi_t,  \gamma ] \right]
-\left[\Phi_t  \wedge [\Phi_t^*, \gamma] \right].
\end{equation*}

The set of exceptional points in the analysis below is a union $Z \cup D$, $D = D_s \cup D_w$, where $Z$ is the set of zeros of 
$\det \varphi$ (all simple, by assumption), and $D_s$ and $D_w$ are the sets of strongly and weakly parabolic points, respectively. 
Near any $p \in D$, choose a holomorphic frame $(e_1,e_2)$ compatible with the flag at $p$ and the model metric 
\[
h_{\vec{\alpha}} = \begin{pmatrix}
 |z|^{2\alpha_1} & 0 \\
 0 & |z|^{2\alpha_2}	
 \end{pmatrix}.
\]
In this holomorphic frame, the Chern connection $A_{\vec\alpha}$ of $h_{\alpha}$ equals
\[
d_{A_{\vec\alpha}} = \bar \partial +\partial^{h_{\vec\alpha}}= \de + \begin{pmatrix} \alpha_1 & 0 \\ 0 & \alpha_2 \end{pmatrix} \frac{\de z}{z},
\]
while in the unitary frame $(\widetilde e_1,\widetilde e_2)$, where $\widetilde e_i=e_i/|z|^{\alpha_i}$, it equals
\[
d_{A_{\vec\alpha}} %d + \frac{1}{2}   \begin{pmatrix} \alpha_1 & 0 \\ 0 & \alpha_2 \end{pmatrix} \left( \frac{\de z}{z} - \frac{\de \bar z}{\bar z} \right)
= \de + \begin{pmatrix} i \alpha_1 & 0 \\ 0 & i\alpha_2 \end{pmatrix} \de\theta. 
\]

We see from these expressions that near each $p \in D$, $A_{\vec\alpha}$ has a simple pole, $A_t$ differs from $A_{\vec\alpha}$ only 
by lower order terms, and the pole of $\Phi_t$ is simple. This means that the operator $\mathcal L_t$ is an elliptic operator of 
conic type, with singularities at points of $D$. Of course, the coefficients of $\cL_t$ are smooth near points of $Z$, but 
this operator develops conic singularities as $t \to \infty$ and we must keep track of this `emergent' behavior.   

We now describe the precise local expressions for $\cL_t$ near each of these three types of points.  In preparation for analyzing
the mapping properties of $\cL_t$, both for each fixed $t$ and uniformly as $t \to \infty$,  we also describe the indicial roots of $\cL_t$.
By definition, a number $\nu$ is called an indicial root of a conic operator $L$ if there exists $\psi \in \calC^\infty(S^1)$ such that 
$L( r^\nu \psi)  = O(r^{\nu-1})$. The exponent on the right should normally be $\nu-2$, so this condition entails a leading
order cancellation. In fact, the coefficient of $r^{\nu-2}$ is a type of eigenvalue equation for $\psi$.   The indicial
roots play a significant role in determining the mapping properties of $L$ and regularity properties of its solutions.

\begin{rem}\label{rem:Q1}
In the local analysis, we assume $Q=1$.  Because $Q$ is smooth and everywhere positive, this simplification
does not change the mapping properties of the operator.
\end{rem}

\subsection{\texorpdfstring{$\cL_t$}{Lt} near simple zeros} \label{sec:linearizationzero}
Recall from Proposition \ref{prop:simplezeromodel} that
\[
d_{A_t} = \de + F_t^0 \begin{pmatrix} i & 0 \\ 0 & -i \end{pmatrix} \de\theta,\ \ \Phi_t = \begin{pmatrix} 0 & |z|^{1/2} \e^{\ellt} \\
z |z|^{-1/2} \e^{-\ellt} & 0 \end{pmatrix} \de z,
\]
where $\ellt$ solves \eqref{eq:Painleve} and $F_t^0(r) = \frac12 (\frac12 + r\del_r \ellt)$.  Writing
$\gamma=\begin{pmatrix} u_0&u_1\\\bar u_1&-u_0\end{pmatrix}$, then we calculate that 
\begin{multline*}
\cL_t \gamma = (\Delta_{A_t} - i \star t^2 M_{\Phi_t}) \begin{pmatrix}   u_0 & u_1 \\ \bar{u}_1 & -u_0 \end{pmatrix} \\  = 
\begin{pmatrix}
\Delta_0 u_0  & \Delta_0 u_1 - 4F_t^0 i \del_\theta u_1 + 4(F_t^0)^2 u_1 \\ 
\Delta_0 \bar{u}_1 + 4F_t^0 i \del_\theta \bar{u}_1 + 4(F_t^0)^2 \bar{u}_1 &  -\Delta_0 u_0 \end{pmatrix}  \\ +  
8r t^2 \begin{pmatrix} 2 \cosh(2\ellt) u_0  &  \cosh(2\ellt) u_1 - \e^{-i\theta} \bar{u}_1 \\
\cosh(2 \ellt)  \bar{u}_1 - \e^{i\theta} u_1 &  -2 \cosh(2\ellt) u_0 \end{pmatrix}.
\end{multline*}

Note that $F_t^0$ vanishes and $r \cosh(2\ellt)$ is regular at $r= 0$; in fact $\cL_t$ has only polar coordinate
singularities at $r=0$. Thus its indicial roots are simply those of the scalar Laplacian $\Delta_0$, namely
$\mathbb Z$.

The Hermitian operator $-i \star M_{\Phi_t}$ is positive definite for every $t > 0$, with eigenvalues 
\[
\lambda_0 = 16r \cosh 2\ellt,\ \ \lambda_1 = 8r (\cosh 2\ellt  - 1),\ \ \lambda_2 = 8r (\cosh 2\ellt + 1).
\]
The eigenvector for $\lambda_0$ has $\gamma$ diagonal, i.e., $u_1 \equiv 0$, while the other two eigenspaces
are spanned by vectors with $u_0 \equiv 0$. 

We also consider the limits of the two summands $\Delta_{A_t}$ and $M_{\Phi_t}$ as $t \to \infty$. Near any point of $Z$, 
\begin{equation*}\label{eq:Ainftynearzero}
d_{A_{\infty}}= \de +\frac{1}{4}\begin{pmatrix}i&0\\0&-i	
\end{pmatrix}\, \de\theta,\qquad \Phi_{\infty}=\begin{pmatrix}
0&r^{\frac{1}{2}}	\\ z r^{-\frac{1}{2}}&0 \end{pmatrix}\, \de z,
\end{equation*}
and hence 
\[
\Delta_{A_\infty}\gamma=\begin{pmatrix}
\Delta_0 u_0  &  \Delta_{\frac12} u_1\\ \overline{\Delta_{\frac12} {u}_1} & -\Delta_0 u_0 \end{pmatrix}, \qquad
-i \star M_{\Phi_\infty} \gamma = 8r\begin{pmatrix}
2u_0&u_1-e^{-i\theta}\bar u_1\\
\bar u_1-e^{i\theta}\bar u_1&-2u_0	
\end{pmatrix},
\]
where
\[
\Delta_{1/2} := -\partial_r^2	u_1-\frac{1}{r}\partial_r u_1-\frac{1}{r^2}\left(\partial_{\theta} + \frac i2\right)^2.
\]
The limiting operator $-i \star M_{\Phi_\infty}$ is only semi-definite, with eigenvalues
\[
\lambda_0 = 16r,\ \ \lambda_1 = 0, \ \ \lambda_2 = 16r. 
\] 

Later on we shall consider the conic operator $\cL_\infty^o = \Delta_{A_\infty} - i \star M_{\Phi_\infty}$.  This has 
indicial roots consisting of the integers, for the diagonal terms, and $\mathbb Z + \frac12$ for the off-diagonal terms. 

\subsection{\texorpdfstring{$\cL_t$}{Lt} near strongly parabolic points}  \label{sec:linearizationpole}
Following \eqref{eq:unitarygaugemodel}, the model solution at a strongly parabolic point, in unitary gauge, is
\begin{eqnarray*} 
\de_{A_t} &=& \de + \left(\frac{1}{2}\begin{pmatrix} i & 0 \\ 0 & i \end{pmatrix} + 2F_t^p(r) 
\begin{pmatrix} i & 0 \\ 0 & -i \end{pmatrix} \right) d \theta\\ \nonumber 
\Phi_t &=& \begin{pmatrix} 0 &  \e^{\sigma_t(r)} \\ z^{-1} \e^{-\sigma _t(r)} \end{pmatrix} \de z,
\end{eqnarray*}
where $\mt$ and $F_t^p$ are as in \eqref{vt} and \eqref{ftpr}, and $\sigma_t(r)=\mt(r)-\frac{1}{2}\log r$. 
%Let 
%$$
%H_t=|z|^{\alpha_1+\alpha_2}\begin{pmatrix}
%	|z|^{-\frac{1}{2}}e^{v_t(|z|)}&0\\
%	0&|z|^{\frac{1}{2}}e^{-v_t(|z|)}
%\end{pmatrix}=:|z|^{\alpha}\begin{pmatrix}
%	 e^{\sigma_t(|z|)}&0\\0&e^{-\sigma_t(|z|)}
%\end{pmatrix}
%$$
%be the model Hermitian metric. 
%The corresponding model solution of the self-duality equations is then $(A_t,\Phi_t)=g_t^{\ast}(\Phi_0,\bar\partial)$, where the complex gauge transformation $g_t$ satisfies $g_t^2=H_t^{-1}$. Defining
%$$
%g_t=|z|^{-\alpha/2}\begin{pmatrix}
%	 e^{-\sigma_t(|z|)/2}&0\\0&e^{\sigma_t(|z|)/2}
%	 \end{pmatrix}=\exp\begin{pmatrix}
%	-\frac{\sigma_t(|z|)}{2} -\frac{\alpha}{2}\log|z|&0\\
%	0&\frac{\sigma_t(|z|)}{2}-\frac{\alpha}{2}\log|z|
%\end{pmatrix},
%$$
% where
% $$
% r_t:=-\frac{\alpha}{2}\log|z|+\frac{1}{2}\sigma_t\qquad\textrm{and}\qquad s_t:=
% $$
%we obtain that
% $$
% A_t=\begin{pmatrix}
%	a_t&0\\0&b_t
%\end{pmatrix}\de \bar z+\begin{pmatrix}
%	-\bar a_t&0\\0&-\bar b_t
%\end{pmatrix}\de z, 
%$$
%where
%$$
%a_t:=\bar\partial\Big(-\frac{\sigma_t(|z|)}{2} -\frac{\alpha}{2}\log|z|\Big),\qquad b_t:=\bar\partial\Big(\frac{\sigma_t(|z|)}{2}-\frac{\alpha}{2}\log|z|\Big),
%$$
% and 
% $$
% \Phi_t=\begin{pmatrix}
%0&e^{\sigma_t}\\\frac{e^{-\sigma_t}}{z}&0	
%\end{pmatrix}.
%$$
% Let 
% \[
% \cL_t=\Delta_{A_t} - i \star t^2M_{\Phi_t}
% \]
% denote the linearized operator at this model solution. 
%Note that
% $$
% \Delta_{A_t}=2i\ast \bar\partial_{A_t}\partial_{A_t}-i\ast F_{A_t}.
% $$
We then obtain
\begin{multline*}
\cL_t \gamma = \begin{pmatrix}
\Delta u_0 & \Delta u_1 -\frac{2i}{r^2}(4F_t^p)\partial_\theta u_1 + \frac{(4  F_t^p)^2}{r^2} u_1\\ \Delta \bar u_1 +\frac{2i}{r^2}(4F_t^p)
\partial_\theta \bar u_1 + \frac{(4F_t^p)^2}{r^2} \bar u_1  &-\Delta u_0 
\end{pmatrix}
% \begin{multline*}
 %	\Delta_{A_t}\gamma=\\
 %	\begin{pmatrix}
%	\Delta u_0 & \Delta u_1-8a_t\partial u_1+8\bar a_t\bar\partial u_1+(-4\partial a_t+4\bar\partial \bar a_t+16|a_t|^2) u_1\\ \Delta \bar u_1+8  b_t  \partial \bar u_1 -8\bar b_t\bar \partial \bar u_1 +(4\partial b_t-4\bar\partial \bar b_t+16|b_t|^2) \bar u_1  &-\Delta u_0 
%\end{pmatrix}.
%\end{multline*}
%Now since $\partial a_t =\partial\bar\partial w_t$ and the function $w_t$ is real-valued, it follows that $\bar \partial \bar a_t =\bar \partial \partial w_t$ (and similarly for the terms involving $b_t$), so that the above expression simplifies to 
% \begin{multline*}
% 	\Delta_{A_t}\gamma=\\
% 	\begin{pmatrix}
%	\Delta u_0 & \Delta u_1-8a_t\partial u_1+8\bar a_t\bar\partial u_1 +16|a_t|^2 u_1\\ \Delta \bar u_1+8  b_t  \partial \bar u_1 -8\bar b_t\bar \partial \bar u_1 +16|b_t|^2 \bar u_1  &-\Delta u_0 
%\end{pmatrix}.
%\end{multline*}
% Furthermore, with 
% \[
% M_{\Phi_t}\gamma=\left[\Phi_t^* \wedge [\Phi_t,  \gamma ] \right]
% -\left[\Phi_t  \wedge [\Phi_t^*, \gamma] \right]
% \]
% we obtain that
%	- i \star M_{\Phi_t}\gamma % &=4\begin{pmatrix}
%\frac{4}{|z|}\cosh(2v_t)u_0&-2\frac{\bar u_1}{\bar z}+(e^{2\sigma_t}+e^{-2\sigma_t}/|z|^2)u_1\\-2\frac{u_1}{z}+(e^{2\sigma_t}+e^{-2\sigma_t}/|z|^2)\bar u_1&-\frac{4}{|z|}\cosh(2v_t)u_0
%\end{pmatrix}\\
\\ + \frac{8 t^2}{r}\begin{pmatrix}
2\cosh(2\mt)u_0&-\bar u_1 \e^{i\theta} +\cosh(2\mt) u_1\\-u_1 \e^{-i\theta} + \cosh(2\mt)\bar u_1&- 2\cosh(2\mt)u_0
\end{pmatrix}.
\end{multline*}
The matrix $- i \star M_{\Phi_t}$ is Hermitian and strictly positive, with eigenvalues
$$
\lambda_0=\frac{16}{r}\cosh(2\mt),\qquad \lambda_1=\frac{8}{r}(\cosh(2\mt)-1),\qquad \lambda_2=\frac{8}{r}(\cosh(2\mt)+1).
$$
As before, the eigenvector corresponding to $\lambda_0$ occurs when $\gamma$ is diagonal.

Now, $\mt(r) \sim (\tfrac{1}{2}+\alpha_1-\alpha_2)\log r$,  $\sigma_t(r) \sim ( \alpha_1-\alpha_2)\log r$ and $F_t^p(r) \sim  \frac{\alpha_1 - \alpha_2}{4}$ 
as $r \to 0$, so rewriting
\[
d_{A_t} = d + \begin{pmatrix} i \alpha_1 & 0 \\ 0 & i \alpha_2  \end{pmatrix} \de \theta +
2 \Bigl(F_t^p - \frac{\alpha_1-\alpha_2}{4}\Bigr) \begin{pmatrix} i & 0 \\ 0 & -i \end{pmatrix} \de \theta,
\]
then the final term vanishes in the limit as $r \to 0$, and hence the leading part of $d_{A_t}$ is the Chern connection $d_{A_{\vec\alpha}}$.
Consequently, to leading order as $r \to 0$, 
\begin{multline*}
\Delta_{A_t}\gamma \sim \Delta_{A_{\vec\alpha}} \gamma =  \\ \begin{pmatrix}
\Delta u_0& \Delta u_1-\frac{2i}{r^2}(\alpha_1-\alpha_2)\partial_{\theta}u_1+\frac{(\alpha_1-\alpha_2)^2}{r^2}u_1\\
	\Delta \bar u_1+\frac{2i}{r^2}(\alpha_1-\alpha_2)\partial_{\theta}\bar u_1+\frac{(\alpha_1-\alpha_2)^2}{r^2}\bar u_1&-\Delta u_0,
\end{pmatrix}.
\end{multline*}

Furthermore, $\cosh 2\mt \sim r^{1 + 2(\alpha_1 - \alpha_2)}$, and since $|\alpha_1 - \alpha_2| < 1$, $|- i \star M_{\Phi_t}| \leq C r^{-2 + \delta}$
for some $\delta > 0$. Hence the indicial roots of $\cL_t$ are the same as those for $\Delta_{A_{\vec\alpha}}$.  
The indicial roots corresponding to $\gamma$ diagonal are the integers, just as before.   The induced operator for the off-diagonal part is 
$$
u_1 \mapsto \left(- \frac{\del^2\,}{\del r^2} - \frac1r \frac{\del\,}{\del r} - \frac{1}{r^2} (\partial_{\theta}+i(\alpha_1-\alpha_2))^2)\right) u_1.
$$
Writing $-(\del_\theta + i(\alpha_1 - \alpha_2))^2 = T$, then $\nu$ is an indicial root if and only if $\nu^2$
is an eigenvalue of $T$.  But $(T - \nu^2)\zeta=0$ has a nontrivial solution if and only if $\nu^2=
(\ell+\alpha_1-\alpha_2)^2$ for some $\ell\in \Z$, and hence 
$$
\Gamma(\cL_t) = \Gamma(\Delta_{A_t})=\Z\cup \{\pm(\ell+\alpha_1-\alpha_2) \mid \ell\in \Z\}.
$$
For later reference, $\Gamma(\Delta_{A_t}) \cap (-1,1) = \{\pm (\alpha_2 - \alpha_1)$ and $\pm (1 - (\alpha_2 - \alpha_1))\}$. 
% \[
% \begin{cases}
%  -( \alpha_2-\alpha_1),\ \alpha_2-\alpha_1,\ (\alpha_2-\alpha_1)\  \qquad &\mathrm{if}\ 0<\alpha_2-\alpha_1<\frac{1}{2}; \\	
%  -1+(\alpha_2-\alpha_1),\ 1-( \alpha_2-\alpha_1),\ \alpha_2-\alpha_1\ & \mathrm{if}\ \frac{1}{2}<\alpha_2-\alpha_1<1.
% \end{cases}
% \]

%\begin{align*}
%0<\alpha_1<\frac{1}{4}&\colon & \nu=-2\alpha_1,\quad 2\alpha_1,\quad 1-2\alpha_1,\quad 1+2\alpha_1;\\	
%\frac{1}{4}<\alpha_1<\frac{1}{2}&\colon & \nu=-1+2\alpha_1,\quad 1-2\alpha_1,\quad 2\alpha_1,\quad 2-2\alpha_1;\\
%\frac{1}{2}<\alpha_1<\frac{3}{4}&\colon & \nu=1-2\alpha_1,\quad -1+2\alpha_1,\quad 2-2\alpha_1,\quad 2\alpha_1;\\	
%\frac{3}{4}<\alpha_1<1&\colon & \nu=-2+2\alpha_1,\quad 2-2\alpha_1,\quad -1+2\alpha_1,\quad 3-2\alpha_1.
%\end{align*}

%In unitary gauge, 
%\begin{equation}\label{model_simple_pole}
%d_A= \de+ \frac{\alpha_1+\alpha_2}{2}  \begin{pmatrix}i&0\\0& i\end{pmatrix} \de \theta   - \frac{1}{4} \begin{pmatrix}i&0\\0& -i\end{pmatrix} \de \theta, \qquad 
%\Phi= \begin{pmatrix} 0 & |z|^{-1/2} \\ |z|^{1/2}/z & 0 \end{pmatrix} \de z.
%\end{equation}

As $t \to \infty$, the fields converge to 
\begin{equation*}\label{eq:Ainftynearpole}
d_{A_{\infty}}=\de+ \frac{1}{2}  \begin{pmatrix}i&0\\0& i\end{pmatrix} \de \theta   - \frac{1}{4} \begin{pmatrix}i&0\\0& -i\end{pmatrix} \de \theta,\qquad
   \Phi_{\infty}=\begin{pmatrix} 0&r^{-\frac{1}{2}}	
\\ z^{-1}r^{\frac{1}{2}}&0 \end{pmatrix}\, \de z, 
\end{equation*}
and hence,
denoting
\[
\Delta_{-1/2} := -\partial_r^2	u_1-\frac{1}{r}\partial_r u_1-\frac{1}{r^2}\left(\partial_{\theta} - \frac i2\right)^2,
\]
%\[
%\Delta_{\boldsymbol{\alpha}} u_1:= \left(- \frac{\del^2\,}{\del r^2} - \frac1r \frac{\del\,}{\del r} - \frac{1}{r^2} \Big(\partial_{\theta}+i\big(\alpha_1+\alpha_2-\frac{1}{2}\big)^2\Big)\right) u_1,
%\]
%, in the special case $\alpha_1 + \alpha_2  = 1$, 
the operators are 
\[
\Delta_{A_\infty}\gamma=\begin{pmatrix}
\Delta_0 u_0  &  \Delta_{-1/2} u_1 \\ \overline{\Delta_{-1/2} u_1} & -\Delta_0 u_0 \end{pmatrix},\quad
-i \star M_{\Phi_\infty} \gamma = \frac{8}{r}\begin{pmatrix}
2u_0&u_1-e^{i\theta}\bar u_1\\ \bar u_1-e^{-i\theta}\bar u_1&-2u_0	
\end{pmatrix}.
\]
Thus, here too $M_{\Phi_\infty}$ has no effect on the indicial roots, and to leading order, $\cL_{\infty}^o   := \Delta_{A_\infty} - i \star M_{\Phi_\infty} 
\sim \Delta_{A_\infty}$ so $\Gamma(\cL_\infty^o) = \Z \cup \big(\Z+\frac{1}{2}\big)$.

\subsection{\texorpdfstring{$\cL_t$}{Lt} near weakly parabolic points}  \label{sec:linearizationweakpole}
Finally, if $p \in D_w$, then near $p$, the model solution is the $t$-independent pair of fields
\[
d_{A}=\de+ \begin{pmatrix} i \alpha_1 & 0 \\ 0 & i\alpha_2 \end{pmatrix} \de \theta, \quad 
\Phi=\frac{\sigma}{z} \begin{pmatrix} 1 & 0 \\ 0 & - 1 \end{pmatrix} \de z.
\]
Thus $d_A=d_{A_{\vec\alpha}}$ is the Chern connection of the model metric and  $\Delta_A=\Delta_{A_{\vec\alpha}}$.
In addition, 
\[
- i \star M_{\Phi}\gamma  =16 \frac{|\sigma|^2}{r^2}
\begin{pmatrix} 0 & u_1 \\ \bar u_1 & 0 \end{pmatrix};
\]
this transformation has eigenvalues
$$
\lambda_0=0 ,\qquad \lambda_1= 16 \frac{|\sigma|^2}{r^2}, \qquad \lambda_2 = 16 \frac{|\sigma|^2}{r^2},
$$
hence is nonnegative with nullspace consisting of the diagonal matrices.  This is a `Hardy-type'
potential term and the indicial roots depend on $\sigma$.  

The action of $\Delta_A$ on diagonal $\gamma$ is simply $\Delta_0$, so the indicial roots for this part are
the integers $\mathbb Z$.  For the off-diagonal terms of $\cL_t$, the indicial roots are 
\[
\{ \pm [ (\ell + \alpha_1 - \alpha_2)^2 + 16 |\sigma|^2 ]^{1/2}\}.
\]

\section{Mapping properties}
The operator $\cL_t$ acts on various natural function spaces. To accommodate its conic structure, the most common 
are weighted $b$-Sobolev or $b$-H\"older spaces. (There are also closely related conic versions of these spaces which
we do not introduce for simplicity.) The former are more convenient when using duality and other Hilbert space arguments, 
while the latter are more convenient for the nonlinear aspects of the problem. The geometric microlocal approach makes 
it possible to switch back and forth between the two. 

We remind the reader that in this section we take $Q=1$ for notational simplicity, as discussed in Remark \ref{rem:Q1}.

\subsection{Fredholm map and weighted spaces}
We begin by defining the basic $b$-H\"older space. Let $P$ be a finite set of points on $C$, which we take
below to be either $D$ or $Z \cup D$. For any $0 < \alpha < 1$, define 
\[
\cC^{0,\alpha}_b(C,P) = \left\{u \in \cC^0(C \setminus P): \ \sup_{z \neq z'} 
\frac{ |u(z) - u(z')| (r+r')^\alpha}{\mathrm{dist}_g(z,z')^\alpha} := [u]_{b; 0,\alpha} < \infty\right\},
\]
where we use polar coordinates $z = re^{i\theta}$ near each point of $P$. The space is defined in the standard way 
elsewhere on $C$. Note that $\mathrm{dist}_g(z,z')^\alpha$ is comparable in a neighborhood of each such $p$ 
to $|r-r'|^\alpha + (r+r')^\alpha |\theta -\theta'|^\alpha$.  Observe also that functions with this $b$-H\"older 
regularity need not be continuous at $P$.  

It is useful in defining the higher regularity $b$-H\"older (and Sobolev) spaces to pass from the surface $C$ to
its blowup at the points of $P$. This entails replacing each $p \in P$ by its circle of unit normal vectors; thus  $C_P$ is a 
manifold with boundary equal to a union of circles and with smooth structure determined  by the lifts of the smooth functions 
on $C$ and polar coordinates $(r,\theta)$. Now consider the space $\cV_b$, the span over $\cC^\infty(C_P)$ of the vector fields 
$r\partial_r$ and $\partial_\theta$. Invariantly, $\cV_b$ consists of all smooth vector fields on $C_P$ which are tangent to all 
boundaries.  In terms of these, the higher 
regularity spaces $\cC^{\ell,\alpha}_b(C, P; i\, \mathfrak{su}(E))$ consist of those sections $u$ for which $V_1 \ldots V_j 
u \in \cC^{0,\alpha}_b(C, P; i\, \mathfrak{su}(E))$ for any $j \leq \ell$ and where each $V_i \in \cV_b$.  
We also define their weighted versions $r^\nu \cC^{\ell,\alpha}_b(C, P; i\, \mathfrak{su}(E))$ for any $\nu \in \R$.

The weighted $b$-Sobolev spaces are defined by 
\[
r^\delta H^\ell_b = \{u = r^\delta v: 
V_1 \cdots V_j u \in L^2\ \forall\, j \leq \ell\ \ \mbox{and}\ V_i \in \cV_b\}.
\]

All of these definitions adapt immediately to sections of vector bundles over $C$. 

There is a useful `comparison' between the $b$-Sobolev and $b$-H\"older spaces. Functions in $r^{\nu+1}L^2$ just fail 
to lie in $r^\nu \cC^{0,\alpha}_b$, and conversely, $r^{\nu+\epsilon} \cC^{0,\alpha}_b$ is contained in $r^{\nu+1}L^2$ for any 
$\epsilon > 0$ but does not lie in this weighted $L^2$ space if $\epsilon \leq 0$. We say that the two spaces 
$r^\nu \cC^{0,\alpha}_b$ and $r^{\nu+1}L^2$ are {\it commensurate}. 

We now refocus on the case where $P = D$. It is obvious from the definitions that for every $\nu \in \R$ and $\ell = 0, 1, 2, \ldots$, 
\begin{equation}
\mathcal L_t: r^\nu \cC^{\ell+2,\alpha}_b(C, D; i \,\mathfrak{su}(E))\to r^{\nu-2}\cC^{\ell,\alpha}_b(C, D; i \,\mathfrak{su}(E))
\label{genmp}
\end{equation}
is bounded.   The first main result explains when this map is Fredholm.
\begin{prop}
Writing $\Gamma(\cL_t)$ as the set of indicial roots of $\cL_t$, fix $\nu \in R \setminus \Gamma(\cL_t)$. 
\begin{itemize}
\item[(i)] For any $\ell \geq 0$, the operator 
\[
\mathcal{L}_t:r^\nu \cC^{\ell+2,\alpha}_b(C, P; i \,\mathfrak{su}(E))\to r^{\nu-2} \cC^{\ell,\alpha}_b(C, P; i \,\mathfrak{su}(E))
\]
is Fredholm, with index and nullspace both constant as $\nu$ varies over each connected component 
of $\R\setminus\Gamma(\mathcal{L}_t)$.
\item[(ii)] If $\cL_t\gamma = \eta \in r^{\nu-2}\cC^{\ell,\alpha}_b(C, P; i \,\mathfrak{su}(E))$ and $\gamma\in r^{\nu'+1} 
L^2(C, P; i \,\mathfrak{su}(E))$ for some $\nu' < \nu$, then $\gamma = \sum \gamma_j r^{\nu_j} + \widetilde{\gamma}$, where
the finite sum is over all indicial roots with $\nu_j \in [\nu', \nu)$ and $\widetilde{\gamma} \in r^{\nu} \calC^{\ell+2,\alpha}_b$.
(If $0 \in (\nu', \nu)$, one must allow an additional term $\widetilde{\gamma}_0 \log r$ in this sum.) 
\item[(iii)] If $\cL_t \gamma = \eta \in r^{\nu-2}\calC^{\ell,\alpha}_b$ and $\gamma \in r^{\nu+1}L^2$, then
$\gamma \in r^{\nu} \calC^{\ell+2}_b$. 
\item[(iv)] If $\eta$ is polyhomogeneous and $\gamma \in r^{\nu} \calC^{2,\alpha}_b$ for some $\nu$, $\gamma$ is
also polyhomogeneous, with exponents in its expansion determined by those in the expansion for $\eta$ together 
with the indicial roots of $\cL_t$ lying in $[\nu, \infty)$.  In particular, any element of the nullspace of $\cL_t$ is 
polyhomogeneous. 
\end{itemize}
\label{mainmp}
\end{prop}
\begin{rem}
As a clarification of part iv) of this Proposition, a function (or section) $u$ is said to be polyhomogeneous means if it is smooth
away from the punctures and admits a classical expansion of the form
\[
u \sim \sum_{j=0}^\infty \sum_{\ell=0}^{N_j} u_{j\ell}(\theta) r^{\sigma_j} (\log r)^\ell,
\]
where $\sigma_j$ is a sequence of possibly complex numbers with real parts tending to infinity, and where each $u_{j\ell}(\theta)$ is smooth.
This expansion is to be interpreted in the sense that the difference between $u$ and any partial sum of the expression on the right
vanishes like the next term in the asymptotic series, and that similar estimates hold for any derivative of $u$ and the corresponding
term-by-term derivative of the formal series. 
In our specific situation, and more generally for solutions of conic elliptic operators, the exponents which occur are closely related
to the indicial roots of the problem.  Polyhomogeneity is a reasonable substitute for smoothness, but one which is broad enough to contain
solutions of conic elliptic equations.
\end{rem}

\subsection{Pseudo-Friedrichs extensions}\label{Friedrichs_domain}
We now adapt and modify a concept from $L^2$ theory: the Friedrichs extension of $\cL_t$.  Classically, this is a canonical realization of $\cL_t$ as a 
self-adjoint operator, using only its semiboundedness and associated quadratic form, and starting from the core domain $\cC^\infty_0 (C \setminus P)$.  
Briefly, for any $\gamma$ in this core domain we have that
\begin{equation}
\langle \mathcal L_t \gamma, \gamma\rangle_{L^2} = \| d_{A_t} \gamma\|_{L^2}^2 + 2t^2 \|[\Phi_t,\gamma]\|_{L^2}^2 \geq 0,
\label{quadform}
\end{equation} 
cf.\ \cite[Proposition 5.1]{MSWW14}.  Then there is a unique self-adjoint extension of $\cL_t$, defined on its Friedrichs domain $\cD_{\Fr}^{L^2}$,
with the property that \eqref{quadform} remains valid and for which the domain of this quadratic form is minimal amongst all such extensions. It is well-known 
that, despite the singularities of $\cL_t$, $\cD_{\Fr}^{L^2}$ is compactly included in $L^2$ and hence $(\cL_t, \cD_{\Fr}^{L^2})$ has discrete spectrum; 
by \eqref{quadform}, all eigenvalues lie in $[0,\infty)$.  

We now consider a H\"older space analogue which includes some additional flexibility in the range space. Fix any $\nu \geq 0$ and define
\begin{equation*}
\cD_{\Fr}^{\ell,\alpha}(\cL_t)(\nu) = \{\gamma \in \cC^{\ell, \alpha}_b:  \cL_t \gamma \in r^{\nu-2}\cC^{\ell,\alpha}_b\}.
\label{friedhold}
\end{equation*}
Since $\cL_t \gamma \in r^{\nu-2}\cC^{\ell,\alpha}_b$ is guaranteed only when $\gamma \in r^\nu \cC^{\ell+2,\alpha}_b$,
the a priori hypothesis in this definition that $\gamma \in \cC^{\ell,\alpha}_b$ implies that $\gamma$ must have
some special regularity properties in order that $\cL_t \gamma$ lies in the correct space.  Indeed, there is a regularity theorem 
for conic elliptic operators which states that if $\gamma \in \cD_{\Fr}^{\ell,\alpha}(\cL_t)(\nu)$ for some $\nu > 0$, then 
necessarily $\gamma$ has a finite expansion involving the indicial roots of $\cL_t$ in $[0,\nu)$ plus a remainder
in $r^\nu \calC^{\ell+2, \alpha}_b$, i.e., 
\begin{equation*}
\cD_{\Fr}^{\ell,\alpha}(\cL_t)(\nu) = \left\{ \gamma = \sum_{0 \leq \nu_j < \nu} \gamma_j r^{\nu_j} + \widetilde \gamma : \gamma_j \in 
\cC^\infty(S^1), \ \widetilde \gamma \in r^\nu \cC^{\ell+2,\alpha}_b \right\}.
% \cD_{\Fr}^{\ell,\alpha}(\cL_t)(\nu) = \left\{ \gamma = \sum_{0 \leq \nu_j < \nu} \gamma_j r^{\nu_j} + \widetilde \gamma : \gamma_j \in \R, \ \widetilde \gamma 
% \in r^\nu \cC^{\ell+2,\alpha}_b \right\}.
\label{FHexpansion}
\end{equation*}
The coefficients $\gamma_j$ here are eigenfunctions of the induced operator on $S^1$ with eigenvalue $-\nu_j^2$. 
Since $\gamma$ is assumed in this definition to be bounded, this partial expansion omits the unbounded term $\log r$,
which also corresponds to the indicial root $0$.  Note that we are making the simplifying assumption here that $\nu < 1$; in general, 
when $\nu$ is larger, this partial expansion will also involve terms of the form $r^{\nu_j + i}$, $i \in \mathbb N_0$, so long as $\nu_j  + i < \nu$. 
These extra terms are needed to account for error terms caused by the higher order terms in the Taylor expansion of the coefficients
of $\cL_t$; the guiding principle is that $\cL_t$ must annihilate all terms up to order $\nu-2$. 

\begin{prop}
The graph of $\cL_t$ over $\cD_{\Fr}^{\ell,\alpha}(\cL_t)(\nu)$ is a closed subspace of $\cC^{\ell,\alpha}_b \oplus 
r^{\nu-2} \cC^{\ell,\alpha}_b$, and the induced graph norm is equivalent to 
\[
||\gamma||'_{\cD_{\Fr}^{\ell,\alpha}(\cL_t)(\nu)} = \sum_{0 \leq \nu_j < \nu} \sup |\gamma_j| + 
||\tilde{\gamma}||_{r^{\nu}\cC^{\ell+2, \alpha}_b}.
\]
\label{grnorm}
\end{prop}
\begin{proof} The span $V$ of the functions $\gamma_j(\theta)r^{\nu_j}$ is finite dimensional, and has only
trivial intersection with $r^{\nu}\cC^{\ell+2,\alpha}_b$. Furthermore, by virtue of the a priori estimate
\[
||\gamma||_{r^{\nu}\cC^{\ell+2,\alpha}_b} \leq C ( ||\cL_t \gamma||_{r^{\nu-2}\cC^{\ell,\alpha}_b} + ||\gamma||_{\cC^{\ell,\alpha}_b}),
\]
$r^\nu \cC^{\ell+2,\alpha}_b$ is a closed subspace in the graph norm, so the natural map $V \oplus r^{\nu}\cC^{\ell+2,\alpha}_b \to 
\cD_{\Fr}^{\ell,\alpha}(\cL_t)(\nu)$ is continuous and bijective, hence by the open mapping theorem is a topological isomorphism.  
This shows that $\cD_{\Fr}^{\ell,\alpha}(\cL_t)(\nu)$ is a closed subspace in the graph norm, and furthermore, since the norms
on these two spaces are the norms we wish to compare, they must be equivalent.
\end{proof}

\begin{rem}\label{leadingterm}
Near any point $p \in D$, the leading term $\gamma_0$ in the expansion of any $\gamma \in \cD_{\Fr}^{0,\alpha}(\nu)$ is a constant diagonal Hermitian matrix. 
Indeed, the indicial root $0$ for $\cL_t$ is associated with the $0^{\mathrm{th}}$ eigenvalue of the Laplacian on $S^1$, and occurs only for the indicial equation
on the diagonal part of $\gamma$.  
\end{rem}

\begin{lem}
If $\gamma \in \cD_{\Fr}^{0,\alpha}(\nu)$, then \eqref{quadform} holds whenever $\nu > 0$. 
\end{lem}
\begin{proof}
The boundary term in the integration by parts leading to \eqref{quadform} is 
\[
\lim_{\epsilon \searrow 0}  \int_{r = \epsilon}  \langle \del_r \gamma, \gamma \rangle \, r \de \theta.
\]
The leading term $\gamma_0$ is annihilated by $\del_r$, so the leading term in the pointwise inner product is $O(r^{\mu-1})$, where
$\mu$ is the minimum of $\nu$ and the first positive indicial root of $\cL_t$ for the diagonal operator.   Multiplying by $r \de \theta$, 
this boundary term vanishes since $\mu > 0$.   We note also that both $\langle M_{\Phi_t} \gamma, \gamma \rangle$ and
$|[\Phi_t, \gamma]|^2$ are integrable. Indeed, in the strongly parabolic case, these are each bounded pointwise by $Cr^{2(\alpha_1 - \alpha_2)}$,
while in the weakly parabolic case, the leading part of $\Phi_t$ and its adjoint are both diagonal, hence commute with $\gamma_0$, so
these two terms are bounded by $Cr^{2\mu-2}$ for some $\mu > 0$. 
\end{proof}

\begin{lem}
The map $\cL_t: \cD_{\Fr}^{L^2} \to L^2$ is invertible provided $Z \cup D_s \neq \emptyset$. 
\end{lem}
\begin{rem} The sum of the orders of the zeros and poles of any meromorphic quadratic differential on $C$ equals $4g-4$. 
If there are no zeros, then there can be at most two poles of order $2$, and in fact the only cases where $Z \cup D_s = \emptyset$ 
is when $C = S^2$ and $q = c^2 \de z^2/z^2$, or when $C=T^2$ and $q$ has no zeros or poles. We henceforth exclude these 
two trivial cases. 
\end{rem}
\begin{proof}
If $\gamma$ is in the classical $L^2$ Friedrichs domain, then \eqref{quadform} holds, and hence if $\cL_t \gamma = 0$, then
$d_{A_t} \gamma = 0$ and $[\Phi_t, \gamma] = 0$. The first equality implies that $\gamma$ is parallel, so $|\gamma|^2$ is constant on $C \setminus D$.
If $D_s$ is nonempty, then in a standard holomorphic coordinate near a strongly parabolic point, the equation  $[\Phi, \gamma]=0$ yields 
\[
\gamma=\begin{pmatrix}
0 &u_1\\\bar u_1& 0	
\end{pmatrix}.
\]
Taking into account the indicial root set for this off-diagonal term, we must have $u_1= O(r^{\nu_0})$, where $\nu_0$ is the 
first positive indicial root for the off-diagonal operator.  Hence in this case we obtain that $\gamma \equiv 0$. Similarly, near any point 
of $Z$, $\gamma$ is also off-diagonal and in this case $[\Phi, \gamma]=0$ yields that $u_1 z= r \e^{2 \ellt} \bar u_1$.  A simple 
computation shows that $u_1 \equiv 0$ in the region where $\ellt>0$, hence $\gamma \equiv 0$ once again.
Since $\cL_t$ on this Friedrichs domain is self-adjoint and injective, it is bijective and hence an isomorphism.
\end{proof}

\begin{lem}\label{isom_HFr_domain}
For $\nu > 0$ and $\gamma \in \cD_{\Fr}^{0,\alpha}(\cL_t)(\nu)$, if $\cL_t \gamma = 0$, then 
$d_{A_t}\gamma=0$ and $[\Phi_t, \gamma]=0$.  Hence if $Z \cup D_s \neq \emptyset$, then $\gamma = 0$.
In this case, $\cL_t: \cD_{\Fr}^{0,\alpha}(\cL_t)(\nu)  \to r^{\nu-2} \calC^{0,\alpha}_b$ is an isomorphism.
\end{lem}
\begin{proof}
Since $\gamma = \gamma_0 + \widetilde{\gamma}$, the boundary term $ \int_{r=\epsilon}  \langle \del_r \widetilde{\gamma}, \gamma_0 \rangle \, r \de \theta$
vanishes as $\epsilon \to 0$, so $d_{A_t}\gamma=0$ and $[\Phi_t, \gamma]=0$ as in the previous Lemma.   Thus repeating the same arguments,
except for the trivial case where $Z \cup D_s$ is empty, $\cL_t$ is injective. 

Now fix any $\eta \in r^{\nu-2} \calC^{0,\alpha}_b$. If $\nu > 1$, then $r^{\nu-2} \calC^{0,\alpha}_b \subset L^2$, so there exists some $\gamma \in 
\cD_{\Fr}^{L^2}$ such that $\cL_t \gamma = \eta$.  By parts (iii) and (iv) of Proposition \ref{mainmp}, $\gamma$ decomposes as a finite sum $\sum r^{\nu_j} \gamma_j$,
 with each indicial root $\nu_j \in [-1, \nu)$ and a remainder term $\widetilde{\gamma} \in r^\nu \calC^{2,\alpha}_b$.  However, no term with $r^{\nu_j}$ for $\nu_j < 0$ or $\log r$
can lie in the $L^2$ Friedrichs domain, so the sum is actually only over nonnegative $\nu_j$. Thus $\gamma$ lies in $\cD_{\Fr}^{0,\alpha}(\nu)$. 
By the open mapping theorem, $\cL_t: \cD_{\Fr}^{0,\alpha} \to \calC^{0,\alpha}_b$ is an isomorphism. 

If $0 < \nu \leq 1$, we must argue differently since we cannot obtain the solution $\gamma$ from the $L^2$ theory.  For this we note that, extending
part (i) of Proposition \ref{mainmp}, $\cL_t: r^{-\nu'} \calC^{2,\alpha}_b \to r^{-\nu'-2} \calC^{0,\alpha}_b$ is surjective when $\nu' > 0$. This can be
deduced from a duality statement on weighted $L^2$ spaces and some version of part (iii) of this same Proposition. We refer to \cite{MazzeoWeiss} for
details. In any case, now apply part (ii) of this Proposition to obtain that $\gamma$ is a finite sum of terms $\gamma_j r^{\nu_j}$ with $\nu_j \in [-\nu', \nu)$, 
a term $\widetilde{\gamma}_0 \log r$ since $0 \in (-\nu', \nu)$, and a `remainder' term $\widetilde{\gamma} \in r^{\nu} \calC^{2,\alpha}_b$.  We can take $\nu'$ arbitrarily small, so
the only potentially  problematic term in this expansion is $\widetilde{\gamma}_0 \log r$. 

At this point we have obtained a solution $\gamma$ to $\cL_t \gamma = \eta$ which has the property that $\gamma \sim \widetilde{\gamma}_0(p) \log r + \gamma_0(p) + \ldots$
at each $p \in D$.  We claim that there is another solution to this same equation which is bounded, i.e., does not have these logarithmic terms at any
of the singular points of $q$.  To prove this, we show that there exists an element $\widehat{\gamma}$ in the nullspace of $\cL_t$ which has the same
asymptotics $\widehat{\gamma} \sim \widetilde{\gamma}_0(p) \log r + O(1)$ at each $p \in D$.  Equivalently, we claim that the map
\[
\mathrm{ker}\,(\cL_t) \cap r^{-\nu'} \calC^{2,\alpha}_b \ni \gamma \mapsto ( \widetilde{\gamma}_0(p))_{p \in D} \in \R^{|D|}
\] 
is a bijection.  (Here $|D|$ is the number of points in $D$.) 

To prove this claim, we quote two further facts, both discussed in \cite{MazzeoWeiss}. First, consider the two Fredholm mappings $\cL_t: r^{\pm\nu'} \calC^{2,\alpha}_b 
\to r^{\pm \nu' - 2} \calC^{0,\alpha}_b$. There is a relative index theorem which states that the difference between the indices of these two maps is the
algebraic multiplicity of the indicial root $0$, which in this case equals $2|D|$. (This is because at each $p \in D$ the multiplicity of the indicial root $0$ 
for the one-dimensional diagonal part of $\cL_t$ is $2$, while $0$ is not an indicial root for the nondiagonal part).   Next using that $\cL_t$ is 
injective on positively weighted spaces and surjective on negatively weighted spaces, we see that this difference of indices is in fact equal to
twice the dimension of the nullspace of $\cL_t$ on $r^{-\nu'}\calC^{2,\alpha}_b$.  Finally, the map from this nullspace to the vector
$(\widetilde{\gamma}_0(p))_{p \in D} \in \R^{|D|}$ is injective, since any nullspace element with no log terms lies in $\cD_{\Fr}^{0,\alpha}(\nu')$,
and hence must be trivial.  This proves that the map is bijective, and hence we can find a nullspace element $\widehat{\gamma}$ with
precisely the correct $\log r$ coefficients at each $p \in D$.  This proves the claim.
\end{proof}
	
We have now established the existence of an inverse $\cL_t^{-1}$ both in this H\"older setting, and we denote this by $G_t$. 
It is a conic pseudodifferential operator of order $-2$, and the structure of its Schwartz kernel will play a role below.

\subsection{Uniform mapping properties}\label{sec:linearizationglobal}
In this section we estimate the growth of the norm of $G_t = \cL_t^{-1}: \cC^{0,\alpha}_b \to \cD_{\Fr}^{0,\alpha}$ as $t \to \infty$. 
We begin with $L^2$ and Sobolev estimates.
\begin{prop}\label{prop:uniformL2estimate}
The norm of the inverse $G_t: L^2 \to L^2$ is uniformly bounded as $t \to \infty$. 
\end{prop}
\begin{proof}
Let $\lambda_t$ denote the smallest eigenvalue of $\cL_t$. We have already shown that $\lambda_t > 0$ for every $t$, and this proposition
is equivalent to the claim that there exists $\kappa > 0$ such that $\lambda_t \geq \kappa$ for every $t \geq 1$.   Define
$\cL_t^o = \Delta_{A_t}-i\ast M_{\Phi_t}$.  Since $-i\ast M_{\Phi_t} \geq 0$, it follows that $\cL_t \geq \cL_t^o$ for $t\geq1$, and
thus it suffices to show that the smallest eigenvalue $\lambda_t^o$ of $\cL_t^o$ is bounded below by some $\kappa^o>0$. 

A special role is played in this argument by points in the sets $Z$ and $D_s$ since the structure of $\cL_t$ changes as $t \to \infty$ at these points
(but not at points in $D_w$).  Choose a local holomorphic coordinate $z$ centered at each point $p \in Z \cup D_s$ and a disc $\D_p$
around each such $p$. We do not label these coordinate patches separately, and for simplicity, tacitly assume that $\D_p = \{|z| \leq 3/4\}$. 
Define the family of smooth positive weight functions $\mu_t$ on $C$ such that
\begin{equation*}
\mu_t(z)= \begin{cases}
 	(t^{-\frac{4}{3}}+|z|^2)^{\frac{1}{2}} \quad & \textrm{on each disk}\; \D_p,\; p\in Z,\\
 	(t^{-4} + |z|^2)^{\frac{1}{2}} \quad & \textrm{on each disk}\; \D_p,\; p\in D, \\
         1 \quad & \textrm{on}\ C \setminus \bigcup_{p \in Z \cup D} \D_p.
\end{cases}
\end{equation*}

For each $t \geq 1$, denote by $\psi_t$ the eigenfunction of $\cL_t^o$ with eigenvalue $\lambda_t^o$,  normalized so that 
\begin{equation*}\label{eq:supremrescaledfct}
\sup_{C \setminus D} \mu_t^{\delta}|\psi_t|=1,	
\end{equation*}
where $\delta > 0$ is a fixed constant specified below.  We assume by way of contradiction that $\lambda_t^o \to 0$,
at least for some sequence of values $t_j \to \infty$.  Write $\mu_j = \mu_{t_j}$ and $\psi_j = \psi_{t_j}$, and
for each $j$, choose a point $q_j$ such that $\mu_j^\delta(q_j) |\psi_j(q_j)| = 1$. 

\medskip

\noindent {\bf Case 1.}    Suppose that (at least some subsequence) of the $q_j$ does not converge to $Z \cup D$. 
By the local boundedness of the $\mu_j$, elliptic regularity, and a diagonalization argument, we may choose a 
subsequence of the $\psi_j$ (labeled again as $\psi_j$) which converges in $\calC^\infty$ on any compact 
subset of $C^\times := C \setminus  (Z \cup D)$ to a limit $\psi_\infty$.  Since  $\lambda_j^o \to 0$, 
this limiting function satisfies $\cL_\infty^o \psi_\infty = 0$ and $|\psi_\infty| \leq |z|^{-\delta}$.  Furthermore, 
$\psi_\infty$ is not identically $0$ since it is nonvanishing at the point $\lim q_j$.

As described earlier, $\cL_{\infty}^o$ is a conic differential operator.  By local conic elliptic regularity theory, 
$\psi_\infty$ has a complete asymptotic expansion near each $p \in Z \cup D$ in powers $r^{\nu_i}$ where the $\nu_i$ are 
indicial roots of $\cL_\infty^o$, and possibly also including the term $\log r$. Choosing $\delta$ smaller than the absolute 
value of the first nonzero indicial root at any point of $Z \cup D$, each $\nu_i$ must be nonnegative, but a priori the $\log r$ 
term may still be present.    To rule this out, integrate by parts when $t$ is large but finite to get
\begin{equation*}
\lambda_j^o \|\psi_j\|^2 = \langle \cL_j^o \psi_j, \psi_j \rangle = \int |d_{A_j} \psi_j|^2 + | [\Phi_j, \psi_j]|^2.
\label{ibp}
\end{equation*}
Since $|\psi_j|^2 \leq C \mu_j^{-2\delta}$ is uniformly $L^1$, the left hand side tends to $0$ and we conclude that
$d_{A_\infty} \psi_\infty = 0$, $[\Phi_\infty, \psi_\infty] = 0$. The latter equation implies that $\psi_\infty$ is
a multiple of $\Phi_\infty$, hence purely off-diagonal near $Z \cup D_s$, and hence must vanish at least like $r^{1/2}$.
The first equation implies that $|\psi_\infty|$ is constant, and hence $\psi_\infty \equiv 0$.  This is a contradiction
and hence this case cannot occur. 

\medskip

\noindent {\bf Case 2.}
Next suppose that $q_j$ converges to a point $q_\infty \in Z$, and that $\mu_j^{\delta}|\psi_j|$ converges to zero on any compact
subset of $C^\times$.  The point $q_j$ corresponds to $z_j$ in a fixed holomorphic coordinate $z$ around $q_\infty$, and we must
distinguish two cases, depending on the rate at which $z_j \to 0$.  

To understand this, we first digress and consider the rescaled problem. Near any point of $Z$, write $\varrho = t^{2/3} r$  
(note the difference with the change of variables $\rho = \frac83 t r^{3/2}$ used earlier).   Since $A_t^{\app}$ equals $A_t^{\model}$ 
in this neighborhood, its coefficients depend only on $\varrho$ here.  A brief calculation shows that 
%The Laplacian  $\Delta_{A_t} = r^{-2} \widehat{\Delta_t}$  where $\widehat{\Delta_t}  = -(r\del_r)^2 + 
%(-i \del_\theta + a(t^{2/3} r))^2$ for some Hermitian matrix $a$. Thus $\widehat{\Delta_t} = \widehat{\Delta_\rho}$ depends 
%only on $(\varrho, \theta)$ but not $t$. Since $r^{-2} = t^{4/3} \varrho^{-2}$, we obtain altogether that 
$\Delta_{A_t} = t^{4/3} \Delta_\varrho$, the connection Laplacian associated to $A_1^{\model}$.  Recall that the difference
$\Delta_\varrho - \Delta_{A_\infty^{\mathrm{fid}}}$ decays exponentially as $\varrho \to \infty$. The other term $-i \star t^2 M_{\Phi_t}$ 
behaves similarly: the matrix entries of $\Phi_t$ and $\Phi_t^*$ equal $r^{1/2}$ times functions of $t^{2/3} r = \varrho$, so 
$t^2 M_t =  t^{4/3} \cM_{\varrho}$, where $\cM_\varrho$ is an endomorphism with coefficients depending only on $(\varrho, \theta)$. 
Altogether, in this neighborhood, 
\begin{equation*}\label{defLt}
\cL_t = t^{4/3}(\Delta_\varrho  -i \star \cM_\varrho),\qquad \cL_t^o = t^{4/3} ( \Delta_{\varrho} - i \star t^{-4/3} \cM_{\varrho}).
\end{equation*}
%Here $\Delta_\infty$ is the Laplacian for $A_\infty^{\model}$ and $M_\infty =-2 \pi^{\skew}(i \ast [ (\Phi_\infty^{\model})^* \wedge 
%[ \Phi_\infty^\model \wedge \cdot ]])$, both expressed in terms of $\varrho$.
%Similarly, $\cL_t^o$ transforms to $t^{4/3} \cL_w$, $\cL_w = \Delta_w -i\ast t^{-2} M_w$. 

%The connection Laplacian $\Delta_w$ is an elliptic cone operator on $\C\cup\{\infty\}$ with singularities 
%at $w=0$ and $w=\infty$.

Return now to the argument at hand.  Set $w = t_j^{\frac23} z$ and $w_j = t_j^{\frac23} z_j$. We must distinguish
two cases based on whether or not $w_j$ remains bounded. 

Suppose that $|w_j| \leq C$ for all $j$.  Set $\widetilde{\psi}_j(w) = t^{-4\delta/3} \psi_j( t_j^{-2/3} w)$ so that 
$|\psi_j(z)| \leq \mu_j^{-\delta}$ with equality at $z_j$ is the same as $|\widetilde{\psi}_j(w)| \leq (1 + |w|^2)^{-\delta/2}$ 
with equality at $w_j$.  This function solves
\[
(\Delta_\varrho - i \star t_j^{-4/3} \cM_{\varrho}) \widetilde{\psi}_j(w) = t_j^{-4/3} \lambda_j^o \widetilde{\psi}_j(w).
\]
%or more simply $(\Delta_\varrho - i \star t_j^{-4/3} M_{\varrho}) \widetilde{\varphi}_j(w) = t_j^{-4/3}\lambda_j^o \widetilde{\varphi}_j(w)$.
Taking the limit, we obtain a nontrivial  $\widetilde{\psi}_\infty$ defined on all of $\C$, which decays like $|w|^{-\delta}$ as $w \to \infty$, 
and which solves $\Delta_{\varrho} \widetilde{\psi}_\infty = 0$.     The splitting of $\widetilde{\psi}_\infty$ into its diagonal
and off-diagonal parts is now global on $\C$; we call these parts $\psi_\infty'$ and $\psi_\infty''$, respectively.
The operator induced by $\Delta_{\varrho}$ on $\psi_\infty'$ is simply the scalar Laplacian $\Delta_0$, and hence
by standard conic theory, $\psi_\infty'$ decays like $\varrho^{-1}$.  This is sufficient to justify the integration by
parts $\langle \Delta_{\varrho} \psi_\infty', \psi_\infty' \rangle = |d_{A_\varrho} \psi_\infty'|^2 = 0$.  Hence
$\psi_\infty'$ has has constant norm, but also decays at infinity, so it must vanish identically.   
On the other hand, the induced operator on the off-diagonal part is $\Delta_0 - F_1^o i \del_\theta + 4 (F_1^o)^2$.
Expanding $\psi_\infty''$ into Fourier series, then the $k^{\mathrm{th}}$ Fourier component satisfies 
$ (-r^{-2} (r\del_r)^2 - (i \del_\theta - 2 F_1^o)^2) \psi_{\infty,k}'' = 0$.  As $\varrho \to \infty$, this operator converges 
to $- r^{-2} (r\del_r)^2 - (k - 1/2)^2$, and by standard ODE theory, any bounded solution must decay exponentially.
Hence by the same argument, this term too must vanish identically.

The other possibility is that $\sigma_j := |w_j| \to \infty$.   We now rescale further, letting $\widehat{w} = w/\sigma_j$
(or altogether, $\widehat{w} = t_j^{2/3}\sigma_j^{-1} z$). 
Defining $\widehat{\psi}_j(\widehat{w}) = \sigma_j^{-\delta} \widetilde{\psi}_j( \sigma_j \widehat{w})$,  then 
\[
\sigma_j^{-\delta}|\widetilde{\psi}_j(\sigma_j \widehat{w})| \leq \sigma_j^{-\delta}( 1 + \sigma_j^2 |\widehat{w}|^2)^{-\delta/2} 
\Rightarrow |\widehat{\psi}_j(\widehat{w})| \leq (\sigma_j^{-2} + |\widehat{w}|^2)^{-\delta/2},
\]
with equality at some point $\widehat{w}_j$ with $|\widehat{w}_j| = 1$. 

The limit $\widehat{\psi}_\infty$ satisfies $|\widehat{\psi}_\infty| \leq |w|^{-\delta}$ and $\Delta_{A_\infty} \widehat{\psi}_\infty = 0$.
This operator is conic at both $0$ and $\infty$, and in fact is homogeneous of degree $-2$.  This means
that if we expand $\widehat{\psi}_\infty$ into Fourier series in $\theta$, then each coefficient is a sum of
at most two monomials $r^{\nu_j^\pm}$; the Fourier mode at energy $0$ has coefficient $r^0$ or $\log r$. 
However, none of these terms are bounded by $r^{-\delta}$ at both $0$ and $\infty$,  which is a contradiction.
Hence this case cannot occur.

\medskip

\noindent {\bf Case 3.}
The next scenario is that $q_j \to q_\infty \in D_s \cup D_w$. As before, we distinguish between two cases, depending on 
the rate at which the points $z_j$ tend to $0$.   Near strongly parabolic points, the appropriate scaling factor is $t^2$
so we define $\widetilde{\psi}_j(w) = t_j^{-2\delta} \psi_j( w/t_j)$ which gives $|\widetilde{\psi}_j(w)| \leq (1 + |w|^2)^{-\delta/2}$,
with equality at some point $w_j$.  Note that since $\psi_j \in \cD_{\Fr}^{L^2}$, it is bounded but does not necessarily extend 
smoothly across the origin. 

Assume first that $w_j$ remains bounded. We can repeat the arguments for the previous case almost verbatim and
obtain a nontrivial limit $\widetilde{\psi}_\infty(w)$ on $\C \setminus \{0\}$ which satisfies $\Delta_\varrho \widetilde{\psi}_\infty = 0$,
where now $\Delta_{\varrho}$ is the connection Laplacian associated to the fiducial solution around a strongly parabolic
point if $q_\infty \in D_s$ and the $t$-independent fiducial solution when $q_\infty \in D_w$. In addition 
$|\widetilde{\psi}_\infty(w)| \leq (1 + |w|^2)^{-\delta/2}$ with  equality at some point $w_\infty$.
By conic elliptic theory, $\widetilde{\psi}_{\infty}$ has a complete asymptotic expansion at $w=0$ and another 
as $w \to \infty$.   Since $\widetilde{\psi}_\infty$ is bounded near $w=0$, its expansion there has only nonnegative
exponents and no $\log r$ term. Choosing $0 < \delta < \min\{ \alpha_2 - \alpha_1, 1 - (\alpha_2 - \alpha_1) \}$, 
then $\widetilde{\psi}_\infty$ decays fast enough to justify the usual integration by parts, which then implies
as before that $|\widetilde{\psi}_\infty|$ is constant. It cannot vanish at infinity unless it is identically zero.

If $w_j=t_j^{2}z_j$ is unbounded, then proceeding as in the second part of Case 2, we arrive at a nontrivial solution
$\widetilde{\psi}_\infty$ which satisfies $\Delta_{A_\infty} \widetilde{\psi}_\infty = 0$  and $|\widetilde{\psi}_\infty| \leq 
|\widehat{w}|^{-\delta}$, where $\widehat{w} = w/|w_j|$. This case is ruled out as before.
\end{proof}

We have now proved that the unbounded operator $\cL_t: L^2 \to L^2$ is invertible and has inverse
$G_t = \cL_t^{-1}$ which has norm bounded independently of $t$ for $t \geq 1$.  We now determine the
behavior of the norm of $G_t$ mapping between other spaces. 

First consider $G_t: L^2 \to H^2_b$, which is well-defined since $\cD_{\Fr}^{L^2} \subset H^2_b$. Fix $\eta \in L^2$
and rewrite $\cL_t \gamma = \eta$ as $\cL_t^o \gamma = \eta + (t^2 - 1) i\ast M_{\Phi_t} \gamma$.  

Since $\|\gamma\|_{L^2} \leq C \|\eta\|_{L^2}$
with $C$ independent of $t$, consider the right hand side as a function $\widetilde{\eta}$ with $\|\widetilde{\eta}\|_{L^2}
\leq Ct^2 \|\eta\|_{L^2}$.     The $H^2$ bound for $\gamma$ on any compact set disjoint from $Z \cup D$ follows
from standard elliptic theory and the $L^2$ estimate for $\gamma$.   We consider this estimate near each of the different
types of points of $Z \cup D$ in turn.

The easiest is in fact the estimate near $q \in D_w$ for the simple reason that the operator is $t$-independent there.
Thus $\cL_t^o \gamma = \widetilde{\eta}$ is a fixed conic operator in a neighborhood around such a $q$ and the
estimate follows from standard conic theory, cf.\ \cite{MazzeoWeiss}. 

If $q \in D_s$, then $\cL_t^o$ is an elliptic conic operator for all $t$ including $t = \infty$, but the
indicial root structure changes in the limit as $t \to \infty$.   

It is possible to construct a uniformly $t$-dependent family of parametrices $G_{p,t}$ for $\cL_t^o$ 
in the `calculus with bounds'. This implies that any combination of up to two $b$-derivatives
involving $r\del_r$ and $\del_\theta$ applied to $G_{p,t}$ is bounded on $L^2$ with norm independent
of $t$, which is what we require. However, we can derive this in a more elementary way using the Mellin 
transform. Recall that if $f(r,\theta) \in L^2(r \de r \de \theta)$ is supported in $r \leq 1$, then 
\[
f_M(\zeta, \theta) = \int_0^\infty f(r,\theta) r^{i\zeta} \, \frac{\de r}{r}
\]
is holomorphic in the lower half-plane $\mathrm{Im}\, \zeta < 1$ with $L^2$ norm on each line
$\mathrm{Im}\, \zeta = -\epsilon$ uniformly bounded as $\epsilon \nearrow 0^-$.  If $f \in r^\nu L^2(r \de r \de\theta)$
then $f_M$ is holomorphic in the lower half-plane $\mathrm{Im}\, \zeta < \nu - 1$.  Now write $\cL_t \gamma = \widetilde{\eta}$ as
$r^2 \Delta_{A_t} \gamma = i \star r^2 M_{\Phi_t} \gamma + r^2 \widetilde{\eta}$. The operator on the left is 
$- (r\del_r)^2 - \del_\theta^2$ for the diagonal part and $-(r\del_r)^2 + (i\del_\theta - 2F_t^p(r))^2$ for the off-diagonal part. 

Now, $i\star r^2 M_{\Phi_t} \gamma$ is uniformly in $r^{\delta}L^2(r \de r \de \theta)$ as $t \leq \infty$ for some $\delta > 0$.   For the
diagonal part, pass to the Mellin transform and write the equation as $(\zeta^2 + \del_\theta^2) \gamma_M = F_M$
where $F$ is the sum of the two terms on the right and $F_M$ its Mellin transform.   Thus $F_M$ is holomorphic
in the half-plane $\mathrm{Im}\, \zeta < 1 + \delta$. The resolvent $(\zeta^2 + \del_\theta^2)^{-1}$ has poles at
the points $\I k$, $k \in \mathbb Z$, hence applying it to both sides we see that $\gamma_M$ is meromorphic
in $\mathrm{Im}\, \zeta < 1+\delta$ with poles possibly at $k i$, $k = 1, 0, -1, \ldots$. However,  since $\gamma \in L^2$,
then a priori $\gamma_M$ is holomorphic in $\mathrm{Im}\, \zeta < 1$ with uniform $L^2$ bounds on each horizontal
line; this prohibits any of the possible poles of $\gamma_M$.  The expression $\gamma_M = (\zeta^2 + \del_\theta^2)^{-1} F_M$
shows that in fact $\gamma_M(\zeta, \theta)$ takes values in $H^2(S^1)$ and both $\zeta \gamma_M$ and $\zeta^2\gamma_M$
satisfy uniform $L^2$ estimates on each horizontal line $\mathrm{Im}\, \zeta < 1$. Taking the inverse Mellin transform,
we find that $\gamma \in H^2_b$.

The argument for the off-diagonal part has one additional step. Indeed, we can write the equation as
\[
- ((r\del_r)^2 + \del_\theta^2)\gamma = \I \star r^2 M_{\Phi_t}\gamma  + r^2 \widetilde{\eta} + 4 \I F_t^p \del_\theta \gamma -
4(F_t^p)^2 \gamma.
\]
The right hand side lies in $L^2$ as a function of $r$, uniformly in $t$, with values in $H^{-1}(S^1)$.   Now take the Mellin
transform and apply the resolvent $(\zeta^2 + \del_\theta^2)^{-1}$; this shows that $\gamma_M$, $\zeta \gamma_M$
and $\zeta^2 \gamma_M$ are holomorphic in $\mathrm{Im}\, \zeta < 1$ with values in $H^1(S^1)$.  We can
recycle this information back into the initial equation, so that the right hand side is now $L^2$ in both $r$ and $\theta$
so we deduce that $\gamma \in H^2_b$ uniformly in $t$ as claimed. 

The analysis near points of $Z$ is essentially the same, and we leave the details to the reader.  In summary, we have
proved the
\begin{prop}
The norm of the inverse $\cL_t^{-1}: L^2 \to H^2_b$ is bounded by $C t^2$ for some $C > 0$. 
\end{prop}

We  may now prove bounds for the norm of this inverse acting between $b$-H\"older spaces.   We show first that
$\cL_t^{-1}: \calC^{0,\alpha}_b \to \calC^{0,\alpha}_b$ has norm bounded by $C t^2$.  This is not an optimal estimate,
but is sufficient for our purposes.   The key to this is Sobolev embedding and the scale-invariant nature of the $b$-H\"older 
and $b$-Sobolev norms.    More specifically, using the embedding $H^2 \subset \calC^{0,\alpha}$ on any compact set 
of $C \setminus (Z \cup D)$, it suffices to establish the norm estimate in a neighborhood $\D_p$ of each $p \in Z \cup D$.
Using cutoff functions, we may as well assume that $\gamma$ is supported in such a neighborhood.  Decompose 
$\D_p \setminus \{p\}$ into a countable union of annuli $A_j = \{2^{-j-1} \leq r \leq 2^{-j}\}$. Denote by 
$A_j' = \{2^{-j-2} \leq r \leq 2^{-j+1}\}$ a slight enlargement of $A_j$.  Let $k_j: A_1 \to A_j$ be the dilation
$(r,\theta) \to (2^{-j}r, \theta)$, and let $\{\chi_j\}$ be a partition of unity relative to this cover. Then, by definition
of the $b$-norms
\[
\|\eta\|_{\calC^{0,\alpha}_b}  \cong  \sup_j \|\chi_j \eta\|_{\calC^{0,\alpha}_b}  = \sup_j  \| k_j^*( \chi_j \eta)\|_{\calC^{0,\alpha}(A_1)}
\]
and 
\[
\|\gamma\|_{H^2_b}^2 \cong \sum_{j=0}^\infty \| \chi_j \gamma\|_{H^2_b}^2 = \sum_{j=0}^\infty \|k_j^*(\chi_j \gamma)\|_{H^2(A_1)}^2.
\]
Now using that $H^2(A_1)  \subset \calC^{0,\alpha}(A_1) \subset L^2(A_1)$, we have
\begin{multline*}
\qquad \|\gamma\|^2_{\calC^{0,\alpha}_b} \leq C \sup_j \|\chi_j \gamma\|^2_{\calC^{0,\alpha}_b} \leq C \sup_j  \|\chi_j \gamma\|_{H^2_b}^2   \\ 
\qquad \qquad \qquad \leq C \sum_j \|\chi_j \gamma\|_{H^2_b}^2 \leq C \|\gamma\|_{H^2_b}^2 \leq Ct^2 \|\eta\|_{L^2} \leq Ct^2 \|\eta\|_{\calC^{0,\alpha}_b}. \hfill
\end{multline*} 
The second inequality is the result of dilating by $k_j^*$, applying the ordinary Sobolev embedding bound, then dilating back
by $(k_j^{-1})^*$. This establishes that
\begin{prop}
The norm of the inverse $\cL_t^{-1}: \calC^{0,\alpha}_b \to \calC^{0,\alpha}_b$ is bounded by $C t^2$ for some $C > 0$. 
\end{prop}

Finally, write $\cL_t^o \gamma = \widetilde{\eta} := \eta + i \star t^2 M_{\Phi_t} \gamma$ again, and observe that
$\|\widetilde{\eta}\|_{\calC^{0,\alpha}_b} \leq Ct^2 \|\eta\|_{\calC^{0,\alpha}_b}$.    Localizing to each annulus $A_j$ and
applying ordinary Schauder estimates to the rescalings of $\gamma$ and $\widetilde{\eta}$ there, we obtain
\[
\| \gamma|_{A_j}\|_{\calC^{2,\alpha}_b}  \leq C \left( \|\widetilde{\eta}|_{A_j'}\|_{\calC^{0,\alpha}_b(A_j')} + \|\gamma_{A_j'} \|_{\calC^{0,\alpha}_b(A_j')}\right),
\]
so we now conclude that 
\begin{cor}
The norm of the inverse $\cL_t^{-1}: \calC^{0,\alpha}_b \to \calC^{2,\alpha}_b$ is bounded by $C t^4$ for some $C > 0$. The image of
this mapping is the H\"older Friedrichs domain $\cD_{\Fr}^{0,\alpha}$.
\label{bddest}
\end{cor}

We now come to the main estimate. 
\begin{prop}\label{prop:mainestimate}
If $\nu > 0$ is less than the smallest positive indicial root of $\cL_t$, then the norm of the inverse $\mathcal L_t^{-1}\colon r^{\nu-2}\mathcal \calC^{0,\alpha}_b \to
\cD_{\Fr}^{0,\alpha}(\nu)$ is bounded by $Ct^4$ for some $C>0$. 
\end{prop}
\begin{proof}
Fix any $\eta \in r^{\nu-2}\calC^{0,\alpha}_b$ and write $\gamma = \cL_t^{-1} \eta \in \cD_{\Fr}^{0,\alpha}(\nu)$. 

For each $p \in Z \cup D$, choose a nonnegative smooth cutoff function $\chi_p$ supported in the unit disk $\D_p(1)$ and equaling $1$ in $\D_p(1/2)$,
and define $\gamma_p = \chi_p G_t^{\model} (\chi_p  \eta)$.  Now,  $\cL_t \gamma_p = \eta$ in $\D_p(1/2)$ since $\cL_t G_t^\model = \mathrm{Id}$ in this
neighborhood. Thus if $\zeta = \gamma - \sum_{p \in Z \cup D} \gamma_p$, then $\cL_t \zeta  = \widetilde{\eta}$, where $\widetilde{\eta}$ vanishes near 
each $p \in D$ and $\widetilde{\eta} = \eta$ outside $\cup_p \D_p(1)$.  Clearly $\widetilde{\eta} \in \calC^{0,\alpha}_b$, so by  Corollary \ref{bddest},
$\|\zeta\|_{\calC^{2,\alpha}_b} \leq C t^4 \|\widetilde{\eta}\|_{\calC^{0,\alpha}_b}$. 

It remains to estimate each $\|\gamma_p\|_{\calC^{2,\alpha}_b}$ in terms of $\|\eta_p\|_{r^{\nu-2} \calC^{0,\alpha}_b}$, and then to 
estimate  $\|\widetilde{\eta}\|_{\calC^{0,\alpha}_b}$ in terms of $\sum \|\eta\|_{r^{\nu-2} \calC^{0,\alpha}_b}$. 

First let $p \in Z$. Then $G_t^{\model}(z, \widetilde{z}) = G_\rho( t^{2/3} z, t^{2/3} \widetilde{z})$, see \cite[\textsection 5.1]{MSWWgeometry} for details.  Writing $(\eta_p)_\lambda(z) = 
\eta_p(\lambda z)$, we calculate that 
\begin{multline*}
\left\|\int G_t^{\model}(z, \widetilde{z}) \eta_p(\widetilde{z}) \, \de \widetilde{z} \right\|_{\calC^{2,\alpha}_b} = t^{-4/3}  \left\| \int G_\rho( t^{2/3} z, \widetilde{w}) (\eta_p)_{t^{-2/3}} 
(\widetilde{w})\, \de \widetilde{w}\right\|_{\calC^{2,\alpha}_b} \\ \leq C t^{-4/3} \left\| \int G_\rho (w, \widetilde{w})\eta_p(\widetilde{w}) \right\|_{\calC^{2,\alpha}_b} \leq C t^{-4/3} \|\eta_p\|_{\calC^{0,\alpha}_b}
\end{multline*}
by dilation invariance of the $b$-H\"older norms.

Next, if $p \in D_s$ then $G_t^{\model}(z, \widetilde{z}) = G_\rho( t^2 z, t^2 \widetilde{z})\, \de \widetilde{z}$. Rescaling in both $z$ and $\widetilde{z}$, we obtain
\begin{multline*}
\| \int G_t^{\model}(z, \widetilde{z}) \eta_p(\widetilde{z})\, \de \widetilde{z}\|_{\calC^{2,\alpha}_b}= t^{-4}  \left\| \int G_\rho( t^2 z, \widetilde{w}) (\eta_p)_{t^{-2}} 
(\widetilde{w})\, \de \widetilde{w}\right\|_{\calC^{2,\alpha}_b} \\ \leq C t^{-4} \left\| \int G_\rho (w, \widetilde{w})\eta_p(\widetilde{w}) \right\|_{\calC^{2,\alpha}_b}.
%\leq C t^{-4} \|\eta_p\|_{\calC^{0,\alpha}_b}
\end{multline*}
Now write  
\[
\int G_\rho (w, \widetilde w) \eta_p(\widetilde w) = \int G_\rho (w, \widetilde{w}) |\widetilde w|^{\nu-2}  ( |\widetilde w|^{2-\nu} \eta_p(\widetilde w)).
\]
We must estimate the $\calC^{2,\alpha}_b$ norm of this integral in terms of the $\calC^{0,\alpha}_b$ norm of $|\widetilde w|^{2-\nu} \eta_p(\widetilde{w})$.
At least for $|w| \leq 10$, this follows from the general mapping properties of $G_\rho$ as a $b$-pseudodifferential operator of
order $-2$, proved in \cite{Ma91}.

  The hypothesis that needs to be checked is that the integration makes sense.
In other words, in the region where $\widetilde{w} \to 0$ and $w \neq 0$, the Schwartz kernel $G_\rho(w, \widetilde{w}) |\widetilde w|^{\nu-2}$
is asymptotic to $\nu - 2$ since the smallest nonnegative indicial root of $L_\rho$ is zero, and this is integrable with respect to
the area form on $\R^2$.  On the other hand, as $w \to \infty$, the sup norm, and \emph{a fortiori} the $\calC^{2,\alpha}_b$ norm
of this integral is uniformly bounded.  This proves that
\[
\| \int G_t^{\model}(z, \widetilde{z}) \eta_p(\widetilde{z})\, \de \widetilde{z}\|_{\calC^{2,\alpha}_b} \leq C t^{-4} \| \eta_p \|_{r^{\nu-2} \calC^{0,\alpha}_b}.
\]

Finally, if $p \in D_w$, then there is no longer any scaling, and we must simply check that the $\calC^{2,\alpha}_b$ norm
of $\int G \eta_p$ is bounded by the $r^{\nu-2}\calC^{0,\alpha}_b$ norm of $\eta_p$.   For this, note that the expansions of $G$ at
each of the boundaries involve only the nonnegative indicial roots of $\cL_t$, and hence the integral is well-defined
and one can use the same pseudodifferential boundedness theorem.

The final estimate is the trivial observation that 
\[
\|\widetilde \eta\|_{\calC^{0,\alpha}_b} \leq \|\eta\|_{\calC^{0,\alpha}_b} + \sum \|\eta_p\|_{r^{\nu-2} \calC^{0,\alpha}_b}  \leq C \|\eta\|_{\calC^{0,\alpha}_b}.
\]
This completes the proof.
\end{proof}

\section{Correcting to an exact solution}
We now come to the final step, to modify the approximate solution $h_t^\app$ to obtain an exact solution. As we have
explained earlier, this amounts to solving the equation $\calF_t(\gamma)=0$ for some $\gamma \in
\cD_{\Fr}^{0,\alpha}(\delta)$.  

Recall that if $g$ is any Hermitian gauge transformation and $A$ any connection, then 
\[
F_{A^g} = g^{-1}\left( F_{A} + \overline{\del}_{A} ( g^2  \del_A g^{-2}) \right) g.
\]
Thus, keeping in mind that $A_t$ and $\Phi_t$ refer to the to the background approximate solution fields, the equation $\calF_t(\gamma) = 0$ is equivalent to 
\[
F_{A_t} + \overline{\del}_{A_t} (\e^{2\gamma} \del_{A_t} \e^{-2\gamma}) + \e^\gamma \left[ \e^{-\gamma} \Phi_t \e^\gamma \wedge
e^\gamma \Phi_t^* \e^{-\gamma} \right] \e^{-\gamma}=0.
\]
Recall also the general formula
\[
\nabla \e^\gamma =  \nabla \gamma  E(\gamma) \e^\gamma = \e^\gamma E(\gamma) \nabla \gamma,
\]
where 
\[
E(\gamma) =  \frac{ \exp( \mathrm{ad}\, \gamma)  - 1}{\mathrm{ad}\,\gamma}.
\]
Putting all of these identities together, the equation becomes 
\begin{equation*}
\calF_t(\gamma) := F_{A_t} - 2 \overline{\del}_{A_t} ( E(-2\gamma) \del_{A_t} \gamma) + 
\left( \Phi_t \e^{2\gamma} \wedge \Phi_t^* \e^{-2\gamma} + \e^{2\gamma} \Phi_t^* \wedge \e^{-2\gamma} \Phi_t\right) = 0.
\end{equation*}

\begin{prop}
The map $\calF_t\colon \cD_{\Fr}^{0,\alpha}(\delta) \longrightarrow r^{\delta-2} \calC^{0,\alpha}_b$ is smooth.
\end{prop}
\begin{proof}
The expression for $\calF_t$ above is a polynomial in $\overline{\del}_{A_t} \del_{A_t} \gamma$, $\overline{\del}_{A_t} \gamma$ and
$\del_{A_t} \gamma$ with coefficients $E(-2\gamma)$, $DE|_{-2\gamma}$ and $e^{\pm 2\gamma}$.  Schematically, 
\begin{equation}
\calF_t(\gamma) =  B_1(\gamma) \overline{\del}_{A_t}\del_{A_t} \gamma +  B_2(\gamma) \del_{A_t} \gamma \overline{\del}_{A_t} \gamma 
+ B_3(\gamma)(\Phi_t, \Phi_t^*),
\label{strcalF}
\end{equation}
where $B_1$ and $B_2$ are smooth functions of $\gamma$ and $B_3(\gamma)$ is a bilinear form in $\Phi_t$ and $\Phi_t^*$
which coefficients smooth in $\gamma$. Furthermore, each $\gamma \mapsto B_j(\gamma)$ is smooth with respect to the norm 
on $\cD_{\Fr}^{0,\alpha}(\delta)$.      Now recall that $\gamma = \gamma_0 + \widetilde{\gamma}$ where 
$\gamma_0$ is independent of $r$ and $\theta$ and diagonal, and $\widetilde{\gamma} \in r^\delta \calC^{2,\alpha}_b$.  This implies that
$\nabla_{A_t} \gamma = \nabla_{A_t} \widetilde{\gamma} + [A_t, \widetilde{\gamma}]$ since $[A_t, \gamma_0] = 0$ because
$A_t$ is also diagonal.  Hence $\nabla_{A_t} \gamma \in r^{\delta-1} \calC^{1,\alpha}_b$. 
Finally, the products of terms involving entries of $\Phi_t$ and $\Phi_t^*$ are polyhomogeneous and lie in 
$r^{2(\alpha_1 - \alpha_2)} \calC^{0,\alpha}_b \subset r^{\delta-2} \calC^{0,\alpha}_b$.

The structure of the equation as written above clearly implies that each summand lies in $r^{\delta-2} \calC^{0,\alpha}_b$, and
it is also clear that $\calF_t$ is a smooth map. 
\end{proof}

We now expand $\calF_t$ in a Taylor series around $\gamma = 0$, writing
\[
\calF_t(\gamma) = \eta_t + \cL_t \gamma + Q_t(\gamma).
\]
The nonlinear error term $Q_t$ is smooth in $\gamma$ and by inspecting \eqref{strcalF}, we see that
\[
\|Q_t(\gamma)\|_{r^{\delta-2}\calC^{0,\alpha}_b} \leq C \|\gamma\|_{\calC^{2,\alpha}_b}\left(
\|\gamma\|_{\calC^{2,\alpha}_b} + \|\nabla \gamma\|_{r^{\delta-1} \calC^{1,\alpha}_b} + 
\|\nabla^2 \gamma\|_{r^{\delta-2}\calC^{0,\alpha}_b}\right). 
\]
The covariant derivative $\nabla$ here is any reference connection which is smooth across the points of $Z$ and $D$. Furthermore, 
\begin{multline*}
\|Q_t(\gamma_1) - Q_t(\gamma_2)\|_{r^{\delta-2}\calC^{0,\alpha}_b} \leq  \\ C
\|\gamma\|_{\calC^{2,\alpha}_b} \left( \|\gamma_1 - \gamma_2\|_{\calC^{2,\alpha}_b} + 
\|\nabla (\gamma_1 - \gamma_2)\|_{r^{\delta-1} \calC^{1,\alpha}_b} + \|\nabla^2 (\gamma_1 - \gamma_2)\|_{r^{\delta-2}\calC^{0,\alpha}_b}\right). 
\end{multline*}
The constants $C$ which appear in the two preceding inequalities are independent of $t$.  

\begin{thm}\label{thm:perturb} 
There exists $t_0 > 1$ such that for every $t \geq t_0$ there exists $\gamma_t \in \cD_{\Fr}^{0,\alpha}(\delta)$ which
satisfies $\calF_t(\gamma_t) = 0$.  Furthermore, this $\gamma_t$ is unique amongst Hermitian endomorphisms of small norm
and satisfies an estimate $\|\gamma_t\|_{\calC^{2,\alpha}_b} + ||\cL_t \gamma_t||_{r^{\delta-2}\cC^{0,\alpha}_b} 
\leq \e^{-(\mu/2) t}$ for some $\mu > 0$.  Hence 
by Proposition~\ref{grnorm}, $\gamma_t = \sum \gamma_{j,t}(\theta) r^{\nu_j} + \tilde{\gamma}_{t}$ where 
$\tilde{\gamma}_{t} \in r^{\delta}\cC^{2,\alpha}_b$, and 
\[
\sum_j \sup |\gamma_{j,t}| + ||\tilde{\gamma}_t||_{r^\delta \cC^{2,\alpha}_b} \leq C e^{-(\mu/2)t}.
\]
\end{thm}
\begin{proof}
Write the equation to be solved as 
\[
\gamma_t = - G_t ( \eta_t + Q_t(\gamma_t))
\]
where $G_t$ is the inverse described above.  Now recall that $\|\eta_t\| \leq C \e^{-\mu t}$ for some $\mu > 0$.
The norm estimate of $G_t$ and a straightforward contraction mapping argument, using the estimates above for $Q_t$, show that
$\|\gamma_t\|_{\calC^{2,\alpha}_b} \leq C e^{-(\mu/2)t}$. To estimate $||\cL_t \gamma_t||_{r^{\delta-2}\cC^{0,\alpha}_b}$ similarly, 
write $\cL_t \gamma_t = -\eta_t - Q_t(\gamma_t)$ and use the estimates for these quantities noted above. 
\end{proof}

\section{Asymptotic geometry of \texorpdfstring{$\cM'$}{M}} \label{sec:asymptoticgeo}
In this section we establish that for parabolic Higgs bundle moduli spaces, the difference between $g_{L^2}$ and 
$g_{\semif}$ decays exponentially in $t$ along rays in the portion of the moduli space $\cM'$ away from the preimage 
of the discriminant locus.  There are some slight differences between the strongly and weakly parabolic settings,
and we comment on these as we go along.  The statement of the result in the strongly parabolic setting is directly
in line with what we have discussed before: 

\begin{thm}\label{thm:L2vssemiflatwithoutt}
Let $\cM$ be a moduli space of strongly parabolic $SL(2,\C)$ Higgs bundles and $(\delbar_E, \varphi) \in \cM'$ any 
stable Higgs bundle. Suppose that $\dot{\psi}=(\dot{\eta}, \dot{\varphi})$ is an infinitesimal variation 
of the Higgs bundle moduli space. Consider the family of tangent vectors $\dot{\psi}_t=(\dot{\eta}, t \dot{\varphi}) 
\in T_{(\delbar_E, t \varphi)} \cM$ over the ray $(\delbar_E, t \varphi, h_t)$. Then as $t \rightarrow \infty$, 
\begin{equation*} 
\norm{(\dot{\eta}, t \dot{\varphi}, \dot{\nu}_t)}^2_{g_{L^2}}- \norm{(\dot{\eta}, t \dot{\varphi}, \dot{\nu}_t)}^2_{g_{\semif}}=  
O(\e^{-\varepsilon t})
\end{equation*}
for some $\varepsilon > 0$.
\end{thm}

\noindent This extends the analogous result for the Hitchin moduli space over $\SU(n)$, with holomorphic (rather than 
meromorphic) data, in \cite{Fredricksonasygeo} , building on \cite{MSWWgeometry} and \cite{DumasNeitzke}. The proof 
here is a relatively straightforward  adaption of the one in \cite{Fredricksonasygeo}.

The feature here which does not extend to the weakly parabolic case is the fact that $\R^+$ (and in fact
$\CC^*$) acts on the moduli space only in the strongly parabolic case.   The correct way to extend this theory
then is to consider the {\it family} of moduli spaces $\cM_t$, for each $t > 0$, defined  as follows.
%As will be discussed in \S\ref{sec:GMN}, we need a $t$-rescaled $L^2$-metric precisely because $\cM$ is not preserved by the circle action $(\cE, \varphi) \mapsto (\cE, \zeta \varphi)$.
Given $[(\cE, \varphi)] \in \mathcal{M}_{\mathrm{Higgs}}$, define the $t$-Hitchin moduli space
\begin{equation*}
\cM_t = \{[(\cE, \varphi, h_t)]: [(\cE, \varphi)] \in \mathcal{M}_{\mathrm{Higgs}}\}
\end{equation*}
where $h_t$ solves the rescaled Hitchin equations
\begin{equation}
F_{D(\delbar_E, h_t)} + t^2[\varphi, \varphi^{*_{h_t}}]=0
\label{trescaled}
\end{equation}
and is adapted to the parabolic structure.  

The space $\cM_t$ has a natural $L^2$-metric $g_{L^2,t}$, defined in the obvious way, which is hyperk\"ahler as before.
The goal now is to compare $g_{L^2,t}$ to the $t$-rescaled semiflat metric $g_{\semif,t}$. 
\begin{thm}\label{thm:L2vssemiflat}
Fix $\cM_{\mathrm{Higgs}}$ a moduli space of (weakly or strongly) parabolic $SL(2,\C)$ Higgs bundles.
Let $(\delbar_E, \varphi) \in \cM'_{\mathrm{Higgs}}$ be any stable Higgs bundle and $(\dot{\eta}, \dot{\varphi})$  an infinitesimal variation 
of the Higgs bundle moduli space. Identifying $(\dot{\eta}, \dot{\varphi})$ with its image in $T_{(\delbar_E, \varphi, h_t)} \cM_t$,
as $t \rightarrow \infty$, 
\begin{equation*} 
\norm{(\dot{\eta},  \dot{\varphi}, \dot{\nu}_t)}^2_{g_{L^2}(\cM_t)}- \norm{(\dot{\eta},  \dot{\varphi}, \dot{\nu}_\infty)}^2_{g_{\semif,t}}=  O(\e^{-\varepsilon t})
\end{equation*}
for some $\varepsilon > 0$.
\end{thm}

The proofs involve comparing the $L^2$-metrics on the three moduli spaces
\begin{multline*}
\cM'=\{(\delbar_E, \varphi, h)\}/\sim, \ \  \cM'_\infty =\{(\delbar_E, \varphi, h_\infty)\}/\sim, \ \ \mbox{and}\ \ \cM'_\app =\{(\delbar_E, \varphi, h^\app)\}/\sim 
\end{multline*}
(or their $t$-rescaled versions).  
We begin by reviewing the gauged infinitesimal deformations and expressions for $g_{L^2}$. In  \S\ref{sec:semif} we
show that the semiflat metric is naturally identified with the $L^2$-metric on $\cM'_\infty$, the regular locus of the 
moduli space of limiting configurations.  In \S\ref{sec:localDN} and \S\ref{sec:proofasygeo} we establish that 
$g_{\semif} - g_{\app}$ and $g_{\app} - g_{L^2}$ decay exponentially; the proofs of Theorems \ref{thm:L2vssemiflatwithoutt} 
and \ref{thm:L2vssemiflat} follow from this.

\bigskip

\subsection{The tangent space to \texorpdfstring{$\boldsymbol{\cM_{\mathrm{Higgs}}}$}{M}} \label{sec:Fredformulation}
We first describe elements of $T\cM_{\mathrm{Higgs}}$ using the differential-geometric setup as in \S \ref{sec:parabolicHiggs}. 
Fix the smooth complex  vector bundle $E$ over $C$ and complete flag $\mathcal{F}(p)$ and weight vector $\vec{\alpha}(p)$ 
at each $p \in D$. In terms of these, $\cM$ is the quotient space $\cM = \mathcal{H}_{\vec{\alpha}}/\mathcal{G}_\C$, 
where $\mathcal{H}_{\vec{\alpha}} \subset \mathcal{H}$ is the space of $\vec\alpha$-stable Higgs bundle structures 
$(\bar\partial_E, \varphi)$ on $(E,\mathcal{F})$ and $\mathcal{G}_\C$ is the group of complex gauge transformations 
preserving $\mathcal{F}(p)$ at each $p \in D$. Recall that $\bar\partial_E$ is a nonsingular holomorphic structure on $E$, $\varphi$ 
has only simple poles at $D$ and $\bar\partial_E \varphi=0$ on $C \setminus D$.   The residue of $\varphi$ at $p$ is nilpotent 
with respect to the flag in the strongly parabolic case, but only preserves it in the weakly parabolic case with eigenvalues 
$\sigma(p)$ and $-\sigma(p)$. In either setting we  say that it is compatible with the flag.

Differentiating the equality $ (\delbar_E)_\eps \varphi_\eps = 0$ for a family of parabolic Higgs bundle 
\begin{equation*} \label{eq:defhol}
(\delbar_E)_\eps = \delbar_E + \eps \dot \eta + O(\varepsilon^2), \quad \varphi_\eps 
= \varphi + \eps \dot \varphi + O(\varepsilon^2),
\end{equation*}
yields 
\begin{equation*}
\delbar_E \dot \varphi + [\dot \eta, \varphi]=0
\end{equation*}
on $C \setminus D$.  Since the singular structure of $(\delbar_E)_\eps$ is fixed, $\dot \eta \in \Omega^{0,1}(\End_0E)$ is nonsingular
and represents an infinitesimal deformation of the holomorphic structure; on the other hand, $\dot\varphi \in \Omega^{1,0}(\End_0 C)$ 
can have at most simple poles at $p \in D$, with residues compatible with $\mathcal F(p)$. More precisely, choose a local family 
of $(\delbar_E)_\eps$-holomorphic frames and a holomorphic coordinate $z$ on the disk $\D_p$, and write $\varphi_\varepsilon 
= f_\varepsilon \, \de z$; then there exists a smooth $\mathfrak{sl}(2, \C)$-valued function $\dot \gamma$ 
which preserves the flag and a meromorphic $\dot f$ with simple pole at each $p$ and residue compatible with $\mathcal F(p)$ such that
\begin{equation}\label{eq:regdef}
\dot\eta = \bar\partial \dot \gamma \qquad \text{and} \qquad \dot \varphi = \dot f \de z +  [\varphi, \dot \gamma] .
\end{equation}

Infinitesimal gauge transformations act on the space of infinitesimal deformations by
\begin{eqnarray} \label{eq:infinitesimal}
(\dot \eta, \dot \varphi) \mapsto ( \dot{\eta},\dot{\varphi})+(\delbar_E \dot{\gamma},[\varphi, \dot{\gamma}]).
\end{eqnarray}
%where $\dot \gamma$ is a smooth infinitesimal complex gauge transformation $\dot \gamma$ which is required, in the parabolic setting,
%to respect the flag structure. 
We then set 
\begin{equation*}
 T_{(\delbar_E, \varphi)} \cM'_{\mathrm{Higgs}} = \{\, (\dot{\eta}, \dot{\varphi}) \; | \; \delbar_E \dot{\varphi} + [\dot{\eta}, \varphi]=0 \, \}/\sim,
\end{equation*}
where $\ \sim\ $ is the equivalence \eqref{eq:infinitesimal}. 

%There is a convenient local model for representatives of tangent vectors near $p \in D$. 
\begin{lem}\label{lem:higgssimpleshape}
Suppose that $q$ has a simple pole at $p \in C$, and that $z$ is a holomorphic coordinate such that $q=z^{-1} \de z^2
=-\det \varphi$ near $p$. Fix a stable Higgs bundle $(\delbar_E, \varphi) \in \pi^{-1}(q)$ and an infinitesimal
Higgs bundle deformation $[(\dot{\eta}_1, \dot{\varphi}_1)] \in T_{(\delbar_E, \varphi)} \cM'_{\mathrm{Higgs}}$.  In the flat metric $|\de z|^2$ 
and the holomorphic gauge where 
\begin{equation*} \label{eq:simplezero}
\delbar_E =\delbar, \qquad \varphi = \begin{pmatrix}0 & 1 \\ z^{-1} & 0 \end{pmatrix} \de z,
 \end{equation*} 
there is a unique $(\dot \eta_2, \dot \varphi_2) \sim (\dot \eta_1, \dot \varphi_1)$ on this disk such that
\begin{equation*} \label{eq:defshapelem}
\dot{\eta}_2 =0, \qquad  \dot{\varphi}_2 = \begin{pmatrix} 0 &  0 \\ \frac{\dot{P}}{z} & 0 \end{pmatrix}\de z, \qquad \mbox{
$\dot P$ holomorphic}.
\end{equation*}

If $q$ has a double pole, and the holomorphic gauge is chosen so that
\begin{equation*} 
\label{eq:doublepole}
\delbar_E =\delbar, \qquad \varphi = \frac{\sigma}{z}\begin{pmatrix}1 & 0 \\ 0 & -1 \end{pmatrix} \de z,
\end{equation*} 
then there exists a unique infinitesimal Higgs bundle deformation $(\dot \eta_2, \dot \varphi_2)$ in the equivalence class 
such that 
\begin{equation*} \label{eq:defshapeweakly}
\dot{\eta}_2 =0, \qquad  \dot{\varphi}_2 =  \frac{\dot{P}}{z}\begin{pmatrix} 1& 0 \\ 0 & - 1\end{pmatrix} \de z, 
\quad \dot P\ \mbox{holomorphic.} 
\end{equation*}
\end{lem}
\begin{proof} 
By \eqref{eq:regdef} and the Poincar\'e lemma, there exists a representative such that $\dot \eta_2=0$. 
We then adjust by a constant matrix so that the flag at each $p$ is preserved.  In the strongly parabolic case, the residue 
of the Higgs field is nilpotent, hence at this stage 
\[
\dot \varphi = \begin{pmatrix} \dot P_1 & \dot P_2 \\ \frac{\dot P_3}{z} & -\dot P_1\end{pmatrix} \de z,
\]
where the $\dot P_j$ are holomorphic.  Now set
\[
\dot\gamma = \begin{pmatrix}  \frac{\dot P_2}{2} & z(\dot{P}_1 + f) \\  f & -\frac{\dot P_2}{2}
\end{pmatrix}, 
\]
where $f$ is an arbitrary holomorphic function; this too respects the flag structure at $z=0$. Then 
\[ 
\dot{\varphi}  + [\varphi, \dot\gamma] = \begin{pmatrix} 0 & 0  \\  \frac{\dot{P}_2 + \dot{P}_3}{z} & 0 \end{pmatrix} \de z; 
\]
setting $\dot{P} = \dot{P}_2 + \dot{P}_3$, this has the desired form.

In the weakly parabolic setting, the residue is not nilpotent, so at first we only have 
\begin{equation*}
\dot{\varphi} = \begin{pmatrix} \frac{\dot{P}_1}{z} & \dot{P}_2\\
\frac{\dot{P}_3}{z} & -\frac{\dot{P}_1}{z}
\end{pmatrix},
 \end{equation*}
where $\dot{P}_1, \dot{P}_2, \dot{P}_3$ are all holomorphic. Taking
\begin{equation*}
 \dot{\gamma} = \begin{pmatrix} f & - \frac{\dot{P}_2 z}{2 \sigma} \\ \frac{\dot{P}_3}{2 \sigma}& -f
\end{pmatrix},
\end{equation*}
which respects the flag structure at $z=0$, we obtain
\begin{equation*}
 \dot{\varphi}= \dot{\varphi}_1 +[\varphi, \dot{\gamma}]=\frac{\dot{P}_1}{z}\begin{pmatrix} 1& 0 \\ 0 & - 1\end{pmatrix} \de z
\end{equation*}
as claimed, with $\dot P=\dot P_1$. 
\end{proof}

\subsection{\texorpdfstring{Hyperk\"ahler}{Hyperkahler} metric}
The usual definition of $g_{L^2}$ involves the unitary formulation of the Hitchin equations, cf.\ 
\textsection \ref{sec:background}.  However, following \cite{Fredricksonasygeo}, it is also possible
to define this metric in terms of Higgs bundle deformations  $(\dot{\eta}, \dot{\varphi})$ and 
one-parameter deformations of Hermitian metrics expressed in terms of an $\mathfrak{sl}(E)$-valued section $\dot{\nu}$ as
\begin{equation*}
 h_\eps(v,w) = h(\e^{\eps \dot{\nu}} v, \e^{\eps \dot{\nu}} w).
\end{equation*}
Note that $h_\eps(v,w) = h(v, w) + \eps \left( h( \dot{\nu} v, w) + h(v, \dot{\nu}w) \right) + O(\eps^2)$. 
In the $h_0$-unitary formulation of Hitchin's equations, the corresponding deformation is 
\begin{equation*} \label{eq:defun}
 \dot A^{0,1}
=H^{1/2}  \left( \dot{\eta} - \delbar_E \dot{\nu} \right) H^{-1/2}, \qquad \dot{\Phi} = H^{1/2} \left(\dot{\varphi} + [\dot{\nu}, \varphi] \right)  H^{-1/2},
\end{equation*}
where  $H$ is an $\End E$-valued $h_0$-Hermitian section such that $h(v,w)=h_0(Hv,w)$.
By \cite[Proposition 2.2]{Fredricksonasygeo}, the linearized Hitchin equation and Coulomb gauge condition are satisfied if
and only if
\begin{equation}\label{eq:ingaugetriple}
 \del_E^h \delbar_E \dot{\nu} - \del_E^h \dot{\eta} - \left[\varphi^{*_h}, \dot{\varphi} + [\dot{\nu}, \varphi]\right]=0,
\end{equation}
or equivalently,
\[
\mathcal P \dot \nu := \partial^h_E\bar\partial_E\dot \nu - \left[\varphi^{*_h}, [\dot \nu, \varphi]\right] 
= \partial^h_E \dot \eta + [\varphi^{*_h}, \dot \varphi]. 
\]
Here the Coulomb gauge condition demands the pair $(\dot A, \Dot \Phi)$ to be $L^2$-orthogonal to the unitary gauge orbit at $(A,\Phi)$, i.e.\ $L^2$-orthogonal to all pairs $(d_A\dot\gamma, [\Phi,\dot\gamma])$ for infinitesimal unitary gauge transformations $\dot\gamma$. 

In the compact case it is easy to see that if $(\bar\partial_E, \varphi)$ is stable, then for a given Higgs bundle deformation 
$(\dot \eta, \dot \varphi)$ there is a unique $\dot{\nu}$ solving \eqref{eq:ingaugetriple}. Indeed, the index of
$\mathcal P$ is $0$, so we need only prove that this operator is injective, which can be checked by taking the inner product 
with $\dot \nu$, integrating by parts and using that a stable Higgs bundle is simple, cf.\ \cite[Corollary 2.5]{Fredricksonasygeo}.  
Having determined $\dot\nu$ in this way, we then define the $L^2$-metric on the Hitchin moduli space by 
\begin{equation*} \label{eq:gL2triple}
\begin{split}
\norm{(\dot{\eta}, \dot{\varphi}, \dot{\nu})}_{g_{L^2}}^2 & = 2\int_C \left( \left| \dot{\eta} - \delbar_E \dot{\nu}\right|^2_h  +
\left|\dot{\varphi} +[\dot{\nu}, \varphi]\right|^2_h\right) \\
& = 2\int_C \IP{ \dot{\eta} - \delbar_E \dot{\nu},  \dot{\eta}}_h  +
 \IP{\dot{\varphi} +[\dot{\nu}, \varphi], \dot{\varphi}}_h;
\end{split}
\end{equation*}
the second equality follows from \eqref{eq:ingaugetriple} and an integration by parts.  
As explained in \cite[Proposition 2.6]{Fredricksonasygeo}, this corresponds to the previous definition as follows: 
start with the usual formula for $g_{L^2}$ in terms of pairs $(\dot{A}^{0,1}, \dot{\Phi})$ in Coulomb gauge and calculate: 
\begin{equation}
\label{eq:gL2tripleIBP}
\begin{split}
 \norm{(\dot{A}^{0,1}, \dot{\Phi})}_{g_{L^2}} &= 2 \int_C \left| \dot{A}^{0,1} \right|^2_{h_0} + \left|\dot{\Phi}\right|^2_{h_0}\\
&= 2 \int_C \left| H^{1/2}  \left( \dot{\eta} - \delbar_E \dot{\nu} \right) H^{-1/2} \right|^2_{h_0} + \left|H^{1/2} \left(\dot{\varphi} + [\dot{\nu}, \varphi] \right)  H^{-1/2}\right|^2_{h_0}\\
&= 2\int_C \left| \dot{\eta} - \delbar_E \dot{\nu}\right|^2_h  +
 \left|\dot{\varphi} +[\dot{\nu}, \varphi]\right|^2_h \\
& =   \norm{(\dot{\eta}, \dot{\varphi}, \dot{\nu})}_{g_{L^2}}^2
   %= 2\int_C \IP{ \dot{\eta} - \delbar_E \dot{\nu},  \dot{\eta}}_h  +  \IP{\dot{\varphi} +[\dot{\nu}, \varphi], \dot{\varphi}}_h.
\end{split}
\end{equation}

In the parabolic setting, $\varphi$, $\dot \varphi$ and $h$ have singularities, so we must discuss the solvability 
of \eqref{eq:ingaugetriple} in suitable function spaces. Regularity of $\dot \nu$ is needed to justify the 
integration by parts in \eqref{eq:gL2tripleIBP}. To that end, %rewrite \eqref{eq:ingaugetriple} as
first recall the K\"ahler identities, cf.\ \cite[\S 5.1]{MSWW14}, 
\[
2\partial^h_E\bar\partial_E\dot \nu = F_{A_h} +i \star \Delta_{A_h}.
\]
Now, $\mathcal P$ is a conic operator and acts on weighted $b$-H{\"o}lder spaces as before.  Its indicial roots
are the same as those for the operator associated to the flat model metric $h_{\vec{\alpha}}$, i.e.,
to $2\partial^h_E\bar\partial_E\dot \nu = i\star \Delta_{A_{h_{\vec{\alpha}}}}$. We compute the other term in $\mathcal P$.
Write $\dot \nu$ in the holomorphic frame $e_1, e_2$ and in the unitary frame
$\widetilde e_1 =  e_1/|z|^{\alpha_1}$, $\widetilde e_2 = e_2/|z|^{\alpha_2}$  are
\[
\begin{pmatrix} a & b \\ c & -a \end{pmatrix}, \ \ \mbox{and}\ \ \ 
\begin{pmatrix} 
\widetilde a & \widetilde b \\ \widetilde c & - \widetilde a 
\end{pmatrix}
=
\begin{pmatrix} 
a & |z|^{\alpha_1-\alpha_2} b \\ |z|^{\alpha_2 - \alpha_1}c & - a
\end{pmatrix},
\]
respectively. The expression for $\varphi =\begin{pmatrix}  0 &1 \\ \frac{1}{z} & 0	  \end{pmatrix}\de z$ in
holomorphic frame becomes 
\[
\varphi = \begin{pmatrix}
 0 & |z|^{\alpha_1 - \alpha_2} \\
 \frac{|z|^{\alpha_2 - \alpha_1}}{z} & 0 	
 \end{pmatrix}
\de z
\quad \text{and} \quad
\varphi^{*_{h_{\vec{\alpha}}}}=\begin{pmatrix}
	0 & \frac{|z|^{\alpha_2 - \alpha_1}}{\bar z} \\
	|z|^{\alpha_1 - \alpha_2} & 0
\end{pmatrix}
\de \bar z
\]
in unitary frame, and hence 
\[
\left[\varphi^{*_{h_{\vec{\alpha}}}}, [\dot \nu, \varphi]\right] = \begin{pmatrix}
 	-2 \widetilde a (|z|^{2(\alpha_2 - \alpha_1-1)} + |z|^{2(\alpha_1-\alpha_2)}) & 2 \frac{\widetilde c}{\bar z} - 2 \widetilde b |z|^{2(\alpha_2-\alpha_1 -1)}\\
2 \frac{\widetilde b}{z}  -2 \widetilde c |z|^{2(\alpha_1-\alpha_2)} & 2 \widetilde a (|z|^{2(\alpha_2 - \alpha_1-1)} + |z|^{2(\alpha_1-\alpha_2)}) 
 \end{pmatrix}
 \de \bar z \wedge \de z.
\]
The coefficients in this matrix are all bounded by $|z|^{-2+\varepsilon}$, hence do not contribute to the indicial roots of 
$\mathcal P$.  Thus the set of indicial roots of $\mathcal P$ is the same as the set for $\Delta_{A_{h_{\vec{\alpha}}}}$,
which we computed earlier as
\[
\Gamma(\Delta_{A_{h_{\vec{\alpha}}}})=\Z\cup \{\pm(\ell+\alpha_1-\alpha_2) \mid \ell\in \Z\}.
\]
The integers correspond to the action on the diagonal components and $\{\pm(\ell+\alpha_1-\alpha_2)\}$ corresponds to 
the off-diagonal part. 

In the weakly parabolic setting, the expression $\varphi=\frac{1}{z}\begin{pmatrix} \sigma &0 \\ 0 & -\sigma \end{pmatrix} \de z $
leads to the equality
\[
\left[\varphi^{*_{h_{\vec{\alpha}}}}, [\dot \nu, \varphi]\right] = \frac{1}{|z|^2}\begin{pmatrix}
 	0 &  - 4|\sigma|^2\widetilde b\\
- 4|\sigma|^2\widetilde c  & 0  
 \end{pmatrix}
 \de \bar z \wedge \de z,
\]
in unitary frame;  this is now of order $|z|^{-2}$. The indicial root set for the off-diagonal part changes to
\[
\left\{\pm\left((\ell+\alpha_1-\alpha_2)^2 + 16 |\sigma|^2\right)^{1/2} \mid \ell\in \Z\right\}.
\]
The key point is that there is a gap around $0$ for the off-diagonal part in both the strongly and weakly parabolic cases.

We now consider the issue of solvability. 
\begin{prop}\label{prop:existenceuniquenessmetricdeform}
Given any infinitesimal Higgs bundle deformation $(\dot \eta, \dot \varphi)$ of the parabolic Higgs bundle 
$(\bar\partial_E, \varphi)$, there is a unique $\dot \nu \in \cD_{\Fr}^{0,\alpha}(\mathcal P)(\mu)$, where $\mu >0$ is
sufficiently small, which solves 
\[
\del_E^h \delbar_E \dot{\nu} - \del_E^h \dot{\eta} - \left[\varphi^{*_h}, \dot{\varphi} + [\dot{\nu}, \varphi]\right]=0.
\]
\end{prop}
\begin{proof}
If $\dot \nu \in \cD_{\Fr}^{0,\alpha}(\mu)$ and $\mathcal P \dot \nu = 0$, then 
\[
0 = \int_C \left\langle \partial_E^h \bar\partial_E \dot \nu, \dot \nu \right\rangle_h - \left\langle \bigl[\varphi^{*_h}, [\dot \nu, \varphi]\bigr], \dot\nu \right\rangle_h
=\int_C \de \left\langle \bar\partial_E \dot\nu, \dot \nu \right\rangle_h + \int_C \left|\bar\partial_E \dot\nu\right|_h^2 + \left|[\dot\nu,\varphi]\right|_h^2.
\]
By Stokes' theorem,
\[
\int_C \de \left\langle \bar\partial_E \dot\nu, \dot \nu \right\rangle_h = 
\lim_{\delta \to 0} \int_{\partial C_\delta} \left\langle \bar\partial_E \dot \nu, \dot \nu \right \rangle_h,
\]
where $C_\delta = C \setminus \sqcup_{p \in D} \D_p(\delta)$.  In any one of these disks we have
\[
\dot \nu = \begin{pmatrix} \widetilde a & \widetilde b \\ \widetilde c & - \widetilde a \end{pmatrix}
\quad \text{and} \quad
\bar\partial_E = \bar\partial - \frac{1}{2} \begin{pmatrix} \alpha_1 & 0\\ 0 & \alpha_2 \end{pmatrix} \frac{\de \bar z}{\bar z}
\]
in the unitary frame $(\widetilde e_1, \widetilde e_2)$, and so
\[
\left\langle \bar\partial_E \dot \nu, \dot \nu \right \rangle_h = \left\langle \begin{pmatrix} \partial_{\bar z} \widetilde a & \partial_{\bar z} \widetilde b 
- \frac{1}{2\bar z} (\alpha_1 - \alpha_2) \widetilde b\\ \partial_{\bar z} \widetilde c -\frac{1}{2\bar z}(\alpha_2 - \alpha_1) \widetilde c& 
-\partial_{\bar z} \widetilde a \end{pmatrix}, \begin{pmatrix} \widetilde a & \widetilde b \\ \widetilde c  & - \widetilde a 
\end{pmatrix} \right\rangle \de \bar z. 
\]
Now, $\widetilde a = \widetilde a_0 + \widetilde a_1$ where $\widetilde a_0$ is constant and $\widetilde a_1, \widetilde b, \widetilde c = O(r^{\mu})$, 
so $\int_{\partial C_\delta} \left\langle \bar\partial_E \dot \nu, \dot \nu \right \rangle_h = O(\delta^\mu)$. This shows that 
$\bar\partial_E \dot \nu = [\dot \nu, \varphi]=0$.    In particular, the entries of 
\[
\dot \nu =  \begin{pmatrix}  a & b \\ c & -a \end{pmatrix} = 
\begin{pmatrix}
\widetilde a & |z|^{\alpha_2-\alpha_1}\widetilde b\\
|z|^{\alpha_1-\alpha_2} \widetilde c & -\widetilde a
 \end{pmatrix}.
\]
in the holomorphic frame $(e_1, e_2)$ are holomorphic functions.   Now, this equality and the boundedness of
$\widetilde{a}$, $\widetilde{b}$, $\widetilde{c}$ show that $a$ is bounded, $|c| \leq C|z|^{-1+\varepsilon}$ and $|b| \leq C |z|^\varepsilon$,
hence $|c| = O(1)$ and $|b| = O(|z|)$, so in particular $\dot \nu$ must extend holomorphically across 
the puncture and preserve the flag. The stable parabolic Higgs bundle 
$(\bar\partial_E, \varphi)$ is simple, therefore $\dot \nu$ must be a constant multiple of the identity, 
and since $\mathrm{tr}\, \dot \nu = 0$, we see finally that $\dot \nu=0$. As in the proof of Lemma \ref{isom_HFr_domain} this
proves that 
\[
\mathcal{P}: \cD_{\Fr}^{0,\alpha}(\mathcal{P})(\mu)  \to r^{\mu-2} \calC^{0,\alpha}_b
\]
is an isomorphism. Finally, for a Higgs bundle deformation $(\dot\eta, \dot\varphi)$ we compute using 
Lemma \ref{lem:higgssimpleshape} that 
\[
\partial^h_E \dot \eta + [\varphi^{*_h}, \dot \varphi]  \in r^{\mu-2} \calC^{0,\alpha}_b
\]
for some small $\mu > 0$.  Altogether we have shown that there exists a unique $\dot \nu \in \cD_{\Fr}^{0,\alpha}(\mathcal{P})(\mu)$ 
solving $\mathcal P\dot \nu  = \partial^h_E \dot \eta + [\varphi^{*_h}, \dot \varphi]$.
\end{proof}

Now suppose that $\dot\nu \in \cD_{\Fr}^{0,\alpha}(\mathcal P)(\mu)$ solves $\del_E^h \delbar_E \dot{\nu} - 
\del_E^h \dot{\eta} - \left[\varphi^{*_h}, \dot{\varphi} + [\dot{\nu}, \varphi]\right]=0$.  Following
the proof of \cite[Proposition 2.6]{Fredricksonasygeo}, we justify the equalities in \eqref{eq:gL2tripleIBP} by
checking that 
\[
\lim_{\delta \to 0} \int_{\partial C_\delta} \left\langle \dot \eta-\bar\partial_E \dot \nu, \dot \nu \right \rangle_h =0
\]
However, $\dot \eta$ is bounded and $|\bar \partial_E \dot \nu| \leq C |z|^{-1 + \varepsilon}$, so this follows as above.

\subsubsection*{The \texorpdfstring{$L^2$}{L2}-metrics on \texorpdfstring{$\cM'_\infty$}{the limiting} and \texorpdfstring{$\cM'_{\app}$}{approximate moduli spaces}}
The $L^2$-metrics on $\cM'_\infty$ and $\cM'_{\app}$ are defined similarly.  The metric deformation $\dot{\nu}_\infty$ of 
$h_\infty$ satisfies the decoupled equations
\begin{equation} \label{eq:ingaugetripleinf}
 \del_E^{h_\infty} \delbar_E \dot{\nu}_\infty - \del_E^{h_\infty} \dot{\eta} =0, \qquad  \left[\varphi^{*_{h_\infty}}, \dot{\varphi} 
+ [\dot{\nu}_\infty, \varphi]\right]=0, 
\end{equation}
so the $L^2$-metric on $\cM'_\infty$ is 
\begin{equation*} \label{eq:L2Minf}
 \norm{(\dot{\eta}, \dot{\varphi}, \dot{\nu}_\infty)}_{g_{\infty}}^2=2
\int_C \IP{ \dot{\eta} - \delbar_E \dot{\nu}_\infty,  \dot{\eta}}_{h_\infty}  +
 \IP{\dot{\varphi} +[\dot{\nu}_\infty, \varphi], \dot{\varphi}}_{h_\infty}.
\end{equation*}
 
Likewise, the metric deformation $\dot{\nu}^\app$ of $h^{\app}$ satisfies
\begin{equation*} \label{eq:7app}
\del_E^{h^\app} \delbar_E \dot{\nu}^\app - \del_E^{h^\app} \dot{\eta} - \left[\varphi^{*_{h^\app}}, \dot{\varphi} 
+ [\dot{\nu}^\app, \varphi]\right]=0,
\end{equation*}
so the $L^2$-metric on $\cM'_{\app}$ is  
\begin{equation*} \label{eq:gapp}
 \norm{(\dot{\eta}, \dot{\varphi}, \dot{\nu}^\app)}_{g_{\app}}^2=2
 \int_C \IP{ \dot{\eta} - \delbar_E \dot{\nu}^\app,  \dot{\eta}}_{h^\app}  +
 \IP{\dot{\varphi} +[\dot{\nu}^{\app}, \varphi], \dot{\varphi}}_{h^\app}.
\end{equation*}

The integration by parts in both cases is justified by a similar argument to what was done above.

\subsubsection*{The \texorpdfstring{$t$}{t}-rescaled metric \texorpdfstring{$g_{L^2,t}$}{}} For moduli spaces of strongly parabolic Higgs bundles, we can state the main theorem purely in terms of the metric $g_{L^2}$ defined above. More generally, for moduli spaces of weakly parabolic Higgs bundles, we must define a $t$-rescaled $L^2$-metric $g_{L^2,t}$. As will be discussed in \S\ref{sec:GMN}, we need a $t$-rescaled $L^2$-metric precisely because $\cM_{\mathrm{Higgs}}$ is not preserved by $\C^\times_{\zeta}$-action $(\cE, \varphi) \mapsto (\cE, \zeta \varphi)$.

Given $[(\cE, \varphi)] \in \mathcal{M}_{\mathrm{Higgs}}$, 
define the $t$-Hitchin moduli space
\begin{equation*}
 \cM_t = \{[(\cE, \varphi, h_t)]: [(\cE, \varphi)] \in \mathcal{M}_{\mathrm{Higgs}}\}
\end{equation*}
where $h_t$ solves the $t$-rescaled Hitchin's equations
\begin{equation*}
 F_{D(\delbar_E, h_t)} + t^2[\varphi, \varphi^{*_{h_t}}]=0
\end{equation*}
and is adapted to the parabolic structure.

The non-abelian Hodge correspondence gives a diffeomorphism
\begin{eqnarray*}
 \mathrm{NAHC}_t: \cM_{\mathrm{Higgs}} &\rightarrow& \cM_t\\ \nonumber
 [(\cE, \varphi)] &\mapsto& [(\cE, \varphi, h_t)].
\end{eqnarray*}
Given a Higgs bundle deformation $[(\dot{\eta}, \dot{\varphi})] \in T_{[(\cE, \varphi)]} \mathcal{M}_{\mathrm{Higgs}}$, there
is a corresponding map of tangent spaces
\begin{eqnarray*}
 \de \mathrm{NAHC}_t: T\cM_{\mathrm{Higgs}} &\rightarrow& T\cM_t\\ \nonumber
 [(\dot{\eta}, \dot{\varphi})] &\mapsto& [(\dot{\eta}, \dot{\varphi}, \dot{\nu}_t)],
\end{eqnarray*}
where $\dot{\nu}_t$ is determined by
\begin{equation}\label{eq:ingaugetriplet}
 \del_E^{h_t} \delbar_E \dot{\nu}_t - \del_E^{h_t} \dot{\eta} - t^2 \left[\varphi^{*_{h_t}}, \dot{\varphi} + [\dot{\nu}_t, \varphi]\right]=0.
\end{equation}
Given $[(\dot{\eta}, \dot{\varphi}, \dot{\nu}_t)] \in T_{[(\delbar_E, \varphi, h_t)]} \cM_t$,
the $L^2$-metric on the moduli space $\cM_t$ then takes the form
\begin{equation}\label{eq:trescaledL2}
 \norm{(\dot{\eta}, \dot{\varphi}, \dot{\nu}_t)}_{g_{L^2,t}} = 2 \int_C 
 \IP{\dot{\eta} - \delbar_E \dot{\nu}_t, \dot{\eta}}_{h_t} + t^2\IP{\dot{\varphi}
 +[\dot{\nu}_t, \varphi], \dot{\varphi}}_{h_t}.
\end{equation}

\begin{rem} When $\cM_{\mathrm{Higgs}}$ is preserved by the $\C^\times_\zeta$-action $(\cE, \varphi) \mapsto (\cE, \zeta \varphi)$, the $g_{L^2,t}$ norm of $[(\dot{\eta}, \dot{\varphi}, \dot{\nu}_t)] \in T_{[(\delbar_E, \varphi, h_t)]} \cM_t$ equals 
the usual $g_{L^2}$ norm of  $[(\dot{\eta}, t\dot{\varphi}, \dot{\nu}_t)] \in T_{[(\delbar_E, t\varphi, h_t)]} \cM$.
\end{rem}

\subsubsection*{The \texorpdfstring{$t$}{t}-rescaled \texorpdfstring{$L^2$}{L2}-metrics on \texorpdfstring{$\cM'_\infty$}{the limiting} and \texorpdfstring{$\cM'_{\app}$}{approximate moduli spaces}}  
We similarly define $t$-rescaled $L^2$-metrics on $\cM'_\infty$ and $\cM'_{\app}$. The metric deformation 
$\dot{\nu}_\infty$ of $h_\infty$ is independent of $t$ and still satisfies the decoupled equations
\begin{equation*} 
 \del_E^{h_\infty} \delbar_E \dot{\nu}_\infty - \del_E^{h_\infty} \dot{\eta} =0, \qquad  \left[\varphi^{*_{h_\infty}}, \dot{\varphi} 
+ [\dot{\nu}_\infty, \varphi]\right]=0.
\end{equation*}
In particular, it solves
\begin{equation*} 
 \del_E^{h_\infty} \delbar_E \dot{\nu}_\infty - \del_E^{h_\infty} \dot{\eta} +t^2 \left[\varphi^{*_{h_\infty}}, \dot{\varphi} 
+ [\dot{\nu}_\infty, \varphi]\right]=0.
\end{equation*}
for any value of $t$. The $t$-rescaled $L^2$-metric on $\cM'_\infty$ is 
\begin{equation*} \label{eq:trescaledL2Minf}
 \norm{(\dot{\eta}, \dot{\varphi}, \dot{\nu}_\infty)}_{g_{\infty, t}}^2=2
\int_C \IP{ \dot{\eta} - \delbar_E \dot{\nu}_\infty,  \dot{\eta}}_{h_\infty}  +
t^2 \IP{\dot{\varphi} +[\dot{\nu}_\infty, \varphi], \dot{\varphi}}_{h_\infty}.
\end{equation*}
 
Likewise, the metric deformation $\dot{\nu}_t^\app$ of $h^{\app}_t$ satisfies
\begin{equation*} \label{eq:trescaled7app}
\del_E^{h^\app_t} \delbar_E \dot{\nu}^\app_t - \del_E^{h^\app_t} \dot{\eta} - t^2 \left[\varphi^{*_{h^\app_t}}, \dot{\varphi} 
+ [\dot{\nu}^\app_t, \varphi]\right]=0,
\end{equation*}
so the $t$-rescaled $L^2$-metric on $\cM'_{\app}$ is  
\begin{equation*} \label{eq:trescaledgapp}
 \norm{(\dot{\eta}, \dot{\varphi}, \dot{\nu}^\app_t)}_{g_{\app,t}}^2=2
 \int_C \IP{ \dot{\eta} - \delbar_E \dot{\nu}^\app_t,  \dot{\eta}}_{h^\app_t}  +
 t^2\IP{\dot{\varphi} +[\dot{\nu}_t^{\app}, \varphi], \dot{\varphi}}_{h^\app_t}.
\end{equation*}

The justifications for all these expressions follow by arguments similar to what was done above.

\subsection{The semiflat metric \texorpdfstring{on $\cM'_\infty$}{}}\label{sec:semif}
The portion of the moduli space $\cM'$ not lying over the discriminant locus has, in addition to the $L^2$-metric, another 
natural hyperk\"ahler metric $g_{\semif}$, called the semiflat metric. This is an artifact of the complex integrable system 
structure \cite[Theorem 3.8]{FreedSK}.  For smooth rank $2$ Higgs data,  \cite[Proposition 3.7, Proposition 3.11, Lemma 3.12]{MSWWgeometry} proves that the this semiflat metric actually equals the $L^2$-metric on the moduli space of limiting 
configurations $\cM'_\infty$.  A different proof of this is given in \cite{Fredricksonasygeo}, which holds also in the $\SU(n)$ 
case with smooth data. That uses the description of limiting configurations as pairs consisting of a spectral curve $\Sigma$
and a line bundle $\cL$ over $\Sigma$ with parabolic structure at the ramification points of the covering $\Sigma \to C$.
We recall this latter perspective and adapt it to the present setting.

The first point is to parametrize variations of limiting configurations.  Such a variation consists of a family of embeddings
of a fixed smooth curve $\Sigma_{\mathrm{top}}$ into $\mathrm{Tot}(K_C)$ and a variation of the associated holomorphic
line bundle. The same complex line bundle $L \to \Sigma_{\mathrm{top}}$ underlies all nearby Higgs bundles, so we can 
view spectral data as pair $(\delbar_L, \tau)$ where $\tau$ is the tautological eigenvalue of $\pi^* \varphi$ which
gives the embedding of $\Sigma_{\mathrm{top}}$ and $\delbar_L$ is a holomorphic structure on $L$. This gives an isomorphism 
\begin{eqnarray*} \label{eq:dualitysf}
 %\mathrm{Ab}: 
 T_{(\delbar_E, \varphi)} \cM' &\rightarrow&  H^0(\Sigma_b, K_{\Sigma_b})^*_{\odd}
 \oplus H^0(\Sigma_b, K_{\Sigma_b})_{\odd}\\ \nonumber
(\dot{\eta}, \dot{\varphi}) &\mapsto&(\dot{\xi}, \dot{\tau}),
\end{eqnarray*}
corresponding to the family $\delbar_L + \eps \dot\xi$ and $\tau + \eps \dot \tau$. Here $H^0(\Sigma_b, K_{\Sigma_b})_{\odd}$ is the subspace of sections which are anti-invariant under the involution of $\Sigma_b$.  Variations of the holomorphic structure
on $L$ alone are called vertical variations and involve only $\dot \xi$, while horizontal variations involve only $\dot \tau$. 
The limiting Hermitian metric 
$h_\infty$ on $E \rightarrow C$ is the `orthogonal push-forward' of the unique Hermitian-Einstein metric $h_{\cL}$ on 
$\cL \rightarrow \Sigma$, cf.\ \textsection \ref{sec:limitingconfigurations}  and \cite{FredricksonSLn}.  Given a deformation 
$(\dot{\xi}, \dot{\tau})$ of the spectral data, the associated deformation $\dot{\nu}_L$ of the Hermitian metric $h_L$ is
determined by the equation
\begin{equation}\label{eq:harmonic}
 \del_L^{h_L} \delbar_L \dot{\nu}_L - \del^{h_L}_L \dot{\xi} =0.
\end{equation}

Due to the obvious homogeneities of the data, there is a one-parameter family of semiflat metrics:
\begin{prop} \cite[Proposition 2.14]{Fredricksonasygeo}\label{prop:sfcharacterization} 
For each $t > 0$ there is a semiflat metric $g_{\semif, t}$ (at `scale $t$') characterized by the following three properties: 
\begin{enumerate}
\item On horizontal deformations, the semiflat metric is $ t^2 \int_{\Sigma_b} 2|\dot \tau|^2$;
 \item on vertical deformations, the semiflat metric is $ \int_{\Sigma_b} 2 | \dot{\xi} -\delbar_L \dot{\nu}_L|^2$;
 \item horizontal and vertical deformations are orthogonal.
 \end{enumerate}
\end{prop}
\begin{rem} We use a slightly different convention than in \cite{Fredricksonasygeo} and \cite{FreedSK} where 
horizontal and vertical deformations are weighted by $t$ and $t^{-1}$, respectively.
% \begin{itemize}
% \item On horizontal deformations, the semiflat metric is $ t^{-1}\int_{\Sigma_b} 2|\dot \tau|^2$.
%  \item On vertical deformations, the semiflat metric is $ t \int_{\Sigma_b} 2 | \dot{\xi} -\delbar_L \dot{\nu}_L|^2.$
% \end{itemize}
The convention adopted here is more useful in \S\ref{sec:GMN}.
\end{rem}

The semiflat metric $g_{\semif, 1}$ is called \emph{the} semiflat metric, and denoted simply $g_{\semif}$. 
\begin{rem}
The characterization of $g_{\semif}$ in \cite{Fredricksonasygeo} and here is independent of the Hitchin section, but instead
is phrased in terms of the horizontal subspaces with respect to the Gauss-Manin connection. 
\end{rem}

% We now prove that $g_{\semif}$ can also be characterized as the $L^2$-metric on the moduli space of limiting configurations.  
% This was proved in the smooth setting in \cite{MSWWgeometry} , but the proof here is a straightforward adaption of 
% \cite[Theorem 2.15]{Fredricksonasygeo}.

\begin{cor} \label{thm:semiflatisL2}
The semiflat metric $g_{\semif,t}$ is the natural $t$-rescaled $L^2$-metric on the moduli space of limiting configurations $\cM'_\infty$, 
for deformations in formal Coulomb gauge, and moreover, 
\begin{equation}\label{eq:gsf}
\norm{(\dot{\eta}, \dot{\varphi}, \dot{\nu}_{\infty})}^2 _{g_{\semif,t}} =2 \int_C \IP{\dot{\eta} - \delbar_E \dot{\nu}_{\infty},  
\dot{\eta}}_{h_\infty}  +  t^2 \IP{ \dot{\varphi} +[\dot{\nu}_{\infty}, \varphi], \dot{\varphi}}_{h_\infty}.
\end{equation}
\end{cor}

\begin{proof} The corresponding fact in the setting without poles was proved in \cite{MSWWgeometry}; the proof here is a small 
modification of \cite[Theorem 2.15]{Fredricksonasygeo}. The $t$-rescaled $L^2$-metric on $\cM_\infty'$ equals
\begin{equation} \label{eq:s1}
\begin{split}
\norm{(\dot{\eta}, \dot{\varphi}, \dot{\nu}_{\infty})}^2_{g_{\infty,t}}  = \lim_{\delta \to 0} & \left( 2  
\int_{C_\delta} \IP{\dot{\eta} - \delbar_E \dot{\nu}_{1\infty},  \dot{\eta}}_{h_\infty}  + t^2 \IP{ \dot{\varphi} +[\dot{\nu}_{\infty}, \varphi],
\dot{\varphi}}_{h_\infty}\right.\\ 
 & \qquad \quad -\left. 2 \mathrm{Re} \int_{\del C_\delta} \IP{\dot \eta- \delbar_E \dot{\nu}_{\infty}, \dot{\nu}_{\infty}}_{h_\infty} \right),
\end{split}
\end{equation}
where $C_\delta =C-\cup_{p \in Z \cup D} \D_p(\delta)$. Since $h_\infty$ is singular at $Z \cup D$, $\dot{\nu}_\infty$ may also be 
singular at this same collection of points.  We show that the boundary terms vanish.
% 
% 
% To do this, we pullback and compute the integral on the spectral cover.
%  Let $(\dot \tau, \dot{\xi}, \dot{\nu}_{L})$  be the corresponding deformation of the spectral data.
%  Then, because everything diagonalizes on $\Sigma_b$, given any $1$-cycle $\gamma$ in $C$ with lift $\widetilde{\gamma}$ in $\Sigma_b$,
% \begin{eqnarray}
%  \int_{\gamma} \IP{\dot\eta- \delbar_E \dot{\nu}_{\infty}, \dot{\nu}_{\infty}}_{h_\infty}&=&
%  \frac{1}{n} \int_{\widetilde{\gamma}} \IP{\pi^*  \dot \eta- \pi^*\delbar_E \pi^*\dot{\nu}_{\infty}, \pi^*\dot{\nu}_{\infty}}_{\pi^*h_\infty}\\ \nonumber 
%  &=&
%      \int_{\widetilde{\gamma}} \IP{\dot\xi - \delbar_L \dot{\nu}_{L},  \dot\nu_{L} }_{h_L}.
% \end{eqnarray}

In the setting without poles, if $p \in Z$, the monodromy of the associated abelian Chern connection is fixed to be
$(-1)^{\mathrm{ord}(p)}$, hence its infinitesimal variation vanishes. A short argument \cite{Fredricksonasygeo} shows
that the boundary terms vanish.  

Similarly, if $p \in D$, the monodromy is determined by the fixed parabolic weights, hence 
its infinitesimal variation also vanishes. Indeed, the $h_L$-unitary deformation of the Chern connection is 
$\dot{D}=\dot{\xi}-\delbar_L \dot{\nu}_L - (\dot{\xi} -\delbar_L \dot{\nu}_L)^{*_{h_L}}$ (with apologies for
the different use of $D$ here), and the variation
of monodromy is the integral of $\dot D$ around a small loop surrounding $p$.  It follows from \eqref{eq:ingaugetripleinf}
that $\dot D$ is a harmonic $1$-form.  The fact that the boundary integrals vanish for each loop implies that the corresponding
cohomology class in $H^1(S^1, \R)$ vanishes, thus $\dot D \equiv 0$ and so $\int_{\del C_\delta} \IP{\dot{\xi}-\delbar_L 
\dot{\nu}_L, \dot{\nu}_L}=0$ as well. This proves that the $L^2$-metric is as in  \eqref{eq:gsf}.
% 
% Consequently, by a similar argument as in the proof of Lemma \ref{lem:horizontal},  $\int_{\widetilde{\gamma}} \dot{\xi} - \delbar_L \dot{\nu}_{L} + (\dot{\xi} - \delbar_L \dot{\nu}_{L})^* =0$---for all variations, and not just horizontal variations.  Since this integrand is a harmonic representative for $H^1(\bD^\times, \R)$, 
% the space of $1$-forms on the punctured local disk $\bD^\times$ around $\widetilde{p}$ by \eqref{eq:dogs}, the integrand vanishes, and $\dot{\xi} - \delbar_L \dot{\nu}_{L}=0$.
% Consequently, the boundary term in \eqref{eq:s1} vanishes. 

% \begin{equation*}
% 	\norm{(\dot{\eta}, \dot{\varphi}, \dot{\nu}_{\infty})}^2 _{g_{L^2}(\cM'_\infty)}
% =2
%      \int_C \IP{\dot{\eta} - \delbar_E \dot{\nu}_{\infty},  \dot{\eta}}_{h_\infty}  +
%  \IP{ \dot{\varphi} +[\dot{\nu}_{\infty}, \varphi], \dot{\varphi}}_{h_\infty},
% \end{equation*}
% which appears in 

We now pull back the expression for $g_{\infty, t}$ in \eqref{eq:s1} to the spectral cover. Because everything 
diagonalizes on $\Sigma_b$, 
\begin{align*}
&\IP{(\dot{\eta}_1, \dot{\varphi}_1, \dot{\nu}_{1, \infty}) \;, \; (\dot{\eta}_2, \dot{\varphi}_2, \dot{\nu}_{2, \infty})}_{g_{\infty, t}}\\ 
&=2 \, \mathrm{Re}      \int_C \IP{\dot{\eta}_1 - \delbar_E \dot{\nu}_{1, \infty},  \dot{\eta}_2 }_{h_\infty}  +
 t^2\IP{ \dot{\varphi}_1 +[\dot{\nu}_{1, \infty}, \varphi], \dot{\varphi}_2}_{h_\infty}\\ 
&=  \mathrm{Re}
     \int_{\Sigma_b} \IP{\pi^* \dot{\eta}_1 - \pi^*\delbar_E \pi^*\dot{\nu}_{1, \infty},  \pi^*\dot{\eta}_2 }_{\pi^*h_\infty}  
+  t^2 \IP{\pi^* \dot{\varphi}_1 +[\pi^*\dot{\nu}_{1, \infty}, \pi^*\varphi], \pi^*\dot{\varphi}_2}_{\pi^*h_\infty}\\ 
&= 2  \, \mathrm{Re}
     \int_{\Sigma_b} \IP{\dot\xi_1 - \delbar_L \dot{\nu}_{1,L},  \dot\xi_2 }_{h_L} + t^2
 \IP{\dot\tau_1 +[\dot{\nu}_{1, L}, \lambda], \dot\tau_2}_{h_L}\\  %\label{eq:usefulcharacterization}
&= 2 \,   \mathrm{Re}
     \int_{\Sigma_b} \IP{\dot\xi_1 - \delbar_L \dot{\nu}_{1,L},  \dot\xi_2 - \delbar_L \dot{\nu}_{2,L} } + t^2
 \IP{\dot\tau_1, \dot\tau_2}
\end{align*}
The last line uses integration by parts and \eqref{eq:harmonic}.
 
This expression shows that the natural $t$-rescaled $L^2$-metric $g_{\infty, t}$ on $\cM_\infty$ satisfies the three properties of the semiflat metric $g_{\semif, t}$ in the statement of the Proposition
since $\dot \xi_1 - \delbar_L \dot{\nu}_{1, L}=0$ for the horizontal deformations and $\dot\tau_2=0$
for the vertical deformations. This proves the result. 
\end{proof}

\subsection{Comparing \texorpdfstring{$g_{\semif,t}$ and $g_{\app,t}$}{the semiflat and approximate metrics}}
\label{sec:localDN}
Extending to the parabolic setting  an argument due to Dumas and Neitzke \cite{DumasNeitzke} and Fredrickson 
\cite{Fredricksonasygeo}, we now show that $g_{\app,t} - g_{\semif,t}$  decays exponentially in $t$.
Suppose that $p \in D_s$ and fix $(\cE, \varphi) \in \cM'$, 
so $q=-\det \varphi=z^{-1} \de z^2$ for some holomorphic coordinate $z$. Fix the usual flat 
(non-singular) metric on $\D_p$, and a holomorphic frame for which
 \begin{equation*} \label{eq:simplepole}
  \delbar_E =\delbar, \qquad \varphi = \begin{pmatrix}0 & 1 \\ z^{-1} & 0 \end{pmatrix} \de z,
  \qquad h_\infty = Q \, |z|\begin{pmatrix} |z|^{-1/2} &0 \\ 0& |z|^{1/2} \end{pmatrix}.
 \end{equation*}
The harmonic metric $h_t$ is not necessarily diagonal in this frame, but as in \eqref{eq:htapprox}, we define
the approximate harmonic metric
\begin{equation*}
h_t^\app = Q \, |z|\begin{pmatrix} |z|^{-1/2} \e^{\mt(|z|) \chi(|z|)} &0 \\0 & |z|^{1/2} \e^{-\mt(|z|) \chi(|z|)} \end{pmatrix},
 \end{equation*}
where $Q$ is determined by the frame and the cutoff function $\chi$ is the same as in that earlier setting.  Recall also 
from \cite[Def.\ 5.1]{MSWWgeometry} that an {\it exponential packet} in the region $\{\chi(|z|)=1\}$  is a (possibly matrix-valued) 
function of the form 
\begin{equation*}
\rho_t(z)=t^2\rho(t^2z),
\end{equation*}
where $\rho \in \calC^\infty(\CC)$ converges exponentially to $0$ as $|z|\to \infty$. 
\begin{lem} \label{lem:localnormalform} 
Let $[(\dot{\eta}, \dot{\varphi})] \in T_{(\delbar_E, \varphi)} \cM'$ be an infinitesimal Higgs bundle deformation.
There is a unique representative in this equivalence class which has the form, in some holomorphic frame on $\D_p$, 
 \begin{equation} \label{eq:defshape}
  \dot{\eta} =0, \qquad  \dot{\varphi} = \begin{pmatrix} 0 &  0 \\ \frac{\dot{P}}{z} & 0 \end{pmatrix} \de z, \qquad  \dot{\nu}_\infty = \frac{\dot{P}}{4} \begin{pmatrix}1& 0 \\ 0 & -1 \end{pmatrix}
 \end{equation}
where  $\dot{P}$ is holomorphic. The difference $\dot{\nu}_t^\app-\dot \nu_{\infty}$ is an exponential  packet, as defined above.
\end{lem}
\begin{proof} 
The lemma follows from the following  two claims.
\medskip\\
\textsc{Claim 1:}
\emph{There is a unique representative in the equivalence class of $[(\dot{\eta}, \dot{\varphi})]$ taking
the form \eqref{eq:defshape}.}\\ 
%\begin{proof}
$\triangleright$
The existence of a representative in the class $[(\dot{\eta}, \dot{\varphi})]$ of this form is proved in 
Lemma \ref{lem:higgssimpleshape}. The deformation $\dot{\nu}_\infty$ solves 
$[\varphi^{*_{h_\infty}}, \dot{\varphi} + [\dot{\nu}_\infty, \varphi]]=0$ and $\del_E^{h_\infty} \delbar_E \dot{\nu}_\infty
- \del_E^{h_\infty} \dot{\eta}=0$. Using the posited form of $[(\dot{\eta}, \dot{\varphi})]$, the equation 
$[\varphi^{*_{h_\infty}}, \dot{\varphi} + [\dot{\nu}_\infty, \varphi]]=0$ implies that 
\begin{equation*}
 \dot{\nu}_\infty =  \frac{\dot{P}}{4} \begin{pmatrix} 1 & 0 \\ 0 & -1 \end{pmatrix} + F  \varphi,
\end{equation*}
for some function $F$. The second equation implies that $\delbar F=0$. Since $\dot\nu_{\infty}$ is bounded, we have 
$F(0)=0$.  As in the proof of Lemma \ref{lem:higgssimpleshape}, applying an infinitesimal gauge transformation of the form
\begin{equation*}
\dot\gamma = \begin{pmatrix}  0 &   f z \\  f & 0\end{pmatrix}=zf\varphi,\qquad \mbox{where $f$ is any holomorphic function,}
\end{equation*}
does not alter the representative $(\dot{\eta}, \dot{\varphi})$ and only changes $\dot \nu_{\infty}$ to 
$\dot \nu_{\infty}+\dot\gamma$. Thus setting $f=F/z$ (which is holomorphic) yields $\dot \nu_{\infty}$ 
as in \eqref{eq:defshape}.
$\triangleleft$
\medskip\\
\textsc{Claim 2:} \emph{The difference $\dot{\nu}_t^\app-\dot \nu_{\infty}$ is an exponential  packet. }

\noindent$\triangleright$ We  use $\dot\nu_\infty$ as an approximate solution for the equation determining 
$\dot{\nu}_t^\app$ from $(\dot\eta,\dot\varphi)$ for $h_t^\app$. Since $h_t^\app= h_\infty$ outside the 
disks $\D_p$, the defect 
\[
\rho_t:=\partial_E^{h_t^\app}\bar\partial_E \dot\nu_\infty - \partial_E^{h_t^\app} \dot \eta - \Bigl[ \varphi^{*_{h_t^\app}}, \dot\varphi + [ \dot \nu_\infty, \varphi] \Bigr]
\]
which measures how far $\dot\nu_\infty$ is from solving the equation is supported inside these disks. 
More precisely, computing in a frame as above and using Claim 1 we get $\partial_E^{h_t^\app}\bar\partial_E \dot\nu_\infty=0$ 
(since $\bar\partial_E \dot\nu_\infty=0$), $- \partial_E^{h_t^\app} \dot \eta =0$ (since $\dot \eta=0$) and finally
\[
\rho_t=-\Bigl[ \varphi^{*_{h_t^\app}}, \dot\varphi + [ \dot \nu_\infty, \varphi] \Bigr] %= \frac{\dot P}{2|z|}
= \frac{\dot P}{|z|}\sinh\bigl(2m_t(|z|)\chi(|z|)\bigr) 
\sigma_3.
\]
Now $\rho_t=O(r^{-2 + \varepsilon})$ for some $\varepsilon>0$ and
\[
\mathcal{P}_t:\cD_{\Fr}^{0,\alpha}(\mathcal P_t)(\mu)  \to r^{\nu-2} \calC^{0,\alpha}_b,\quad \dot\nu\mapsto \mathcal{P}_t\dot\nu:= \partial^{h_t^\app}_E\bar\partial_E\dot \nu - \Bigl[\varphi^{*_{h_t^\app}}, [\dot \nu, \varphi]\Bigr] 
\]
is an isomorphism. Hence we can find a unique $\mu_t \in \cD_{\Fr}^{0,\alpha}(\mathcal P_t)(\mu)$ such that $\mathcal{P}_t \mu_t = \rho_t$. Now by definition of $\rho_t$, 
\[
\mathcal{P}_t \dot\nu_\infty  = \rho_t +   \Bigl[ \varphi^{*_{h_t^\app}}, \dot\varphi  \Bigr],
\]
so that with $\mu_t$ as above, 
\[
\mathcal{P}_t(\dot\nu_\infty - \mu_t) =   \Bigl[ \varphi^{*_{h_t^\app}}, \dot\varphi  \Bigr],
\]
i.e.\ $\dot\nu_t^\app=\dot\nu_\infty - \mu_t$ solves the desired equation. 
Consequently, by a straightforward adaption of \cite[Prop.\ 5.3]{MSWWgeometry} to the present case, the inverse $\mu_t=\mathcal{P}_t^{-1}\rho_t$ is the product of an exponential packet supported in $\D_p$ with an 
extra factor $t^{\sigma}$ for some $\sigma$.  In particular, its restriction to $\partial\D_p$ is of order 
$O( \e^{-\gamma t})$. Hence  $\dot\nu_t^\app$ differs from the diagonal $\dot\nu_\infty$ by this exponentially 
small term. Thus   $\dot\nu_t^\app$ is a solution  on $\D_p$ of the boundary value problem for the equation
\begin{equation*}
\mathcal P_t	\dot\nu_t^\app= \Bigl[ \varphi^{*_{h_t^\app}}, \dot\varphi  \Bigr]
\end{equation*}
boundary conditions which are diagonal up to an exponentially small perturbation. We note also that the right-hand side is exactly
diagonal. Since $\mathcal P_t$ decouples on $\D_p$ into diagonal and off-diagonal components, it follows by a standard 
maximum principle argument that the off-diagonal component of   $\dot\nu_t^\app$ is of order  $O
( \e^{-\varepsilon t})$ everywhere on $\D_p$. $\triangleleft$
\end{proof}

\begin{prop} \label{lem:localDN} 
Fix $(\dot{\eta}, \dot{\varphi}) \in T_{(\delbar_E,  \varphi)} \cM'$. Then for some $\varepsilon >0$, 
\begin{equation*} \label{eq:propnear}
 \norm{(\dot{\eta}, \dot{\varphi}, \dot{\nu}^\app_t)}^2_{g_{\app, t}} - 
\norm{(\dot{\eta}, \dot{\varphi}, \dot{\nu}_\infty)}^2_{g_{\semif, t}}  = O\big(\e^{-\varepsilon t}\big).
\end{equation*}
\end{prop}
\begin{proof} 
 We first observe that the integrand localizes on disks $\mathbb{D}=\mathbb{D}_p(\frac{1}{2})$ of radius $\frac 12$ around the zeros and strongly parabolic points, i.e. is exponentially small on their complement. On these disks (where $h^{\mathrm{app}}_t=h^{\model}_t$), we choose the unique representative $(\dot{\eta}, \dot{\varphi})$ of the equivalence class $[(\dot{\eta}, \dot{\varphi})]$ and a holomorphic frame from Lemma \ref{lem:localnormalform}. 
 
For convenience, since our representative of $\dot{\eta}$ vanishes on the disk, we introduce the notation
\begin{equation*} \label{eq:delta}
 \delta\Big( (0, t\dot{\varphi}), (h_1, \dot \nu_1), (h_2, \dot \nu_2) \Big)=
 2t^2 \IP{\dot{\varphi} +[\dot{\nu}_1, \varphi], \dot{\varphi}}_{h_1}-
  2t^2\IP{\dot{\varphi} +[\dot{\nu}_2, \varphi], \dot{\varphi}}_{h_2}.
\end{equation*}
In order to prove the exponential decay of 
\begin{eqnarray*}
\norm{(0, \dot{\varphi}, \dot{\nu}_t^\app)}^2_{g_{\app},t(\mathbb{D})}-\norm{(0, \dot{\varphi}, \dot{\nu}_\infty)}^2_{g_{\semif,t(\mathbb{D})}}&=&  2t^2 \int_{\mathbb{D}} 
 \IP{\dot{\varphi} +[\dot{\nu}_t^\app, \varphi], \dot{\varphi}}_{h_t^\app}-
  \IP{\dot{\varphi} +[\dot{\nu}_\infty, \varphi], \dot{\varphi}}_{h_\infty}\\
  &=&  \int_{\mathbb{D}}  \delta \left( (0, t\dot{\varphi}), (h^{\mathrm{app}}_t, \dot \nu^{\mathrm{app}}_t), (h_\infty, \dot \nu_\infty) \right),
 \end{eqnarray*}
 on a disk of radius $\frac{1}{2}$ around a zero or strongly parabolic point,  we break the integrand into two terms as
\begin{equation*} \label{eq:decomp2}
\begin{aligned}
& \delta \left( (0, t\dot{\varphi}), (h^{\mathrm{app}}_t, \dot \nu^{\mathrm{app}}_t), (h_\infty, \dot \nu_\infty) \right) \\ = &\delta \Big( (0, t\dot{\varphi}), (h^{\model}_t, \dot \nu^X_t), (h_\infty, \dot \nu_\infty) \Big)
+ \delta \Big( (0, t\dot{\varphi}), (h^{\mathrm{model}}_t, \dot \nu^{\mathrm{app}}_t), (h^{\mathrm{model}}_t, \dot \nu^X_t) \Big)
\end{aligned}
\end{equation*}
and consider each term separately.  Here, $\dot\nu^X_t$ is defined using a well-chosen holomorphic variation, as follows.
If $\dot{P}= \sum_{n=0}^\infty a_n z^n$, then following Dumas-Neitzke \cite[Eq.\ 10.12]{DumasNeitzke}, we set
\begin{equation*} \label{eq:chi}
\mathcal{X}= \sum_{n=0}^\infty \frac{a_n}{4n-2} z^{n+1}
\end{equation*}
and check that $z \dot{P} + 2 \mathcal{X} - 4 z \mathcal{X}'=0$. This yields a holomorphic vector field 
$X=\mathcal{X} \frac{\del}{\del z}$ generating the holomorphic deformation adapted 
to $\dot{q}=\frac{\dot{P}}{z} \de z^2$. Now define
\begin{equation*}
 F^X_t=  \del_z\mathcal{X} + 2 \mathcal{X} \del_z\left(-\frac{1}{2} \log |z| + \mt \right),\ \ \mbox{so}\ F_\infty^X = \frac{1}{4}  \dot{P}.
\end{equation*}
Then $\dot{\nu}^X_t = F^X_t \sigma_3$ satisfies the complex variation equation \eqref{eq:ingaugetriple} 
for $(\delbar_E, t\varphi, h^{\model}_t)$ and Higgs bundle variation $(0, t \dot{\varphi})$ as in \eqref{eq:defshape}.
The fact that $\dot\nu^X_t$ satisfies the complex variation equation reduces to the equality
\begin{equation} \label{eq:scalarX}
 \left( \del_z \del_{\zbar} - 4 t^2 |z|^{-1} \cosh (2\mt )  \right) F^X_t +  t^2 \e^{-2\mt} |z|^{-1} \zbar \dot{P} =0.
\end{equation}
The assertion of the proposition then follows from the following two claims.
\medskip\\
\noindent\textsc{Claim 1:} \emph{There is a positive constant $\varepsilon>0$ such that
\begin{equation} \label{eq:claim3a}
 \int_{\D_p(1/2)}\delta \left( (0, t\dot{\varphi}), (h^{\model}_t, \dot \nu^X_t), (h_\infty, \dot \nu_\infty) \right) = O( \e^{-\varepsilon t}).
\end{equation}}
$\triangleright$
The contribution of the $L^2$-metric to the integrand is
\begin{equation}\label{eq:L2integrandmodel}
\langle \dot \varphi_t,	\dot \varphi_t\rangle_{h^{\model}_t}+\langle [\varphi_t,\dot\nu_t],	\dot \varphi_t\rangle_{h^{\model}_t}=\frac{\e^{-2m_t}|\dot P|^2}{|z|}-\frac{2\e^{-2m_t}F_t^X\overline{\dot P}}{|z|}.
\end{equation}
Using $z \dot{P} + 2 \mathcal{X} - 4 z \mathcal{X}'=0$, we rewrite
\begin{equation}\label{eq:expressionFtX}
F_t^X=	\del_z\mathcal{X} + 2 \mathcal{X} \del_z\left(-\frac{1}{2} \log |z| + \mt \right)=\del_z\mathcal{X}-\frac{\mathcal{X}}{2z}+2\mathcal{X}\del_z \mt=\frac{\dot P}{4}+2\mathcal{X}\del_z \mt.
\end{equation}
Then \eqref{eq:L2integrandmodel} becomes
\begin{equation*}
	\langle \dot \varphi_t,	\dot \varphi_t\rangle_{h^{\model}_t}+\langle [\varphi_t,\dot\nu_t],	\dot \varphi_t\rangle_{h^{\model}_t}=\frac{\e^{-2m_t}|\dot P|^2}{2|z|}-\frac{4\e^{-2m_t} \overline{\dot P}\mathcal{X}\partial_z\mt}{|z|}.
\end{equation*}
The contribution of the semiflat metric to the integrand is $-\frac{|\dot P|^2}{2}$, so the integrand in \eqref{eq:claim3a}  is 
\begin{eqnarray*}
I_t=  t^2    \frac{|\dot P|^2(\e^{-2m_t} -1)  - 8 \e^{-2\mt} \overline{\dot P}
  \mathcal{X} \del_z \mt}{|z|}  \, \de z\wedge \de \bar z.
\end{eqnarray*}
A brief calculation using the identities $z \dot{P} + 2 \mathcal{X} - 4 z \mathcal{X}'=0$ and $\partial_z\overline{\dot{P}}=0$
shows that $d\beta_t=I_t$, where
\begin{equation*}
\beta_t=	 4t^2  \frac{(\e^{-2m_t} -1) \overline{\dot{P}} \mathcal{X}}{|z|}\, \de \zbar.
 \end{equation*}
By Stokes' theorem,
\begin{equation}\label{eq:stokesthm}
\int_{\D_p(1/2)}I_t=\int_{\partial\D_p(1/2)}	\beta_t.
\end{equation}
This identity is justified by the fact that the integrand $I_t$ blows up only like $|z|^{-2 + \delta}$ for some $\delta > 0$;
indeed, since $\mathcal{X}(z)\sim z$, we have
\begin{equation*}
 \e^{-2\mt}\del_z \mt\sim |z|^{-1-2(\alpha_1-\alpha_2)}|z|^{-1}=|z|^{-2+2(\alpha_2-\alpha_1)}.
\end{equation*}
Since    $\beta_t$ is exponentially small on $\partial\D_p(1/2)$,  the claim follows from 
\eqref{eq:stokesthm}.$\triangleleft$ 
\medskip\\
\noindent\textsc{Claim 2:} \emph{There is a positive constant $\varepsilon>0$ such that
\begin{equation} \label{eq:claim3b}
 \int_{\D_p(1/2)}\delta \left( (0, t\dot{\varphi}), (h^{\mathrm{model}}_t, \dot \nu^{\mathrm{app}}_t), (h^{\mathrm{model}}_t, \dot \nu^X_t)
\right)= O( \e^{-\varepsilon t}).
\end{equation}
}
$\triangleright$
Since the difference $\dot\nu_t^{\mathrm{app}}-\dot\nu_\infty$ is an exponential packet, $\dot \nu^{\mathrm{app}}_t$ 
is exponentially close to 
\begin{equation*}
	\dot \nu_{\infty}=\frac{\dot{P}}{4} \begin{pmatrix} 1 & 0 \\ 0 & -1 \end{pmatrix}
\end{equation*}
on the boundary of $\D_p(1/2)$. 
This is also true for $\dot \nu^X_t$ as follows from the right-hand side of \eqref{eq:expressionFtX}. Hence 
\begin{equation*}
\left.\left(\dot \nu^{\mathrm{app}}_t-\dot \nu^X_t\right)\right|_{\partial \D_{1/2}(p)}=	O( \e^{-\varepsilon t})
\end{equation*}
(and similarly for all derivatives of that difference).Both are solutions of the equation
\begin{equation*}
\del_E^{h^{\mathrm{model}}_t} \delbar_E \dot{\nu}  - \left[\varphi^{*_{h^{\mathrm{model}}_t}}, \dot{\varphi} + [\dot{\nu}, \varphi]\right]=0.	
\end{equation*}
hence their difference $\mu_t:=\dot \nu^{\mathrm{app}}_t-\dot \nu^X_t$ satisfies the homogeneous equation
\begin{equation}\label{eq:homogenmu}
	\del_E^{h^{\mathrm{model}}_t} \delbar_E \mu_t  - \left[\varphi^{*_{h^{\mathrm{model}}_t}},   [\mu_t, \varphi]\right]=0.
\end{equation}
Now consider the function $|\mu_t|_{h_t^{\mathrm{model}}}^2$; we calculate that 
\begin{equation*}
\frac{1}{2
}\Delta |\mu_t|_{h_t^{\mathrm{model}}}^2=-|\nabla_{D_t^{\mathrm{model}}}\mu_t|_{h_t^{\mathrm{model}}}^2+\langle \Delta_{D_t^{\mathrm{model}}}\mu_t,\mu_t\rangle_{h_t^{\mathrm{model}}}\leq 0,
\end{equation*}
where $\Delta=\de^{\ast}\de$ and $D_t^{\mathrm{model}}$ is the Chern connection with respect to $\bar\partial_E$ and 
$h_t^{\mathrm{model}}$. To show that   $\langle \Delta_{D_t^{\mathrm{model}}}\mu_t,\mu_t
\rangle_{h_t^{\mathrm{model}}}\leq 0$, substitute
\begin{equation*}
    \Delta_{D_t^{\mathrm{model}}}\mu_t=-2i\ast\Big[\varphi^{*_{h^{\mathrm{model}}_t}},   [\mu_t, \varphi]\Big]+i\ast [F_{D_t^{\mathrm{model}}},\mu_t]
\end{equation*}
using \eqref{eq:homogenmu} and the general identity $2\del_E^{h^{\model}_t} \delbar_E =F_{D_t^{\mathrm{model}}}+i\ast \Delta_{D_t^{\mathrm{model}}}$.
We see that $(\delbar_E,\varphi, h_t^{\mathrm{model}})$ is a solution of the self-duality equations, whence
\begin{equation*}
    i\ast [F_{D_t^{\mathrm{model}}},\mu_t]=  i\ast [F_{D_t^{\mathrm{model}}}^{\perp},\mu_t]=  -i\ast \Big[[\varphi,\varphi^{*_{h^{\mathrm{model}}_t}}],\mu_t\Big].
\end{equation*}
Now use the Jacobi identity to compute
\begin{multline*}
\langle \Delta_{D_t^{\mathrm{model}}}\mu_t,\mu_t\rangle_{h_t^{\mathrm{model}}}\\=\Big\langle -2i\ast\Big[\varphi^{*_{h^{\mathrm{model}}_t}},  [\mu_t, \varphi]\Big],\mu_t\Big\rangle_{h_t^{\mathrm{model}}}  
+\Big\langle-i\ast \Big[[\varphi,\varphi^{*_{h^{\mathrm{model}}_t}}],\mu_t\Big],\mu_t\Big\rangle_{h_t^{\mathrm{model}}}\\
=-2\Big|[\varphi,\mu_t]\Big|_{h_t^{\mathrm{model}}}^2-\Big|[ \varphi^{*_{h^{\mathrm{model}}_t}},\mu_t]\Big|_{h_t^{\mathrm{model}}}^2+\Big|[\varphi,\mu_t]\Big|_{h_t^{\mathrm{model}}}^2 \\
=-\Big|[\varphi,\mu_t]\Big|_{h_t^{\mathrm{model}}}^2-\Big|[ \varphi^{*_{h^{\mathrm{model}}_t}},\mu_t]\Big|_{h_t^{\mathrm{model}}}^2\leq 0.
\end{multline*}
This shows that $|\mu_t|_{h_t^{\mathrm{model}}}^2$ is subharmonic and its restriction to $\partial \D_p(1/2)$ 
is $O( \e^{-\varepsilon t})$.  By the maximum principle, $\mu_t$ itself is of order $O
(\e^{-\varepsilon t})$, and this implies \eqref{eq:claim3b}.$\triangleleft$
\end{proof}

%TESTING

\subsection{Proof of Theorem \ref{thm:L2vssemiflat}}\label{sec:proofasygeo}

\begin{prop}\label{prop:L2vsapp}
Fix the moduli space $\cM_{\mathrm{Higgs}}$ of (either weakly or strongly) parabolic $SL(2,\C)$ Higgs bundles.
Let $(\delbar_E, \varphi) \in \cM'_{\mathrm{Higgs}}$ be any stable Higgs bundle and $(\dot{\eta}, \dot{\varphi})$  
an infinitesimal variation. Identify $(\dot{\eta}, \dot{\varphi})$ with its image in $T_{(\delbar_E, \varphi, h_t)} \cM_t$. 
As $t \rightarrow \infty$, the difference between $g_{L^2,t}$ and $g_{\app, t}$ decays exponentially in $t$:
\begin{equation*} 
\norm{(\dot{\eta},  \dot{\varphi}, \dot{\nu}_t)}^2_{g_{L^2,t}}- \norm{(\dot{\eta},  
\dot{\varphi}, \dot{\nu}_t^\app)}^2_{g_{\app,t}}=  O(\e^{-\varepsilon t}). 
\end{equation*}
\end{prop}
\begin{proof}
We use that $h_t(w_1,w_2)= h_t^\app(\e^{\gamma_t} w_1, \e^{\gamma_t} w_2)$ for any $h_0$-Hermitian 
$\gamma_t \in \mathcal{D}^{0, \alpha}_{\mathrm{Fr}}(\delta)$ satisfying $\|\gamma_t\|_{\cC^{2, \alpha}_b} 
\leq \e^{-\varepsilon t}$, cf.\ Theorem \ref{thm:perturb}, as well as the observation that 
$\dot \nu_t - \dot \nu_t^{\mathrm{app}}$ also decays exponentially (which follows since these two quantities satisfy
equations whose coefficients and inhomogeneous terms differ by exponentially small amounts).
\end{proof}

\begin{proof}[Proof of Theorem \ref{thm:L2vssemiflat}]
 The theorem follows directly from Proposition \ref{lem:localDN}, Corollary \ref{thm:semiflatisL2} and Proposition 
\ref{prop:L2vsapp}. 
\end{proof}

\subsection{Review: The conjecture of Gaiotto, Moore and Neitzke} \label{sec:GMN}
We review Gaiotto-Moore-Neitzke's conjecture for a general (strongly or weakly) parabolic Higgs bundles. 
In \eqref{eq:trescaledL2}, we defined the following metric on $\cM_t$: 
\begin{equation*}
 \norm{(\dot{\eta}, \dot{\varphi}, \dot{\nu}_t)}^2_{g_{L^2,t}} = 2 \int_C 
 \IP{\dot{\eta} - \delbar_E \dot{\nu}_t, \dot{\eta}}_{h_t} + t^2\IP{\dot{\varphi}
 +[\dot{\nu}_t, \varphi], \dot{\varphi}}_{h_t}.
\end{equation*}
where $\dot{\nu}_t$ is the unique solution to
\begin{equation*}
 \del_E^{h_t} \delbar_E \dot{\nu}_t - \del_E^h \dot{\eta} - t^2 \left[\varphi^{*_{h_t}}, \dot{\varphi} + [\dot{\nu}_t, \varphi]\right]=0.
\end{equation*}
In \S\ref{sec:semif}, we defined the $t$-rescaled semiflat metric
\begin{equation*}
 \norm{(\dot{\eta}, \dot{\varphi}, \dot{\nu}_\infty)}^2_{g_{\semif, t}} = 2 \int_C 
 \IP{\dot{\eta} - \delbar_E \dot{\nu}_\infty, \dot{\eta}}_{h_\infty} + t^2\IP{\dot{\varphi}
+[\dot{\nu}_\infty, \varphi], \dot{\varphi}}_{h_\infty}
\end{equation*}
where the $t$-independent section $\dot{\nu}_\infty$ solves
  the decoupled equations
\begin{equation*}
 \del_E^{h_\infty} \delbar_E \dot{\nu}_\infty - \del_E^{h_\infty} \dot{\nu}_\infty =0, \qquad \left[\varphi^{*_{h_\infty}}, \dot{\varphi} + [\dot{\nu}_\infty, \varphi]\right]=0.
\end{equation*}

\bigskip

Fixing the Higgs bundle $(\cE, \varphi)$ and deformation $(\dot{\eta}, \dot{\varphi})$, we consider the difference
$\norm{(\dot{\eta}, \dot{\varphi}, \dot{\nu}_t)}_{g_{L^2,t}} - \norm{(\dot{\eta}, \dot{\varphi}, \dot{\nu}_\infty)}_{g_{\semif,t}}$.
Gaiotto-Moore-Neitzke give a beautiful conjectural description of the hyperk\"ahler metric on the Hitchin moduli space;
at the coarsest level, it states
\begin{conj}[Weak form of Gaiotto-Moore-Neitzke's conjecture for $\cM_{\SU(2)}$] \label{conj:weakGMN}
Fix a Higgs bundle $(\delbar_E, \varphi)$ in the regular locus $\cM'_{\mathrm{Higgs}}$.
Then, as $t \to \infty$
\begin{equation*}
 \norm{(\dot{\eta}, \dot{\varphi}, \dot{\nu}_t)}_{g_{L^2,t}}^2 - \norm{(\dot{\eta}, \dot{\varphi}, \dot{\nu}_\infty)}_{g_{\semif, t}}^2=
O(\e^{- 2 M t })
\end{equation*}
where $M$ is the length of the shortest geodesic on the punctured spectral curve $\Sigma_b-\pi^{-1}(D)$ for $b=\mathrm{Hit}(\cE, \varphi)$ which is not a loop around $\pi^{-1}(D)$ (see Figure \ref{fig:trajectories}), as measured in the singular flat metric $\pi^*|\det \varphi|$.
\end{conj}

\begin{figure}[h] 
\begin{centering} 
\includegraphics[height=1.5in]{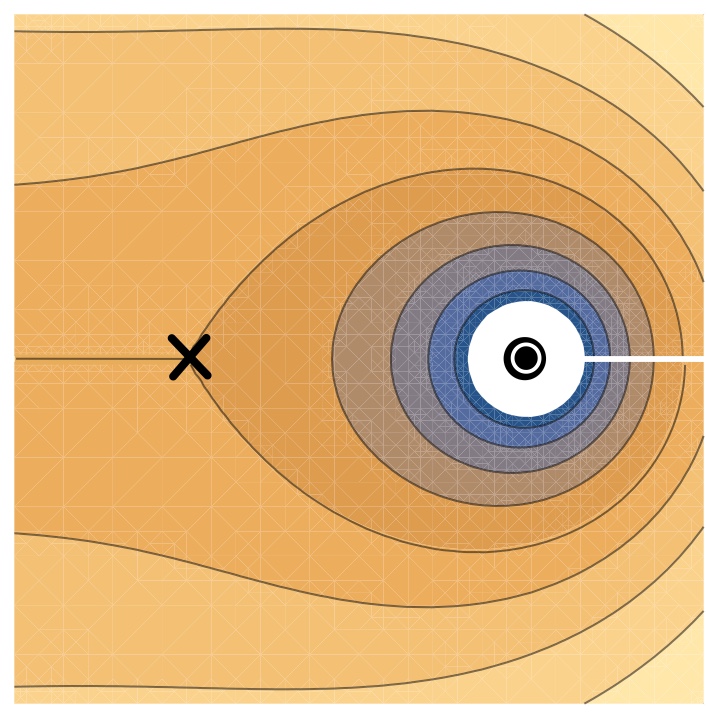}
\caption{\label{fig:trajectories} Some geodesic paths in the metric $|\det \varphi|$
on $C$ near a double pole ($\bullet$) and simple zero ($\times$). Near a double pole with residue $\sigma\frac{\de z}{z}$, there is an annulus of closed geodesics that are each of length $2 \pi |\sigma|$, depending only on moduli space parameters.  Such geodesics are elements of the flavor lattice $\Gamma^{\mathrm{flavor}}$ and do not enter into the conjectured rate of exponential decay for the difference $g_{L^2,t} - g_{\semif,t}$.}
\end{centering}
\end{figure}

Note that although Theorem \ref{thm:L2vssemiflat} proves some degree of exponential decay, the constant is not
the sharp one conjectured by Gaiotto-Moore-Neitzke.  A closer examination of our the methods in this paper (which
we do not explain carefully here) can improve our result to $O(\e^{-(M-\varepsilon)t})$ for any $\varepsilon>0$. (Note,
however, that in a special case considered in the next section we are able to get the exact predicted rate.) 

There is a significantly stronger version of this conjecture which we explain below. The hyperk\"ahler metric $g_{L^2}$ on $\cM$ 
is completely determined by and determines a twisted fiber-wise holomorphic symplectic structure on the twistor space 
$\mathcal{Z}=\cM \times \CP^1$.  In \cite{GMNhitchin, GMNwallcrossing}, the authors conjecture that certain holomorphic Darboux 
coordinates $\mathcal{X}_{\gamma}$ on the twistor space $\mathcal{Z}=\cM \times \CP^1$ solve a certain integral equation, see
\cite[Eq 4.8]{Neitzkehyperkahler}.  Implicit in the solution to this equation is the hyperk\"ahler metric, but it is difficult to pass 
between the twistor space formulation of the conjecture and what it says about $g_{L^2}$.

We now describe a slightly stronger version of Conjecture~\ref{conj:weakGMN} that highlights the ingredients in this
integral equation.  This holds \emph{only on the tangent space to the Hitchin section}.    One may  attempt to solve
the integral equation by a Picard iteration starting from the initial hyperk\"ahler metric $g_{\semif, t}$. The successive iterates
should approach $g_{L^2,t}$.  The first iterate yields the following expression for the difference of the two metrics 
over a ray $(\cE, \varphi, h_t) \in \mathcal{M}_t$:

\begin{equation} \label{eq:firstit}
 g_{L^2,t} = g_{\semif,t} - \frac{2}{\pi} \sum_{\gamma \in \Gamma_b} \Omega(\gamma; b) K_0(2|Z_{\gamma,t}|) (\de |Z_{\gamma,t}|)^2 + \ldots
\end{equation}
Here
\begin{itemize}
\item $\Sigma_b \overset{\pi}{\rightarrow} C$ is the spectral cover;
\item $\Gamma_b$ is the fiber of a certain local system of lattices (known as the charge lattice, see \cite{Neitzkenotes}) 
$\Gamma \rightarrow \cB'$ over $b$, fitting into the exact sequence
\begin{equation*}\label{eq:extension}
0 \rightarrow \Gamma^{\mathrm{flavor}} \to \Gamma %\overset{\mathrm{proj}_\Gamma}{\rightarrow} 
\rightarrow \Gamma^{\mathrm{gauge}} \to  0.
\end{equation*}
The other two lattices $\Gamma^{\mathrm{flavor}}$ and $\Gamma^{\mathrm{gauge}}$ are determined in the following way.  If 
$\overline{\Sigma}_b \subset \mathrm{Tot}(K(D))$ is the compactified spectral curve, then the fiber over $b \in \mathcal B$ of
the gauge lattice is 
$$
\Gamma^{\mathrm{gauge}}_b = \ker \left( H_1(\overline{\Sigma}_b, \Z) \rightarrow H_1(C, \Z)\right).
$$ 
On the other hand, the fiber $\Gamma^{\mathrm{flavor}}_b$ of the flavor lattice is, for generic divisors $D_s \cup D_w$, 
a free $\Z$-module generated by loops $\gamma$  around the points of $\pi^{-1}(D)$ in the punctured spectral curve 
$\Sigma_b-\pi^{-1}(D)$, see \cite{Neitzkenotes}. Hence, for generic $D$, 
$$
\Gamma_b = \ker\left( H_1(\overline{\Sigma}_b- \pi^{-1}(D), \Z) \rightarrow H_1(C, \Z)\right).
$$ 
Note that this involves the homology of the punctured spectral curve. 
\item Finally, $Z_{\bullet, t}$ is the period map 
\begin{equation*}
 Z_{\bullet, t}: \Gamma_b \rightarrow \C \qquad Z_{\gamma,t} = t \oint_\gamma \lambda,
\end{equation*}
where $\lambda$ is the tautological (Liouville) $1$-form on $\mathrm{Tot}(K_C)$.  
\begin{rem}
For loops $\gamma$ in the image of $\Gamma^{\mathrm{flavor}}$, $Z_{\gamma, t}: \cB' \rightarrow \C$ is a constant function. For example, 
for $G = \SU(2)$,  if $\gamma$ is a small counter-clockwise oriented loop around $p \in D_w$ with residue $\frac{\sigma \de z}{z}$, then $Z_{\gamma,t} \equiv 2\pi |\sigma|$. 
 \end{rem}
 
\item 
$K_0$ is the modified Bessel function of the second kind.
%\footnote{The function $K_0(x)$ solves the modified Bessel differential equation $x^2 y''(x) + x y'(x) - x^2 y(x) =0$ on $(0, \infty)$.  Within the two-dimensional family of solutions, the function $K_0(x)$ is determined (up to multiplication by constant) by the property that $\lim_{x \to \infty} K_0(x)=0$.  It is defined by the integral $K_0(x) = \int_0^\infty \frac{\cos(xt)}{\sqrt{t^2+1}} \de t$.}; and 
\item  $\Omega(\gamma; b)$ is an integer-valued generalized Donaldson-Thomas invariant
\footnote{If $\Gamma \rightarrow \cB'$ is the local system with fiber $\Gamma_b$, then $\Omega: \Gamma \rightarrow \Z$.  
Given a section $\gamma$ of $\Gamma$, the function $\Omega(\gamma; \cdot): \cB' \rightarrow \Z$ is typically not continuous,
but jumps at (real) codimension 1 walls in $\cB'$ and satisfies the Kontsevich-Soibelman wall-crossing formula. } 
(see \cite{kontsevichsoibelman, JoyceSong} and the discussion in \cite{Neitzkehyperkahler}).
 \end{itemize}            
The first correction in \eqref{eq:firstit} is from the smallest value $2|Z_{\gamma_0}|$ for which $\Omega(\gamma_0;b) \neq 0$ and 
$\gamma_0 \notin \Gamma^{\mathrm{flavor}}$. This is why, in Conjecture \ref{conj:weakGMN}, we only consider the length of the 
shortest geodesic which is not a loop around $\pi^{-1}(D)$. Since $K_0(x) \sim \sqrt{ \frac{\pi}{2x}} \e^{-x}$, the first correction 
$K_0(2t|Z_{\gamma_0}|) = O \left(\e^{-2|Z_{\gamma_0}|t} \right)$. The cross-terms in \eqref{eq:firstit} are of order 
$O \left(\e^{-4|Z_{\gamma_0}|t}\right)$, see \cite[Eq. 5.3]{Neitzkenotes}.

The constant of exponential decay conjectured for the whole Hitchin moduli space in Conjecture \ref{conj:weakGMN} is equal to 
this smallest allowable exponent $\e^{-2|Z_{\gamma_0}|}$.  For $G = \SU(2)$, $|Z_{\gamma_0}|$ is the length of a geodesic in
the class; indeed, $t \oint_\gamma |\lambda|$ is the length of $\gamma$ with respect to the singular flat metric 
$t^2 \pi^*|\det \varphi|$ on the spectral cover $\Sigma_b$. Note that $|\oint_\gamma  \lambda| \leq \oint_\gamma |\lambda|$, 
with equality if and only if $\gamma$ is a geodesic.  In particular, $Z_{\gamma_0} :=M$ is the length of the shortest 
geodesic on $\Sigma_b$ not surrounding a double pole in the singular flat metric $\pi^*|\det \varphi|$, 
cf.\ \textsc{Figure} \ref{fig:trajectories}.

\bigskip

In the strongly parabolic setting, the Hitchin moduli space admits a $\C^\times_\zeta$-action, $(\cE, \varphi) \mapsto (\cE, \zeta \varphi)$.
We can then phrase the conjecture in terms of the decay of the difference of the Hitchin and semiflat metrics along an 
$\R^+_t$-ray $[(\cE, t \varphi, h_t)]$ rather than comparing metrics on the different moduli spaces $\cM_t$ and $\cM_\infty$.
%Note that a strongly parabolic Hitchin moduli space admits a $\C^\times$-action, i.e. if $(\cE, \varphi)$ is in $\mathcal{M}_{\mathrm{Higgs}}$, then all scalar multiples $(\cE, \zeta \varphi)$ are also in $\mathcal{M}_{\mathrm{Higgs}}$; however, 
There is no $\C^\times$-action in the weakly parabolic case because at the double poles, multiplication of the Higgs field
by $\zeta$ induces $\frac{\de z^2}{z^2} \mapsto \frac{\zeta^2 \de z^2}{z^2}$.  In the strongly parabolic setting, the $\C^\times$-action
identifies $\cM_t$ (with metric $g_{L^2, t}$)  with $\cM_1 = \cM$ (with metric $g_{L^2})$ isometrically via $(\delbar_E, \varphi, h_t) \mapsto (\delbar_E, t \varphi, h_{t})$. 
%\begin{eqnarray*}
% \Xi: \cM_t &\rightarrow& \cM\\ \nonumber
% \end{eqnarray*}
%i.e. if $h_t$ solves the $t$-rescaled Hitchin equation for $(\delbar_E, \varphi)$,
%then $h_t$ solves the Hitchin equation for $(\delbar_E, t \varphi)$. 
On the tangent spaces this becomes $ (\dot{\eta}, \dot{\varphi}, \dot{\nu}_t) \mapsto   (\dot{\eta}, t \dot{\varphi}, \dot{\nu}_t)$.
%\begin{eqnarray*}
%\de\Xi: T_{(\delbar_E, \varphi, h_t)} \cM_t &\rightarrow& T_{\Xi((\delbar_E, \varphi, h_t))}\cM\\ \nonumber
%\end{eqnarray*}
%Thus, the map $\Xi$ gives an isometry between $(\cM_t, g_{L^2,t})$ and $(\cM_T, g_{L^2})$.
%This gives an identification of the $\R^+_t$-ray $(\delbar_E, \varphi, h_t) \in \cM_t$ in the \emph{$\R^+_t$-family of moduli spaces} with the \emph{$\R^+_t$-ray of Higgs bundles} $(\delbar_E, t\varphi, h_t) \in \cM$ in the \emph{single} moduli space $\cM$.

%The semiflat metric is a Riemannian submersion with flat fibers and induced metric on $\mathcal B'$ which is conic.
% the $\C^\times$-action is an isometry. Consequently, 
\begin{conj}[Conjecture \ref{conj:weakGMN} when $\cM_{\mathrm{Higgs}}$ admits a $\C^\times$-action] Fix a ray of Higgs bundles 
$(\cE, t \varphi) \in \cM$ and a corresponding family of tangent vectors $(\dot{\eta},t \dot{\varphi}) \in T_{(\cE, t \varphi)} \cM$; then
\begin{equation*}
\norm{(\dot{\eta}, t \dot{\varphi}, \dot{\nu}_t)}^2_{g_{L^2}} -\norm{(\dot{\eta}, t\dot{\varphi}, \dot{\nu}_\infty)}^2_{g_{\semif}} = 
O(\e^{- 2  M t }),
\end{equation*}
where $M$ is the length of a shortest geodesic with respect to the singular flat metric $\pi^*|\det \varphi|$ 
on the punctured spectral curve $\Sigma_b - \pi^{-1}(D)$ which is not a loop around $\pi^{-1}(D)$. 
\end{conj}
 
\section{The asymptotic geometry of the moduli space of strongly-parabolic Higgs bundles on
the four-punctured sphere} \label{sec:modulispace}
In this section we specialize the entire previous discussion so as to describe the asymptotic geometry of the moduli space 
of $\SL(2,\C)$-Higgs bundles on the 4-punctured sphere.  This is sometimes called the toy model because the moduli space 
in this case is four-dimensional, the lowest dimension possible.  For simplicity we restrict to the strongly parabolic setting 
here so there is a $\C^\times$ action. This action identifies the various torus fibers $\pi^{-1}(q)$, $q \neq 0$, with 
one another.  The only singular fiber is the nilpotent cone $\pi^{-1}(0)$. Thus the discriminant locus is compact
and $\cM'$ contains an entire neighborhood of infinity.  Because of this we can sharpen our results considerably 
to obtain the optimal rate of exponential decay for $g_{L^2} - g_{\mathrm{sf}}$.  We use this to show that
the moduli space is a gravitational instanton of type ALG, decaying exponentially to the flat model metric in 
the sense of \cite{ChenChenIII}. This exponential behavior is exceptional in the class of ALG-metrics.

We begin, in \S\ref{subsec:Higgs4}, with an explicit description of  the elements of the moduli space of Higgs bundles
and the associated spectral data in this 4-punctured sphere setting. The special K\"ahler metric on the base and semiflat 
metric are both quite simple  \S\ref{subsec:semiflat4}. We then turn in \S\ref{GMN} to a discussion of the predictions by GMN 
in this particular case. The optimal rate of exponential decay and attendant curvature decay are derived 
in \S\ref{optimal_decay}. Finally, \S\ref{ALG} contains a description of how this case fits into the Chen-Chen classification of 
ALG metrics \cite{ChenChenIII}.

\subsection{The moduli space \texorpdfstring{$\cM$}{M}}\label{subsec:Higgs4}
We now describe the moduli space of strongly parabolic $\SL(2,\C)$-Higgs bundles on the four-punctured sphere in detail. 
Take $C=\CP^1$ and choose any divisor $D$ with four distinct points of multiplicity $1$. The Higgs bundle moduli 
space depends on the complex structure of $(C,D)$, and without loss of generality, we may use a M\"obius transformation
to arrange that the points of $D$ are $0, 1, \infty, p_0$.

\begin{figure}[h] 
\begin{centering} 
\includegraphics[height=1in]{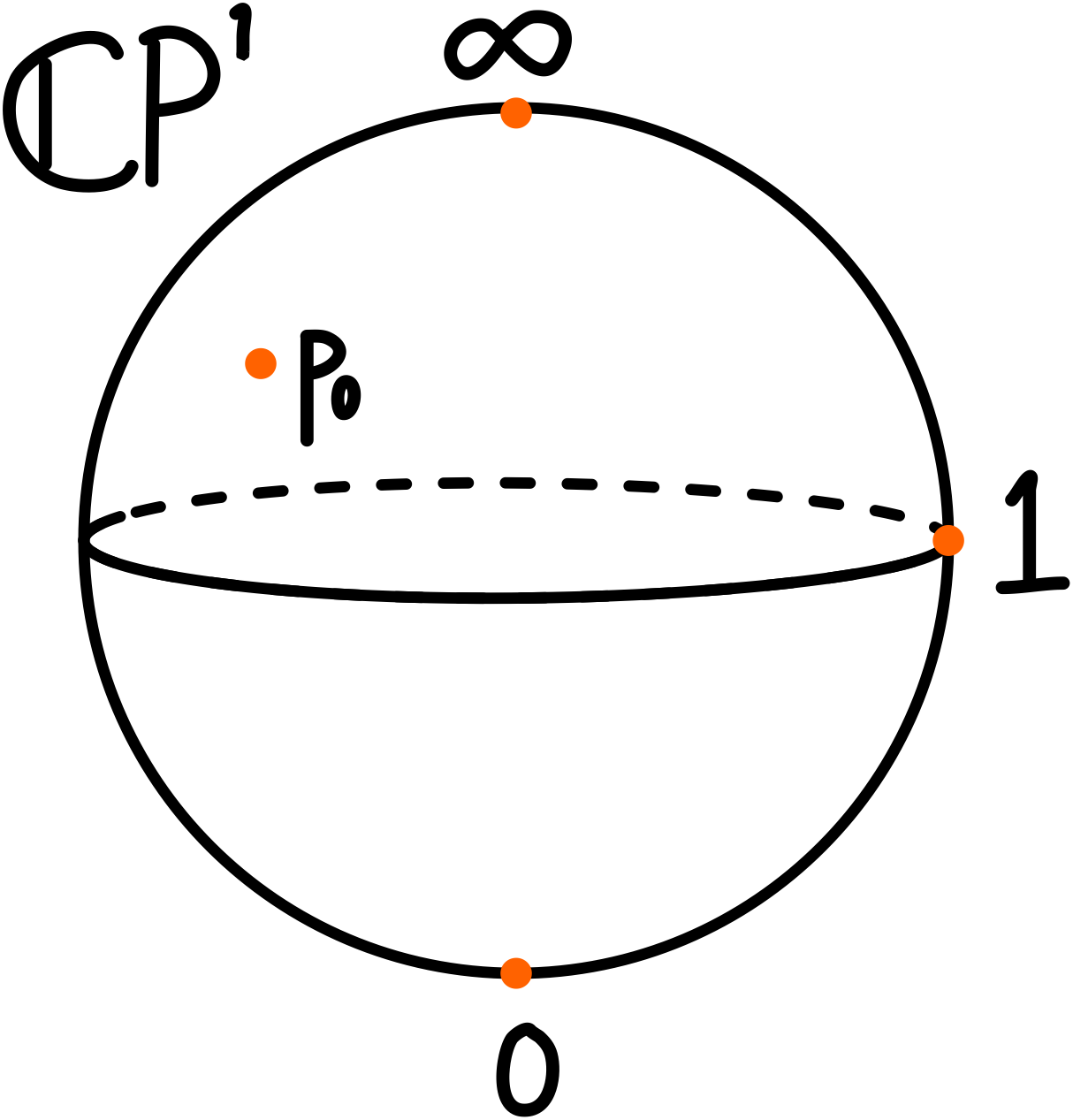}
\vspace{-.1in}\caption{\label{fig:fourpunctured}The four-punctured sphere}
\end{centering}
\end{figure}

%\subsubsection{Choice of fixed data}
As before we consider the case where the parabolic bundle $\mathcal{E}$ has full flags and that $\alpha_i(p) \in [0,1)$ 
with $\alpha_1(p) + \alpha_2(p)=1$ for each $p \in D$, i.e., $\vec{\alpha}(p)=(\alpha_p, 1-\alpha_p)$ for
$\alpha_p \in (0, \frac{1}{2})$. We assume also that $\mathrm{pdeg}_{\vec\alpha}\cE=0$.  In summary, the moduli space $\cM$ 
of strongly parabolic $\SL(2,\C)$-Higgs bundles depends on the choice of fourth point $p_0 \not\in \{0,1,\infty\}$, 
the complex vector bundle $E \rightarrow \CP^1$ of rank $2$ and degree $-4$, and the (real) parabolic
weight vector $\vec{\alpha}(p)$, $p\in D$, with $\alpha_0(p) \in (0, \frac12)$, $\alpha_1(p) = 1 - \alpha_0(p)$. 
For generic weight vector, the moduli space $\cM$ is a noncompact complex manifold of complex dimension $2$
\cite[Theorem 4.2]{BodenYokogawa}. As an algebraic surface it is the blowup of $\C \times T^2_\tau/ \pm 1$ at
four singular points, where $T^2_\tau$ is an elliptic curve \cite{hauseldiss}.  We show later that 
the modulus $\tau \in \mathbb{H}/\mathrm{SL}(2,\Z)$ is determined by the choice of $p_0$.  By \cite{konno93},
$\cM$ carries a complete hyperk\"ahler metric.   The space of weights $(0,\frac 12)^4$ is partitioned into
open chambers by semistability walls, and the weight vector is called generic if it lies in one of these chambers. 
If the weight vector lies on a wall of semistability, the moduli space is singular. Details about this chamber structure 
and the associated wall-crossing phenomena appear in work of Meneses \cite{meneses}. 
We see below that the structure of the regular part $\cM'$ as a complex manifold is the same for all 
weight vectors, generic or not. 

Note finally that for $\SL(2,\C)$-Higgs bundles, we fix the underlying holomorphic and Hermitian structure 
on $\Det \, E$. However, on $\CP^1$, these are determined by the degree $\deg \Det \, E=-4$. Indeed, 
$\Det \, \cE \simeq \cO(-4)$ and we fix the Hermitian metric
\begin{equation*} \label{eq:hdet}
h_{\mathrm{Det} E}= |z|^{2} |z-1|^{2} |z-p_0|^{2}
\end{equation*}
which is adapted to the induced parabolic weights on $\Det \, E \rightarrow \CP^1$.

\subsubsection{Hitchin fibration}\label{sec:Hitchinfibration}
The moduli space $\cM$ fibers over the space of meromorphic quadratic differentials
with simple poles at the points of $D$ by the map
\[
 \mathrm{Hit}: \cM \twoheadrightarrow \cB, \qquad  (\delbar_E, \varphi) \mapsto \det\varphi.
\]
Since $\mathrm{deg} \; K_{\CP^1}^{2}=-4$ equals the number of zeros minus the number of poles 
(counted with multiplicity) for any section of $K_{\CP^1}^{\otimes 2}$, we see that meromorphic 
quadratic differentials in $\cB$ have no zeros. Concretely, fix the usual holomorphic coordinate 
$z$ on $\C=\CP^1-\{\infty\}$. The Hitchin base is then 
\begin{equation*} \label{eq:base}
\cB = \left\{q=\frac{B}{z(z-1)(z-p_0)} \de z^2 \Big| B \in \C \right\} \simeq \C_B.
\end{equation*}
The regular locus is $\cB' \simeq \C^\times_B$.

\subsubsection{Explicit description of Higgs bundles in $\cM'$}
\begin{prop}
With all notation as above, the regular locus $\cM'$ is stratified by the underlying holomorphic bundle 
type of $\cE$, which is either $\cO(-2) \oplus \cO(-2)$ or $\cO(-3) \oplus \cO(-1)$. The large stratum 
is parametrized by triples $(B, u, x) \in \C^\times \times \C \times \C$ solving the cubic
\begin{equation} \label{eq:ellipticfiber}
B x(x-1)(x-p_0) + u^2 =0.
\end{equation}
The corresponding Higgs bundle is
\begin{equation*}\label{eq:bigstrata}
\cE \simeq \cO(-2) \oplus \cO(-2) \qquad 
\varphi_{B, u, x} = \begin{pmatrix}  u & -\frac{B z(z-1)(z-p_0) + u^2}{z-x} \\ z -x &  -u \end{pmatrix}\frac{\de z}{z(z-1)(z-p_0)}.
 \end{equation*}
The small stratum consists of pairs
\begin{equation*}\label{eq:smallstrata}
 \cE \simeq \cO(-1) \oplus \cO(-3) \qquad
 \varphi  = \begin{pmatrix} 0 & 1 \\ -\frac{B}{z(z-1)(z-p_0)} & 0 \end{pmatrix} \de z.
\end{equation*}
Consequently, the fiber over $\frac{B}{z(z-1)(z-p_0)} \de z^2 \in \cB$ is the elliptic curve in \eqref{eq:ellipticfiber}, 
compactified by adding the relevant point in the small stratum.
\end{prop}

\begin{rem}\label{rem:jinvariant} 
The elliptic curve $B x(x-1)(x-p_0)+u^2=0$  is isomorphic to the complex torus $\C/(\Z \oplus \tau \Z)$ 
where $p_0$ is related to $\tau$ by $p_0=\lambda(\tau)$; here $\lambda$ is the elliptic modular lambda function.
\end{rem}

\begin{proof}\label{prop:higgsbundles}
Because $\operatorname{pdeg}_{\vec\alpha}\cE=0$ and the sum of the parabolic weights is $4$, $\deg \cE=-4$. Thus 
using Grothendieck's theorem for vector bundles over $\CP^1$, we have $\cE \simeq \cO(m) \oplus \cO(-4-m)$
$m\geq -2$. Note that for any choice of $m$,
\begin{equation*}
\End \cE \simeq \begin{pmatrix} \cO & \cO(4+2m) \\ \cO(-4-2m) & \cO \end{pmatrix}.
\end{equation*}
In $\cM'$, the flags are determined from the Higgs field.  Indeed, since $\varphi$ has at most simple poles at $0, 1, p_0$,
we can write
\begin{equation*}
\varphi = \frac{\de z}{z(z-1)(z-p_0)}\begin{pmatrix} a(z) & b(z) \\ c(z) & -a(z) \end{pmatrix} \qquad \mbox{where } 
a(z), b(z), c(z) \in \C[z], 
 \end{equation*}
Passing to the coordinate $w=z^{-1}$ using the transition function $w \mapsto w^k$ for $\cO(k)$,
\begin{equation*}
\varphi = -\frac{\de w}{w(1-w)(1-wp_0)} \cdot w^2 \begin{pmatrix} w^0 a(w^{-1}) & w^{4+2m} b(w^{-1}) \\ w^{-4-2m} 
c(w^{-1}) & -w^0 a(w^{-1}) \end{pmatrix}, 
\end{equation*}
and since $\varphi$ has at most a simple pole at $\infty$ ($w=0$), we have
\begin{equation*}
 \deg a \leq 2, \quad  \deg b \leq 6+2m, \quad \deg c \leq -2-2m.
\end{equation*}
Additionally, $\det \varphi = \frac{B}{z(z-1)(z-p_0)} \de z^2$ for some $B \in \C^\times$, i.e., 
\begin{equation} \label{eq:determinant}
 a(z)^2 + b(z) c(z) = -B z(z-1)(z-p_0).
\end{equation}
We can rule out the cases $m \geq 0$ as follows since in these cases, $\deg c<0$ so $c=0$. Thus the left side of
\eqref{eq:determinant} has even degree while the right side is nonvanishing and has odd degree; this is a contradiction.  
The only two possible bundle types which remain are when $m=-2$ and $\cE \simeq \cO(-2) \oplus \cO(-2)$ or 
$m=-1$ and $\cE \simeq \cO(-1) \oplus \cO(-3)$ .

\medskip

Suppose that $\cE = \cO(-2) \oplus \cO(-2)$. In this case $\deg a, \deg b, \deg c \leq 2$, so 
  \begin{equation*}
  \varphi = \frac{\de z}{z(z-1)(z-p_0)}\begin{pmatrix} a_2 z^2 + a_1 z + a_0 & b_2 z^2 + b_1 z + b_0 z \\ c_2 z^2 + 
c_1 z + c_0 & - a_2 z^2 - a_1 z - a_0 \end{pmatrix}.
 \end{equation*}
The complex gauge group is the group of constant $\SL(2,\C)$ gauge transformations. We use this gauge group to 
write down an explicit representative in each equivalence class. From the condition in \eqref{eq:determinant}, 
$a_2^2 + b_2 c_2=0$, so the highest degree part is nilpotent. Using the complex gauge group, we can arrange that 
$\ker \begin{pmatrix} a_2 & b_2 \\ c_2 & -a_2 \end{pmatrix} \supset \begin{pmatrix} 1 \\ 0 \end{pmatrix}$. 
Thus each gauge orbit contains a representative of the form
 \begin{equation*}
  \varphi = \frac{\de z}{z(z-1)(z-p_0)}\begin{pmatrix} a_1 z + a_0 & b_2 z^2 + b_1 z + b_0 z \\ c_1 z + c_0 & - a_1 z - a_0 \end{pmatrix}.
 \end{equation*}
Note that from \eqref{eq:determinant}, $b_2 \neq 0$ and $c_1 \neq 0$. By the diagonal gauge transformation 
\begin{equation*}
g= \begin{pmatrix} c_1^{-1/2} & 0 \\ 0 & c_1^{1/2} \end{pmatrix},
\end{equation*}
we can make $c(z)$ monic. Thus, each gauge orbit contains a representative 
\begin{equation*}
\varphi = \frac{\de z}{z(z-1)(z-p_0)}\begin{pmatrix} a_1 z + a_0 & b_2 z^2 + b_1 z + b_0 z \\ z + c_0 & - a_1 z - a_0 
\end{pmatrix}.
 \end{equation*}
Now use 
 \begin{equation*}
  g= \begin{pmatrix} 1 & -a_1 \\ 0 & 1 \end{pmatrix},
 \end{equation*}
to make $a(z)$ constant: 
\begin{equation*}
\varphi = \frac{\de z}{z(z-1)(z-p_0)}\begin{pmatrix}  a_0 & b_2 z^2 + b_1 z + b_0 z \\ z + c_0 &  - a_0 \end{pmatrix}.
\end{equation*}
For notational simplicity, write $x=-c_0 \in \C$ and $u=a_0 \in \C$. Imposing \eqref{eq:determinant}, we find in 
each equivalence class a representative
\begin{equation*}
\varphi = \frac{\de z}{z(z-1)(z-p_0)}\begin{pmatrix}  u & -\frac{B z(z-1)(z-p_0) + u^2}{z-x} \\ z -x &  -u \end{pmatrix},
 \end{equation*}
where $B, x$ and $u$ satisfy \eqref{eq:ellipticfiber}. 
% \begin{equation*}
%   B x(x-1)(x-p_0) + u^2 =0.
%  \end{equation*}
(Note that when this elliptic  equation holds, the upper right entry is a polynomial.)  There is no further gauge freedom.

\medskip
 
If $\cE \simeq \cO(-1) \oplus \cO(-3)$, then 
 \begin{equation*}
  \End \cE \simeq \begin{pmatrix} \cO & \cO(2) \\ \cO(-2) & \cO \end{pmatrix}
 \end{equation*}
and the entries of the Higgs field
\begin{equation*}
\varphi = \frac{\de z}{z(z-1)(z-p_0)}\begin{pmatrix} a(z) & b(z) \\ c(z) & -a(z) \end{pmatrix}.
 \end{equation*}
are polynomials with $\deg a \leq 2$, $\deg b \leq 4$  and $\deg c =0$, satisfying \eqref{eq:determinant}. From
this equation, $c \not \equiv 0$, otherwise the left side of \eqref{eq:determinant} has even degree while 
the right is nonvanishing of  odd degree. Now make the gauge transformation
\begin{equation*}
 g= \begin{pmatrix} -\I \frac{c^{1/2}}{B^{1/2}} & \I \frac{a(z)}{c^{1/2} B^{1/2}} \\ 0 & \I \frac{B^{1/2}}{c^{1/2}} \end{pmatrix}.
\end{equation*}
to make $\varphi$  off-diagonal with $c(z) =-B$. Thus each gauge orbit contains a representative of the form
\begin{equation*}
 \varphi  = \frac{\de z}{z(z-1)(z-p_0)}  \begin{pmatrix} 0 & z(z-1)(z-p_0)  \\ -B & 0 \end{pmatrix}= \begin{pmatrix} 0 &  1 \\-\frac{B}{z(z-1)(z-p_0)} & 0 \end{pmatrix} \de z.
\end{equation*}
\end{proof}

\subsubsection{\texorpdfstring{$\C^\times$}{C*}-action}
The moduli space $\cM$ admits 
$\C^\times$-action $(\cE, \varphi) \mapsto (\cE, \zeta \varphi)$. This action preserves 
the decomposition $\cM'= \cM'_{\mathrm{big}} \sqcup \cM'_{\mathrm{small}}$.  A coordinate for the Higgs bundles 
in $\cM'_{\mathrm{small}}$ is $B \in \C^\times$, and in terms of this, the $\C^\times$ action is 
%\begin{equation}
$B \mapsto \zeta^2 B.$
%\end{equation}
We have
\begin{equation*}
\zeta \; \varphi_B= g_\zeta \; \varphi_{\zeta^2 B} \;g_\zeta^{-1}, \quad   g_\zeta = \begin{pmatrix} 
\zeta^{-1/2} & \\ & \zeta^{1/2} \end{pmatrix}.
\end{equation*}
The Higgs bundles in $\cM'_{\mathrm{big}}$ are labeled by $(B, u, x) \in \C^\times \times \C \times \C$, and in 
these coordinates, the $\C^\times$ action is $ (B, u, x) \mapsto (\zeta^2 B, \zeta u, x)$. Now,
\begin{equation*}
\zeta \; \varphi_{(B, u, x)}= g_\zeta \; \varphi_{(\zeta^2 B, \zeta u, x)} \;g_\zeta^{-1},  \quad 
g_\zeta = \begin{pmatrix} \zeta^{-1/2} & 0 \\0 & \zeta^{1/2} \end{pmatrix}.
\end{equation*}

\subsection{Semiflat  metric} \label{subsec:semiflat4}
We now determine the semiflat metric: Proposition \ref{prop:sK} shows that in rescaled polar coordinates on the Hitchin base 
$\C_B$ the special K\"ahler metric is the flat conic metric of cone angle $\pi$
\begin{equation*}
 \frac{\de r^2 + r^2 \de \theta^2}{r}.
\end{equation*} Proposition \ref{prop:volume} shows that the flat torus fibers in 
the semiflat metric all equal $T^2_\tau$, normalized so that the area of $T^2_\tau$ is $4 \pi^2$. 

\subsubsection{Special K\"ahler metric}
\begin{prop}\label{prop:sK} 
The special K\"ahler metric on $\cB'$  is
\begin{equation*} \label{eq:sKonbase}
 g_{sK} = \frac{\de r^2 + r^2 \de \theta^2}{r},
  \qquad c_{sK} = \int_{\CP^1}  \frac{ \I \, \de z \wedge \de \zbar}{|z(z-1)(z-p_0)|},
\end{equation*}
in rescaled polar coordinates on $\C_B$ given by $r \e^{ i \theta} = c_{sK} B$.
This has a conic singularity of cone angle $\pi$ at $B=0$. 
\end{prop}
\begin{proof}
Fix $q=\frac{B}{z(z-1)(z-p_0)} \de z^2$, and 
consider the variation $\dot q = \frac{\dot{B}}{z(z-1)(z-p_0)} \de z^2 \in T_{q} \cB'$.
From the definition of the special K\"ahler metric in Proposition \ref{prop:sfcharacterization} and
$\dot\tau = \frac{\dot{q}}{2\sqrt{q}}$,
the special K\"ahler metric is 
\begin{equation} \label{eq:c-sK}
 \norm{\dot{q}}^2_{sK} = \frac{1}{2} \int_{\Sigma_B} \frac{|\dot{q}|^2\;}{|q|} = c_{sK} \frac{|\dot{B}|^2}{|B|},  \qquad \mbox{where } c_{sK} = \int_{\CP^1}  \frac{\I \, \de z \wedge \de \zbar}{|z(z-1)(z-p_0)|}.
\end{equation}
This yields the stated metric.
\end{proof}

\subsubsection{Area of fibers} In complex analytic terms, the nonsingular fibers are all $T^2_\tau$, and 
since the semiflat metric is flat on these fibers, it induces a constant multiple of the Euclidean 
metric on $T^2_\tau=\C/(\Z \oplus \tau \Z)$.  \emph{A priori}, we do not know this constant, or equivalently
the area of each fiber. 
\begin{prop} \label{prop:volume}
For $B \neq 0$, the area of $\mathrm{Hit}^{-1}(B) \simeq T^2_\tau$ in the semiflat metric equals $4\pi^2$,
so $T^2_\tau = \C_w/c_{\mathrm{fib}} (\Z \oplus \tau \Z)$ with Euclidean metric $g_{\mathrm{Euc}} =\de w \de \overline{w} 
= \de x^2 + \de y^2$ and \fixme{New value of $c_{\mathrm{fib}}$!}
\begin{equation}\label{eq:cfib}
 c_{\mathrm{fib}}= \frac{2 \pi}{\sqrt{\mathrm{Im}\; \tau}} . 
\end{equation}
\end{prop}
\begin{proof}
Recall from \S\ref{sec:limitingconfigurations} that the fibers of $\mathrm{Hit}$ are canonically torsors over
$\mathrm{Prym}(\Sigma_b, C)$. The identification of $\mathrm{Hit}^{-1}(B)$ with $\mathrm{Prym}(\Sigma_B, \CP^1)$ 
arises by tensoring by a holomorphic line bundle with fixed Hermitian metric, so the deformation spaces 
of $\mathrm{Hit}^{-1}(B)$ and  $\mathrm{Prym}(\Sigma_B, \CP^1)$ can be identified. 

Since $\mathrm{Jac}(\CP^1)$ is a point, $\mathrm{Prym}(\Sigma_B, \CP^1)=\mathrm{Jac}(\Sigma_B)$.
To describe the holomorphic line bundles in $\mathrm{Jac}(\Sigma_B)$ explicitly, let $L \rightarrow T^2_\tau=
\C_z/(\Z \oplus \tau \Z)$ be the trivial complex line bundle. The space of all holomorphic line bundles 
on $T^2_\tau$ is given by $\cL_\psi = \left(L, \delbar_{L,\psi} = \delbar + \psi \de \zbar \right)$ 
where  %$\psi$ is a constant in the torus
\begin{equation*} 
\psi \in \C/\frac{\pi}{\mathrm{Im}\;\tau}\left(\Z \oplus \tau \Z\right)
\end{equation*}
is constant. (To see that $\cL_\psi$ is trivial if and only if $\psi \in \frac{\pi}{\mathrm{Im}\;\tau} \left(\Z \oplus 
\tau \Z\right)$, consider the space of global holomorphic sections of $\cL_\psi$. If $\delbar_{L,\psi}f=0$, then 
$f(z, \zbar)=h(z) \e^{2\I \psi \mathrm{Im}\;z }$for $h(z)$ some holomorphic function $h(z)$.  Because $f(z, \zbar)$ is 
$\Z \oplus \tau \Z$-periodic, it follows that $h(z+1)=h(z)$ and $h(z+\tau) = h(z)\e^{- 2 \I \psi \mathrm{Im}\; \tau}$. 
One possible solution is $h(z) = \lambda \e^{2 \pi \I m z}$ for $m \in \Z$ provided there is some $n \in \Z$ such 
that $\pi (m \tau +n) = -\psi  \mathrm{Im} \; \tau$. Thus there is a basis $\{h(z) = \lambda \e^{2 \pi \I m z}\}_{\lambda \in \C}$ 
of global holomorphic sections of $\cL_\psi$ trivializing $\cL_\psi$ if and only if $\psi$ lies in the claimed lattice.) 

The tangent space to $\mathrm{Jac}(\Sigma_B)$ at $(\lambda, \delbar_L)$ is 
\begin{equation*}
 T_{(\lambda, \delbar_L)}\mathrm{Hit}^{-1}(B) =\{ (\dot{\lambda}=0 , \dot{\xi} = \dot{\psi} \de \zbar) \} \simeq \C_{\dot{\psi}}.
\end{equation*}
For each deformation of the Higgs bundle spectral data, the associated deformation of the Hermitian-Einstein metric is trivial.
The semiflat metric is characterized in Proposition \ref{prop:sfcharacterization} by the property that on vertical deformations the semiflat metric is 
$ \int_{\Sigma_B} 2 \norm{\dot\xi - \delbar_{L} \dot{\nu}_L}^2.$
%\end{equation}
Hence
\begin{equation*}
 \norm{(0, \dot{\psi} \de \zbar)}^2_{g_{\semif}} =\int_{T^2_\tau} 2 \left\|\dot \psi \de \zbar\right\|^2=4 \, 
\mathrm{Im} \; \tau \norm{\dot{\psi}}^2 .  
\end{equation*}
It follows that the metric on the torus
$ \C_\psi/\left(\frac{\pi}{\mathrm{Im}\;\tau} \Z \oplus \frac{\pi}{\mathrm{Im}\;\tau} \tau \Z\right)$ equals 
$g =4 \, \mathrm{Im}\; \tau|\de\psi|^2$ and this has area
\begin{equation*}
 \mathrm{Vol}_{\semif}(T^2_\tau) = \left(\mathrm{Im} \;\tau\right) \left(\frac{\pi}{\mathrm{Im}\;\tau}\right)^2 
\left(4 \, \mathrm{Im}\; \tau\right)= 4 \pi^2.
\end{equation*}
\end{proof}

\subsection{Gaiotto-Moore-Neitzke's conjecture on the four-punctured sphere}\label{GMN}
\begin{figure}[h] 
\begin{centering} 
\includegraphics[width=2in]{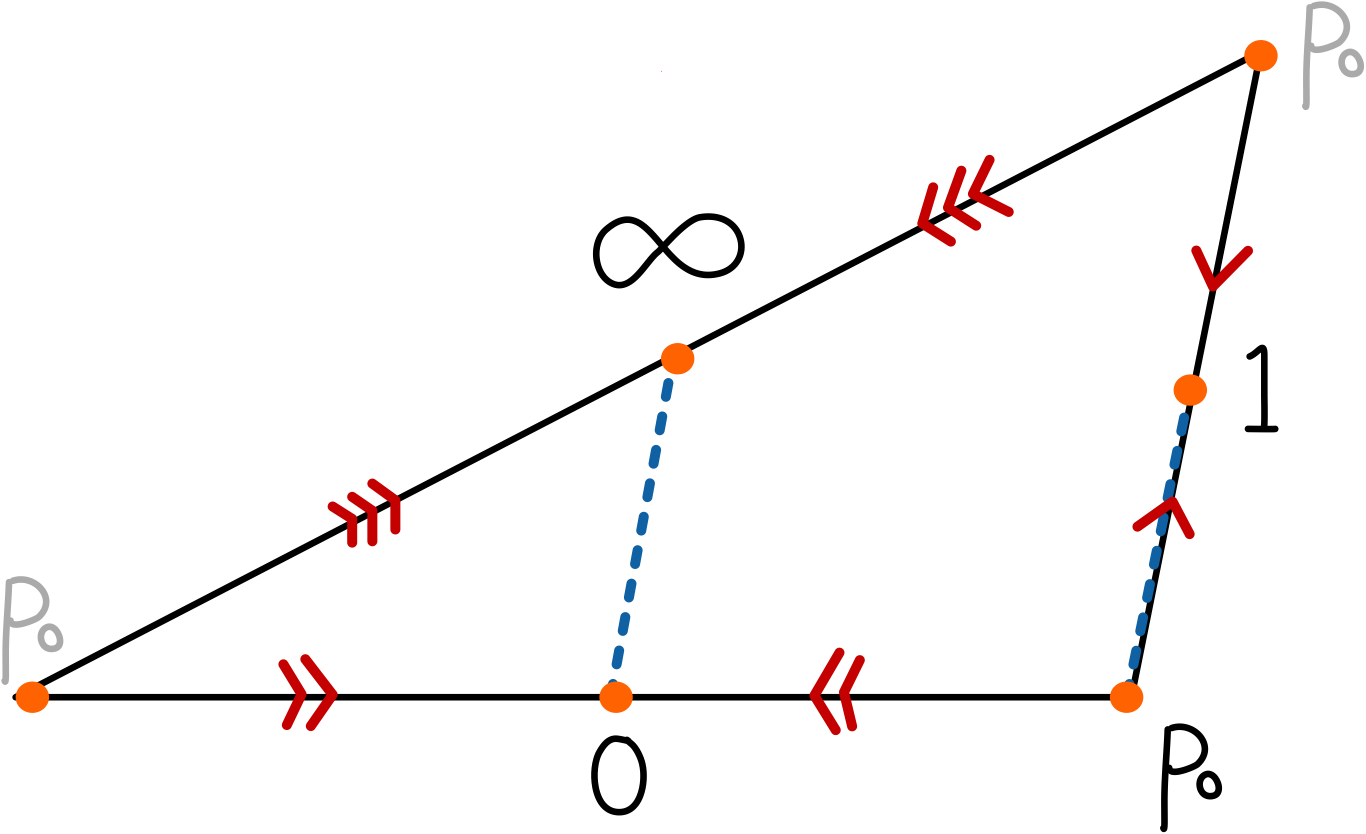} 
\includegraphics[width=2in]{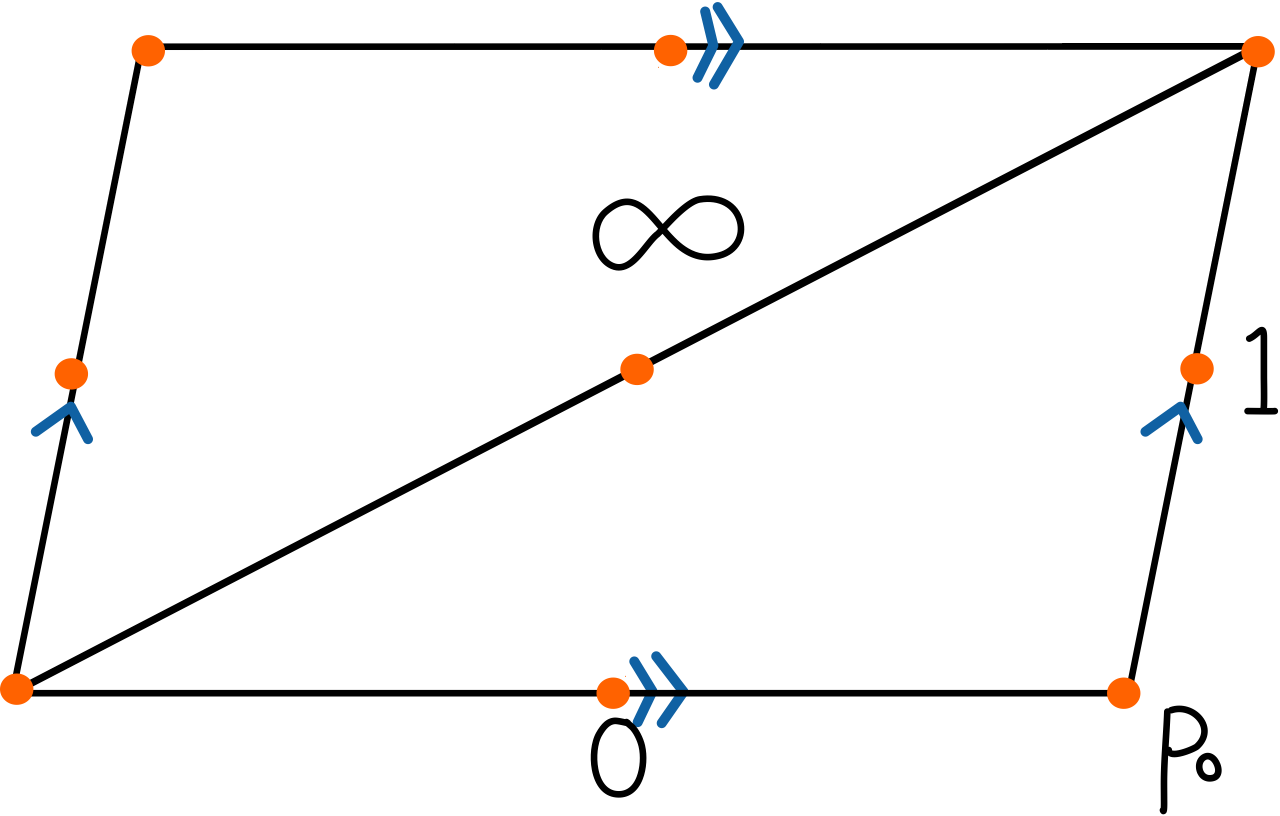}
\caption{\label{fig:flat}
(\textsc{Left}) The four-punctured sphere with flat metric with cone points of angle $\pi$ at each marked point, 
from identifying the edges of a triangle.
\textsc{(Right)} The induced flat metric on $\Sigma$, the double cover, from identifying opposite sides of the parallelogram. The map $\pi: \Sigma \rightarrow \CP^1$ is given by quotienting by the $\Z/2$-action generated by rotating $180^\circ$ around $\infty$.
}
\end{centering}
\end{figure}

Fix a holomorphic quadratic differential $q=\frac{B}{z(z-1)(z-p_0)} \de z^2 \in \mathcal{B}'$, and consider the ray of 
holomorphic quadratic differentials $t^2 q$. The optimal coefficient of exponential decay is $-2M_Bt$ where $M_B$ 
is the shortest geodesic on $\Sigma_B$, as measured in the flat metric $\pi^*|q|$. We can easily compute the length 
of this geodesic because with respect to this metric, the spectral cover is the flat torus $T^2_\tau=\C/(\Z \oplus \tau \Z)$ of area 
%\begin{equation}
 %2 \mathrm{Area}_{|q|} \CP^1, \qquad \mathrm{Area}_{|q|} \CP^1 = \int_{\CP^1} \frac{|B|}{|z(z-1)(z-p_0)|}  \frac{\I}{2} \de z \wedge \de \zbar=
$ |B| c_{sK}$,
%\end{equation}
with $c_{sK}$ as in \eqref{eq:c-sK}. Hence the correct scaling for this flat metric is
%metric on $\C$, the universal cover of $\Sigma_B$ with metric $\pi^*|q|$,  is 
\begin{equation*} 
\frac{|B|c_{sK} }{\mathrm{Im} \tau}(\de x^2 + \de y^2). 
\end{equation*}
Since $\tau$ lies in the usual fundamental domain $\{ |\tau| \geq 1,\ -\frac{1}{2}<\mathrm{Re }\; \tau< \frac{1}{2}\}$
for $\PSL(2, \Z)$, the length of the shortest geodesic in $T^2_\tau$ with unscaled metric $(\de x^2 + \de y^2)$ is $1$,
and the area of $T^2_\tau$ is $\mathrm{Im}\; \tau$. Thus with respect to the correctly scaled metric above, the shortest 
geodesic has length 
\begin{equation*}
M_B =   \sqrt{\frac{|B| c_{sK}}{\mathrm{Im} \tau}}
\end{equation*}

\bigskip

The full Gaiotto-Moore-Neitzke conjecture also specifies coefficients of the exponentially decaying terms given by 
BPS indices. \emph{The following BPS counts for the four-punctured sphere were communicated to the authors by Andy Neitzke.} 
Since $\mathrm{Jac}(\CP^1)$ is a point and $D_w=\emptyset$, $\Gamma_B=H_1(\overline{\Sigma}_B, \Z)$.
Given a saddle connection  $\wp_{\mathrm{saddle}}$  connecting the punctures $D  \subset \CP^1$, the corresponding lift $\gamma_{\mathrm{saddle}}$ is a primitive class
in $\Gamma_B=H_1(\overline{\Sigma}_B, \Z)$ and contributes $\Omega(\gamma_{\mathrm{saddle}})=8$ to the BPS count.
Consider a path $\wp_{\mathrm{pair}}$  which separates the $4$ marked points into $2 + 2$, and let $\gamma_{\mathrm{pair}}$ 
be the corresponding lift. (Note that there are infinitely many distinct possible homotopy classes for $\wp_{\mathrm{pair}}$.)
The BPS invariants receive contributions from an annulus of closed loops homotopic
to (as shown in \textsc{Figure} \ref{fig:BPS}): these give $\Omega(\gamma_{\mathrm{pair}})=-2$.
 \begin{figure}[h] 
 \begin{centering} 
 \includegraphics[height=1.5in]{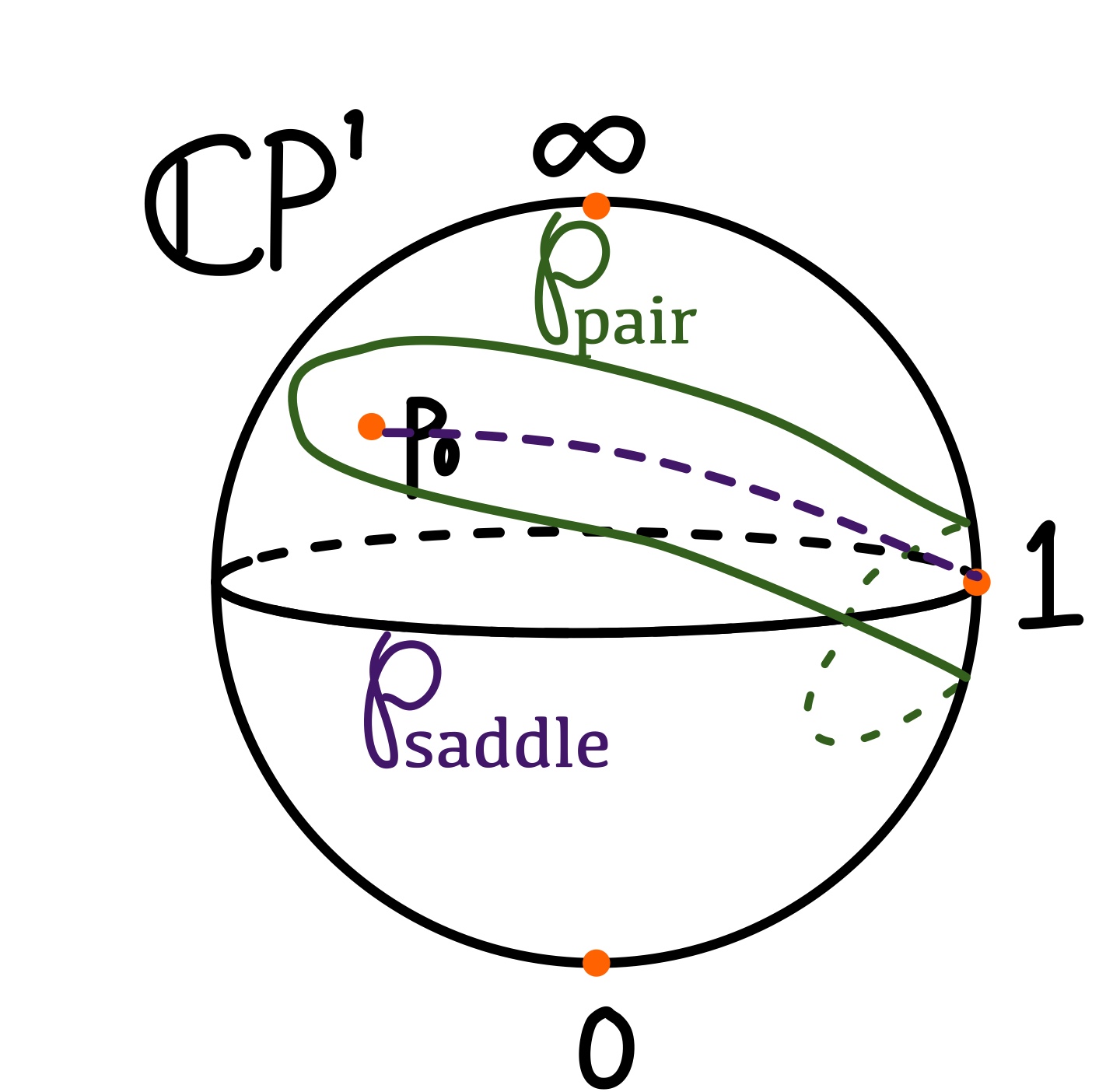} \qquad\qquad  \includegraphics[height=1.5in]{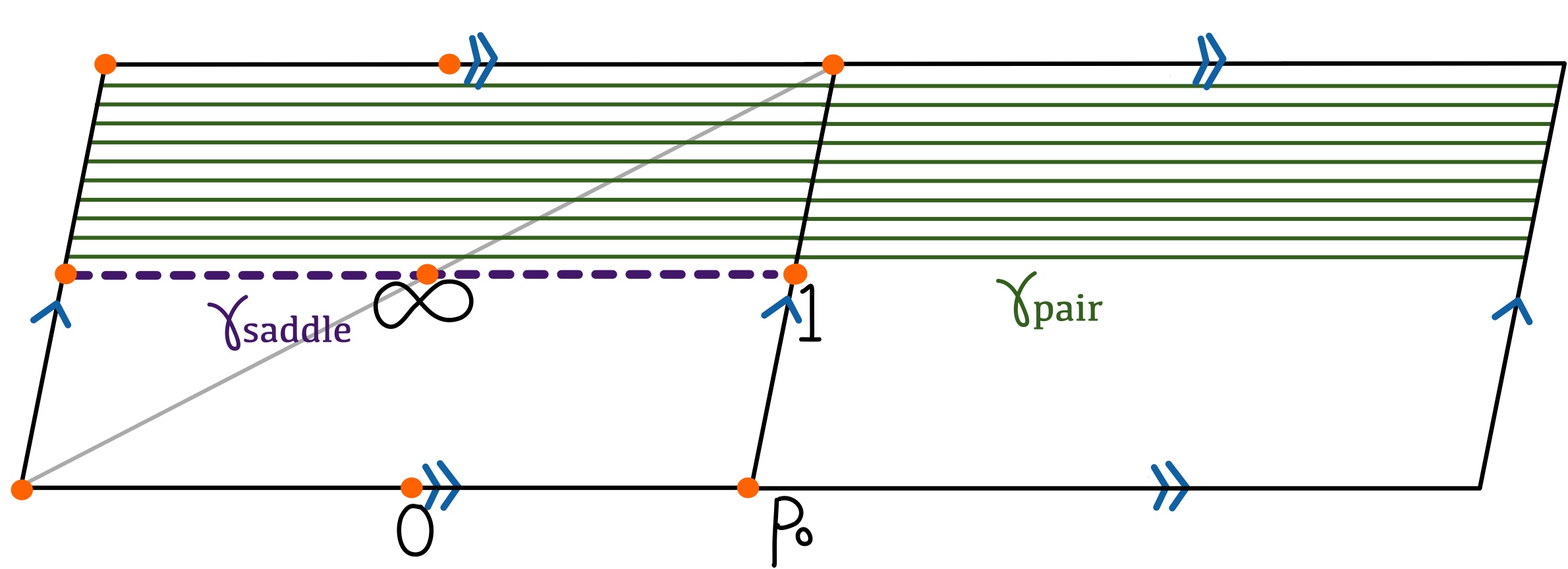}
 \caption{\label{fig:BPS}(\textsc{Left}) Paths on $\CP^1$ and (\textsc{Right}) corresponding paths on spectral torus.
 }
 \end{centering}
 \end{figure}
Note that in $H_1(\overline{\Sigma}_B, \Z)$, there is a correspondence between $\gamma_{\mathrm{saddle}}$ and $\gamma_{\mathrm{pair}}$ such that $2[\gamma_{\mathrm{saddle}}] = [\gamma_{\mathrm{pair}}]$, as illustrated in \textsc{Figure} \ref{fig:BPS}.
Altogether, for $\gamma \in \Gamma=H_1(\overline{\Sigma}_B, \Z)$ primitive, we have
\begin{equation*}
 \Omega\left(n\gamma\right) = \begin{cases} 8 \qquad &\mbox{for } n=1\\
 -2 &\mbox{for } n=2 \\ 0
 &\mbox{for } n>2.    
                    \end{cases}
\end{equation*}

Incorporating these BPS indices, we recall the version \eqref{eq:firstit} of the conjecture restricted to the Hitchin section, and 
supposing that the flat spectral torus has a unique shortest geodesic, i.e., $\mathrm{Im}\; \tau \neq 1$.
(This includes both the square torus corresponding to $p_0=\frac{1}{2}, \tau = \I$ and the double cover of 
the equilateral triangle corresponding to $p_0 = \tau = \e^{\frac{\I \pi}{3}}$.)  
\begin{conj}[Weak version of conjecture on Hitchin section]\label{conj:v2.0} If $\mathrm{Im}\; \tau = 
\mathrm{Im}\; \lambda(p_0) \neq 1$, then on the Hitchin section, and in polar coordinates $r\e^{i \theta} = c_{sK} B \in \C$
\begin{equation*} \label{eq:version2}
 g_{L^2} - g_{\semif} = -\frac{2}{\pi} \cdot 8 \cdot   K_0(2 \sqrt{r/\mathrm{Im}\;\tau})\frac{\de r^2 + r^2 \de \theta^2}{2 r \,\mathrm{Im}\, \tau}
+ O( e^{-\eta \sqrt{r}}),
%-\frac{2^{7/4}}{\pi^{1/2}(\mathrm{Im} \tau)^{3/4}} r^{-5/4} \e^{-2 \sqrt{\frac{2 r}{\mathrm{Im}\, \tau}}} 
%(\de r^2 + r^2 \de \theta^2) + O( r^{-7/4} \e^{-2 \sqrt{\frac{2 r}{\mathrm{Im}\, \tau}}} ).
\end{equation*}
where $\eta > 2\sqrt{2}/\mathrm{Im}\, \tau$.  Since $\de |Z|^2 = \frac{\de r^2 + r^2 \de \theta^2}{2 r \mathrm{Im}\, \tau}$,
this agrees with \eqref{eq:firstit}. 
\end{conj}

This formulation (which does not feature the parameter $t$) is of course sharper than one which only gives decay rates along 
radial paths $t^2 q \in \cB'$.  We prove this weaker form of the conjecture almost entirely, though unfortunately our methods
do not yield the coefficient of this leading exponential term.

\subsection{Proof of optimal rate of exponential decay}\label{optimal_decay} 
Theorem \ref{thm:L2vssemiflat} asserts that the difference between the Hitchin and semiflat metrics decays exponentially along rays 
in the base, but the decay rate in Theorem \ref{thm:L2vssemiflat} is far from optimal. We now establish the sharp predicted decay rate
on the four-punctured sphere in the strongly parabolic case. The key is that in this special case, 
$\dim_\R \mathcal{M}=4$ and the natural $\C^\times$-action on Higgs bundles restricts to a circle action preserving the Hitchin metric and one of
the complex structures, but rotating the other two. This symmetry makes it possible to write the Hitchin metric in terms of a 
single scalar function $u: T^2 \times \R^+ \rightarrow \R$, at least away from the fiber over the discriminant locus.
This reduction appears in LeBrun's paper \cite{LeBrun1}. The function $u$ satisfies the   first of the equations 
in \eqref{eq:LeBrun1} below. We now explain how to use this equation to deduce the optimal decay rate starting from any \emph{a priori} 
exponential decay. 

As noted above, the Gaiotto-Moore-Neitzke conjecture predicts the coefficient of the leading exponential term 
as a BPS index. We hope to return to a determination of this coefficient using global analytic techniques elsewhere.

\begin{figure}[h] 
\begin{centering} 
\includegraphics[height=1.5in]{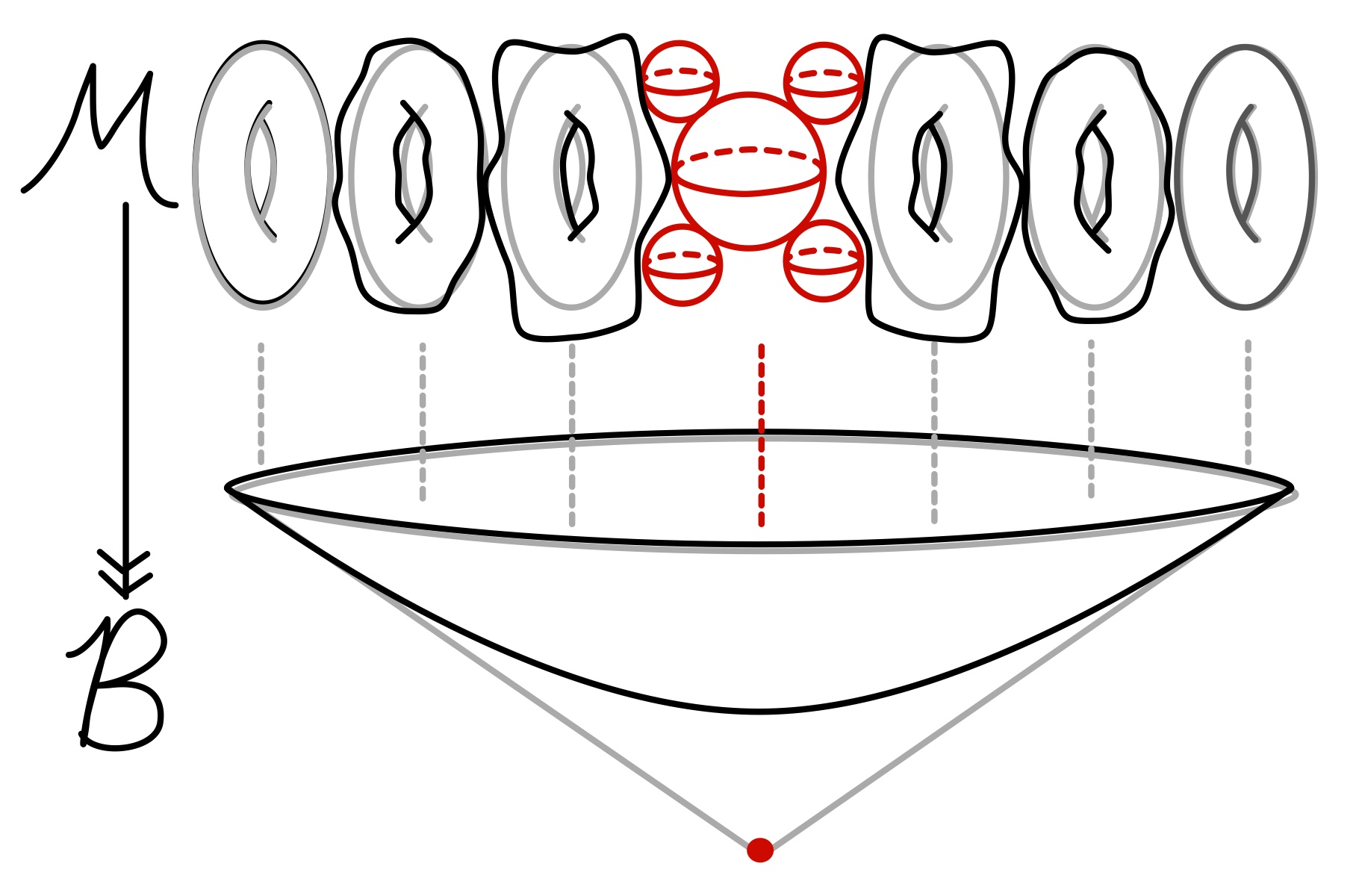}
\vspace{-.2in}
\caption{\label{fig:fourpuncturedschematic} Schematic of $g_{\semif}$ (shown in gray) versus $g_{L^2}$ (shown in black). 
}
\end{centering}
\end{figure}

\subsubsection*{Background: Reduction to a PDE}
Given a principal $\UU(1)$ bundle $M$ over an open set $U \subset \R^3$ and a positive harmonic function $w: U \rightarrow \R^+$, 
the Gibbons-Hawking ansatz produces a hyperk\"ahler metric $g$ on $M$ as follows: define a connection $\omega$ on $M$ with 
curvature $F=\star_{\R^3} \de w$ and then set
\begin{equation*}
 g = w g_{\R^3} + w^{-1} \omega^2.
\end{equation*}
LeBrun's formula generalizes this by describing all Ricci-flat K\"ahler metrics in two complex dimensions 
with a holomorphic circle action in terms of solutions of an elliptic equation. 
\begin{prop}\cite[Proposition 1]{LeBrun1}
Let $w>0$ and $u$ be smooth real-valued functions on an open set $U \subset \R^3$ which satisfy
\begin{align}\label{eq:LeBrun1}
u_{xx} + u_{yy} + (\e^{u})_{zz}&=0\\ \nonumber
w_{xx} + w_{yy} + (w\e^{u})_{zz}&=0.
\end{align}
Suppose also that the deRham class of the closed 2-form
\begin{equation*}
\frac{1}{2\pi} F = \frac{1}{2 \pi} \left( w_x \de y \wedge \de z + w_y \de z \wedge \de x + (w\e^u)_z \de x \wedge \de y \right)
\end{equation*}
is integral, i.e., lies in the image of $H^2(U, \Z) \rightarrow H^2(U, \R)$. Fix an $S^1$ bundle $M \rightarrow U$ such that $[c_1(M)]^{\R} 
= [\frac{1}{2\pi} F]$, and let $\omega$ be a connection $1$-form on $M$ with curvature $F$. Then 
\begin{equation*}
 g=\e^uw(\de x^2 + \de y^2) + w \de z^2 + w^{-1} \omega^2
\end{equation*}
is a scalar-flat K\"ahler metric on $M$; conversely, every scalar-flat K\"ahler surface with $S^1$-symmetry arises this way. 

The metric $g$ is Ricci-flat if, and only, if $u_z = cw$ for some constant $c$.
\end{prop}
\noindent This reduces to the usual Gibbons-Hawking formula when $u$ and hence $c$ both vanish.   

Since $\SU(2)=\Sp(1)$, hyperk\"ahler $4$-manifolds are Calabi-Yau and vice versa. Hence, this result 
describes all hyperk\"ahler metrics with a holomorphic circle action
\begin{cor}
Write $u_z=cw$. If $c \neq 0$, then the pair $(u,w)$ solves \eqref{eq:LeBrun1} if and only if $u$ solves 
\begin{equation*}
 u_{xx} + u_{yy} + (\e^{u})_{zz}=0.
\end{equation*}
\end{cor}
\begin{proof}
When $c \neq 0$, the equation for $w$ is obtained by applying $c^{-1}\del_z$ to the equation for $u$.  Clearly 
$c^{-1}(u_{xx})_z = w_{xx}$, $c^{-1} (u_{yy})_z = w_{yy}$ and $c^{-1}(\e^{u})_{zzz} = (w\e^{u})_{zz}$.
\end{proof}

\subsubsection*{The PDE for the four-punctured sphere}

In the strongly parabolic setting, both the semiflat metric $g_{\semif}$ and the Hitchin metric $g_{L^2}$ are hyperk\"ahler with  
circle action lifted from the Hitchin base which fixes one of the complex structures, and hence, using this formalism, we may 
write these metrics in terms of functions $u_{\semif}, w_{\semif}$ and $u, w$.   One subtlety is that the coordinate $z$ in LeBrun's 
formalism is the Hamiltonian vector field associated to the generator $\del_\theta$ for the $S^1$ action with respect to the 
given metric, i.e. $\de z = -\del_\theta \lrcorner \, \Omega$.  Since we are dealing with two different metrics, there are two radial 
coordinates, which we denote by $r$ and $\har$; these are associated to $g_{\semif}$ and $g_{L^2}$, respectively. In fact, we
have two distinct coordinate systems, $(r,\theta, x, y)$ and $(\hat{r}, \hat{\theta}, \hat{x}, \hat{y})$, with $\hat{r}(r,x,y)$ determined as 
above (independent of $\theta$ because of the $S^1$ symmetry) and $\hat{\theta} = \theta$, $\hat{x} = x$, $\hat{y} = y$.

The coordinates and functions associated to $g_{\semif}$ are particularly simple. Indeed, by Propositions \ref{prop:sK} 
and \ref{prop:volume}, we can write $T^2_\tau = \C_w/(c_{\mathrm{fib}}  (\Z \oplus \tau \Z))$, and 
%\begin{equation*}
$r=c_{\mathrm{sK}}|B| \in \R^+$. %\end{equation*} 
(In other words, the radial coordinate $r$ is simply the obvious one on the Hitchin base.)  We can also choose
$\omega_{\semif} = \de \theta$, $\theta=\mathrm{Arg}(B)$.  We then have
\begin{equation*}
u_{\semif}(x,y,r)= \log r, \qquad w_{\semif}(x,y,r)= r^{-1}. 
\end{equation*}

The estimate $g_{L^2} - g_{\semif} = O(\e^{-\varepsilon \sqrt{r}})$ and the definitions of $r$ and $\har$ above show that,
possibly normalizing $\har$ by an additive constant, 
\begin{equation}
\har - r = O(\e^{-\varepsilon \sqrt{r}}).
\label{wediffr}
\end{equation}

We now obtain asymptotic estimates for the functions $u$ and $w$ corresponding to the Hitchin metric $g_{L^2}$ in
terms of the function $\har$. 
\begin{prop}\label{prop:decay}
If $u:  T^2_{x,y} \times \R^+_\har$ satisfies 
\[
\Delta_{T^2} u+ \partial_\har^2\e^u=0.
\]
and $v = u - \log \har$ satisfies $|v| \leq C \e^{-\varepsilon \sqrt{\har}}$ as $\har \to \infty$, then 
\[
v(x,y,\har) = 
\har^{-1/2} K_1(2 \lambda_T \sqrt{\har}) (A_1 \cos( 2\pi(x,y) \cdot \mu_0)+ A_2 \sin (2\pi (x,y) \cdot \mu_0)) + O( \e^{-\eta \sqrt{\har}})
\]
for some $\eta > 2\lambda_T$ as $\har\to \infty$; here $K_1$ is the Bessel function of imaginary argument, $\lambda_T^2$ is 
the smallest positive eigenvalue of $-\Delta_{T^2}$ and $\cos( 2\pi(x,y) \cdot \mu_0)$ and $\sin( 2\pi(x,y) \cdot \mu_0)$ are 
the corresponding eigenfunctions.  The precise growth rate of this leading term is $\har^{-3/4} \e^{-2\lambda_T \sqrt{\har}}$.
\end{prop}
\begin{proof}
Substitute $u=\log \har + v$ into the equation to get 
\[
\Delta_T(\log \har + v)+ \partial_\har^2e^{v+\log \har}=\Delta_T v + \e^v(\har\partial_\har^2v+\har(\partial_\har v)^2+2\partial_\har v) = 0.
\]
We first transform this by multiplying both sides by $\har$ and setting $\har = \rho^2$. Since $\har\del_\har = \frac12 \rho \del_\rho$, 
this becomes, after some simplification, 
\[
e^v( \rho^2 \del_\rho^2 v + 3 \rho \del_\rho v + (\rho \del_\rho v)^2) + 16\rho^2 \Delta_T v = 0,
\]
which we write more simply as
\begin{equation}
L v :=  \rho^2 \del_\rho^2 v + 3 \rho \del_\rho v + 16 \rho^2 \Delta_T v =  Q(v, \rho \del_\rho v, \rho^2 \del^2_\rho v),
\label{nle}
\end{equation}
where $Q = (1 - \e^v)(\rho^2 \del_\rho^2 v + 3\rho \del_\rho v + (\rho \del_\rho v)^2)$. 

We first observe that if $|f(\rho, x, y)| \leq C_0$ in the region $\rho \geq \rho_0$, then for any $N > 0$, there
exists a constant $C_N$ such that
\begin{equation} \label{eq:Schauder}
\sup_{i + j + k \leq N}  |(\rho \del_\rho)^i (\rho \del_x)^j (\rho \del_y)^k v| \leq C_0 C_N.
\end{equation}
To prove this, note that $\rho\del_\rho, \rho\del_x, \rho \del_y$ are invariant with respect to the dilation
$(\rho, x, y) \mapsto (\lambda \rho, \lambda x, \lambda y)$ for any $\lambda > 0$, and hence so is the entire 
equation \eqref{nle}.  Thus we may estimate these scale-invariant derivatives in some region $2^{-k-1} < \rho < 2^{-k+1}$
and $|x-x_0| + |y-y_0| < 2^{-k}$ for some $(x_0, y_0) \in T^2$ by dilating by the factor $\lambda = 2^k$ and 
then invoking standard Schauder theory in $1/2 \leq \rho \leq 2$, where \eqref{nle} is uniformly elliptic.  In other words,
the a priori estimate when $\rho$ is small reduces to one in a region where $\rho \approx 1$. This argument reflects 
the fact that the operator is of `uniformly degenerate type', and this sort of rescaling argument is standard in 
that context, see \cite{Ma91}.  

We next obtain bounds for the solution to the inhomogeneous linear problem $Lv = f$. For this, decompose $v$ and $f$ into 
eigenfunctions on the torus. Write $T^2 = \R^2/\Lambda$ and denote by $\Lambda^{\smallvee}$ the dual lattice.  The exponentials 
$e^{2\pi i \mathbf{x} \cdot \mu}$, $\mu \in \Lambda^{\smallvee}$, give a complete basis of $L^2(T^2)$ by eigenfunctions of $\Delta_T$ 
with associated eigenvalues $4\pi^2 |\mu|^2$. (Of course, $v$ is real-valued so we really should be working with the real and
imaginary parts of these eigenfunctions.)  This reduces the problem to the family of equations $L_\mu v_\mu = f_\mu$ 
where $v_\mu$ and $f_\mu$ are the eigencomponents of $v$ and $f$, and
\[
L_\mu = \rho^2 \del_\rho^2 + 3 \rho \del_\rho - 16 \pi^2 |\mu|^2 \rho^2.
\]
This is essentially the Bessel equation.  There is a unique (up to constant multiple) solution which decays exponentially
as $\rho \to \infty$, namely $\varphi_\mu(\rho) = |\mu|^{1/2} \rho^{-1} K_1(4 \pi |\mu| \rho)$. This satisfies $\varphi_\mu(\rho) \sim 
C \rho^{-3/2} \e^{-4\pi |\mu| \rho}$, where $C$ is independent of $\mu$.  When $\mu = 0$, the unique (up to constant
multiple) solution decaying at infinity is $\varphi_0(\rho) = \rho^{-2}$. In terms of these we can write a particular solution to
the inhomogeneous equation as 
\[
v_\mu(\rho) = - \varphi_\mu(\rho) \int_a^\rho  \varphi_\mu(s)^{-2} s^{-3} \int_s^\infty \varphi_\mu(\sigma) f_\mu(\sigma) \sigma^3\,
\de \sigma \de s,
\]
where we may take $a = \infty$ if the outer integral converges.  The general solution is the sum of this $v_\mu$ and a multiple
of $\varphi_\mu$. 

Now suppose that $f \in \calC^\infty$ and for every $N \geq 0$, $|\del^N f| \leq C_N \e^{-\eta \rho}$ for some $\eta > 0$.   
Here (and below) $\del^N$ denotes any monomial of order $N$ in the vector fields $\rho \del_\rho, \rho\del_x, \rho \del_y$.  
For every $N \geq 0$, each eigencomponent of $f$ satisfies 
\[
|f_\mu| \leq C_{2N} (1 + |\mu|^2)^{-N} \e^{-\eta \rho}.
\]
We next make some elementary estimates: first, for $\mu \neq 0$, 
\[
\int_s^\infty  \sigma^{3/2} \e^{-4\pi |\mu|\sigma} \e^{-\eta \sigma} \sigma^3 \, d\sigma \leq C s^{3/2} \e^{-(4\pi |\mu| + \eta) s};
\]
next, if $\eta < 4\pi |\mu|$, we must take $a$ to be some finite number and find that
\[
\int_a^\rho  s^3 \e^{8\pi |\mu| s} s^{-3} s^{3/2} \e^{-(4\pi |\mu| + \eta)s}\, ds \leq 
C \rho^{3/2} \e^{ (4\pi |\mu| - \eta)s} + C,
\]
so finally, taking the product of this with $\varphi_\mu$, we conclude that
\begin{equation*}
|v_\mu(\rho)| \leq \sup |v_\mu| \rho^{-3/2} \e^{-4\pi|\mu|\rho} (C \rho ^{3/2} \e^{ (4\pi |\mu| - \eta)\rho} + C) \leq  C \sup |v_\mu| \, \e^{-\eta \rho}.
\end{equation*}
If $\eta > 4\pi |\mu|$, then we take $a = \infty$, and deduce once again that $|v_\mu(\rho)| \leq C \sup|v_\mu| \, \e^{-\eta \rho}$. 
(We may exclude the case that $\eta = 4\pi |\mu|$ for any $\mu$.) The final constant $C$ is independent of $\mu$.  We conclude 
from all of this that for any $N > 0$ and $\mu \neq 0$, 
\[
|v_\mu(\rho)| \leq C' C_{2N} (1 + |\mu|^2)^{-N}  \e^{-\eta \rho},
\]
along with corresponding estimates for any derivative $(\rho \del_\rho)^i v_\mu$.  When $\mu = 0$ we can obtain the same estimate
for $|v_0(\rho)|$ by a slightly more elementary calculation, taking $a = \infty$ in the outer integral.  Reassembling these components, we
obtain that if $\lambda_T = \min_{\mu \neq 0}  2 \pi |\mu| = 2 \pi |\mu_0|$, so that $2\lambda_T = 4\pi |\mu_0|$, then since 
the actual solution $v$ must be the sum of these particular solutions and some homogeneous solution, we have proved that
\[
|v(\rho, x,y)| \leq C \e^{-\eta \rho}\quad \mbox{if}\ \ \eta < 2\lambda_T
\]
and 
\[
v(\rho, x, y) =  A \varphi_{\mu_0}(\rho) \e^{2\pi i \mathbf{x} \cdot \mu_0} + O(\e^{-\eta \rho})\quad \mbox{if}\ \ \eta > 2\lambda_T.
\]
Note that 
\[
\varphi_{\mu_0}(\rho) =  \rho^{-3/2} \left(\sum_{j=0}^\infty a_j \rho^j\right) \e^{-2\lambda_T \rho},
\]
where the sum is convergent, and the remainder term decays at a higher exponential rate than the sum of this series. 
There is a  corresponding statement for all derivatives.  

Finally let us return to the nonlinear equation \eqref{nle}.   We start with the $\calC^0$ estimate $|v| \leq C \e^{-\varepsilon \rho}$.
The Schauder estimates imply that $|\del^N v| \leq C_N \e^{-\varepsilon \rho}$, and hence $|Q(v, \rho\del_\rho v, \rho^2 \del_\rho^2 v)|
\leq C' \e^{-2\varepsilon \rho}$.   Regarding $Q$ as the inhomogeneous term $f$ and applying the argument above, we obtain
that $|v| \leq C \e^{-2\varepsilon \rho}$ so long as $2\varepsilon < 2 \lambda_T$.    Iterating this argument a finite number of
times, and recalling at last that $\rho = \sqrt{\har}$, we conclude finally that the decomposition in the statement of this Proposition holds. 
\end{proof}

%The result above establishes a sharp decay estimate for the difference between the metric $g=u_{\har} \de \har^2 + u_{\har}^{-1} \omega^2 + \e^uu_{\har}(\de x^2 + \de y^2) $ and the metric 
%$\hat{g} = \frac{ \de\har ^2}{\har} + \har \de\theta^2 + \de x^2 + \de y^2$: 
\begin{lem} \label{lem:eigenvalue}
The smallest nonzero eigenvalue of $-\Delta_{T^2}$ on the flat torus $T^2_\tau$
in the semiflat metric is $\lambda_T^2 = 1 / \mathrm{Im} \tau$. 
Hence, $u- \log \har \sim T(x,y) \har^{-5/4} \e^{- \frac{2}{\sqrt{\mathrm{Im}\; \tau}}\sqrt{\har}} + O(\e^{\eta \sqrt{r}})$  for some $\eta >  2\sqrt{2}/\mathrm{Im}\, \tau$ as $\har\to \infty$.
%\begin{equation*}
%g- \hat{g} = O(\har^{-5/4} \e^{-2 \sqrt{\frac{2}{\mathrm{Im}\;\tau}} \sqrt{\har}}).
%\end{equation*}
%In fact, it expresses this difference in terms of a particular Bessel function plus a term decaying
%at a faster exponential rate. In any case, this is the precise constant of exponential decay in Conjecture \ref{conj:v2.0}.
\end{lem}
\begin{proof}
For a $2$-torus $\mathbb{T}^2 = \C/(\alpha \Z \oplus \beta \Z)$, the dual lattice is given by 
$\Lambda^{\smallvee} = \widehat{\alpha} \Z \oplus \widehat{\beta}\Z$ where 
\begin{eqnarray*}
 \widehat{\alpha}= \frac{\I \beta}{\mathrm{Im}(\alpha \overline{\beta})}, \qquad %\\ \nonumber
  \widehat{\beta}= -\frac{\I \alpha}{\mathrm{Im}(\alpha \overline{\beta})}.
\end{eqnarray*}
(If the lattice has basis $B$ the dual lattice has basis $(B^T)^{-1}$.)
In our setting, $T^2_\tau = \C/(c_{\mathrm{fib}} \Z \oplus c_{\mathrm{fib}} \tau \Z)$,
hence $\widehat{\alpha} = - \frac{\I \tau}{c_{\mathrm{fib}} \mathrm{Im}\; \tau}$,
$\widehat{\beta} =  \frac{\I}{c_{\mathrm{fib}} \mathrm{Im}\; \tau}$.
The smallest  value of $|\mu|$ for $\mu \in \Lambda^{\smallvee}$ is $\frac{1}{c_{\mathrm{fib}}\mathrm{Im}\;\tau}$, 
hence the smallest eigenvalue of $-\Delta_{T^2}$ is $\lambda_T^2 =(\frac{2\pi }{c_{\mathrm{fib}}\mathrm{Im}\;\tau})^2$. 
Using the value $c_{\mathrm{fib}} =  \frac{2\pi}{\sqrt{\mathrm{Im} \tau}} $ computed in \eqref{eq:cfib}, 
$\lambda_T = 1 / \sqrt{\mathrm{Im}\; \tau}$.
%hence $u- \log \har \sim C \har^{-5/4} \e^{-2 \sqrt{\frac{2}{\mathrm{Im}\; \tau}}\sqrt{\har}}$.  
\end{proof}
The constant of exponential decay of $u - \log \har$ (and hence of $w - \har^{-1}$) matches exactly the predicted one in Conjecture \ref{conj:v2.0} for $g_{L^2} - g_{\semif}$. However, this is not the end of the story, since $\hat r$ is not the correct radial coordinate and we have not yet taken care of the dependence of $g_{L^2}$ on the connection form $\omega$. 

To conclude the estimate, we must therefore pass from $\har$ to $r$ 
in such a way as to preserve this sharp estimate and obtain some control on the difference between $
\omega$ and $d\theta$.

%One additional remark: the conjecture predicts a specific coefficient of the leading exponential term here. Unfortunately, the present methods 
%do not yield this constant in any obvious way.

\begin{prop} 
There is a constant $A$ such that, restricted to the tangent bundle of the Hitchin section, 
\[
g_{L^2} - g_{\semif} =  A \, K_0(2 \sqrt{r/\mathrm{Im}\;\tau})\frac{\de r^2 + r^2
\de \theta^2}{r} + O(\e^{-\eta \sqrt{r}}),
\]
for some $\eta > 2\sqrt{2}/\mathrm{Im}\, \tau$.
In fact, there is a slightly more complicated expression for the difference
$g_{L^2} - g_{\semif}$ even away from the Hitchin section, which is the sum of terms involving the Bessel functions $K_0(2 \lambda_T \sqrt{r})$ and $K_1(2 \lambda_T \sqrt{r})$, which each decay like $ %\sqrt{\pi \lambda_T} 
r^{-1/4}\e^{- 2 \lambda_T \sqrt{r}}$, and a remainder which decays at a faster exponential rate. This expression will
be recorded at the end of the proof.  
\end{prop}
\begin{proof}  First note that, as $\har \to \infty$, 
\[
v =  \har^{-1/2} K_1(2 \lambda_T \sqrt{\har}) T(x,y) + O( \e^{-\eta \sqrt{\har}})
\]
for some $\eta > 2\lambda_T$, where $T(x,y)$ denotes the appropriate trigonometric factor.  Therefore 
\[
v_\har = \left( \lambda_T \har^{-1} K_1'(2 \lambda_T \sqrt{\har})- \frac12 \har^{-3/2} K_1(2 \lambda_T \sqrt{\har})\right) T(x,y)
+ O( \e^{-\eta \sqrt{\har}})
\]
Using the identity $z K_1'(z) - K_1(z) = - z K_2(z)$, this is the same as
\begin{equation}\label{eq:vr}
v_\har=-\lambda_T \har^{-1}K_2(2 \lambda_T \sqrt{\har}) T(x,y) +O(\e^{- \eta \sqrt{\har}})
\end{equation}

The principal difficulty is to find some sharp comparison of the functions $r$ and $\har$. To accomplish this, we
use the fact that these two metrics are K\"ahler for the same complex structure. This complex structure, which
we denote by $\cI$, satisfies (and is determined by): 
\[
\cI( \de r) = r \de \theta,\ \cI( \de \har) = w^{-1} \omega,\ \ \mbox{and}\ \ \ \cI(\de x) = \de y,\ \cI(\de y) = -\de x.
\]

Let us write 
\begin{equation*}
\begin{split}
\de \har & = a_1 \de r + a_2 \de x + a_3 \de y, \\ 
\de r & = b_1 \de \har + b_2 \de x + b_3 \de y, \ \ \mbox{and} \\
\omega & = \de \theta + c_1 \de \har + c_2 \de x + c_3 \de y.
\end{split}
\end{equation*}
Note that $\de \theta$ does not appear in the first two expressions and none of the $a_i$, $b_i$ or $c_i$ depend on $\theta$ 
because of the $S^1$ symmetry. Then
\begin{equation*}
\begin{split}
\cI \de \har = w^{-1} \omega & = w^{-1} (\de \theta + c_1 \de \har + c_2 \de x + c_3 \de y) \\ 
& = a_1 r\de \theta + a_2 \de y - a_3 \de x,
\end{split}
\end{equation*}
 from which we conclude that $a_1 = \frac{1}{rw}$ as well as 
 $c_1 = 0$, $c_2 = -w a_3$, $c_3 = w a_2$. 
Altogether, 
\begin{equation*}
\begin{split}
\de \har & = (rw)^{-1} \de r + a_2 \de x + a_3 \de y, \\
\de r & = (rw) \de \har -rwa_2 \de x -rwa_3 \de y, \ \ \mbox{and} \\
\omega & = \de \theta - w a_3 \de x + w a_2 \de y.
\end{split}
\end{equation*}

The Hitchin section is $\{(x,y) = (0,0)\}$, so $\de x = \de y = 0$ on its tangent bundle. This follows, since the image of the  Hitchin section is horizontal with respect to the Gauss-Manin connection $\nabla^{GM}$ on the Hitchin fibration and $\nabla^{GM}\de x = \nabla^{GM}\de y=0$ , cf.\ \cite[\textsection 3.2]{MSWWgeometry}.   Therefore, restricted to the Hitchin section,
\begin{equation*}
g_{L^2} = w \de \har^2 + w^{-1} \omega^2 = w ( (rw)^{-1} \de r)^2 + w^{-1} \de \theta^2 = \frac{1}{rw}\left( \frac{\de r^2}{r} + r \de \theta^2\right).
\end{equation*}
Hence,
\begin{equation*}
g_{L^2} - g_{\semif} = \left( \frac{1}{rw} - 1\right) \left( \frac{\de r^2}{r} + r \de\theta^2 \right).
\end{equation*}

We now conclude with a remarkable identity.  On the Hitchin section,
\[
\de \har = \frac{\de r}{rw}  \Longrightarrow w \de \har = \left( \frac{1}{\har} + v_{\har}\right) \de \har = \frac{\de r}{r},
\]
and hence by integration and choosing the constant of integration appropriately, $\log \har + v = \log r$, or equivalently, 
$r = \har \e^v = \har (1 + v + O(v^2))$.  Inserting the asymptotic expression for $v$ we find that
\[
r  = \har  (1 + \har^{-1/2} K_1(2\lambda_T \sqrt{\har}) T + \ldots) = \har + \har^{1/2} K_1(2\lambda_T \sqrt{\hat r})T + O(\e^{-\eta \sqrt{\hat r}})
\]
and so
\begin{multline*}
rw = \left(\har + \har^{1/2} K_1(2\lambda_T \sqrt{\hat r})T + \ldots \right) \left( \frac{1}{\har} - \frac{\lambda_T}{\har} K_2(2\lambda_T \sqrt{\har})T + \ldots \right) \\
= 1 + 2\lambda_T \left(\frac{1}{2\lambda_T \sqrt{\har}} K_1(2\lambda_T \sqrt{\har}) - \frac12 K_2(2\lambda_T \sqrt{\har})\right)T + \ldots
\end{multline*}
We now avail ourselves of the classical formula $K_0(z) - K_2(z) = -(2/z )K_1(z)$ and replace $\har$ by $r$ on the right hand side to see that 
\[
rw = 1 - \lambda_T K_0( 2\lambda_T \sqrt{r})T + O(\e^{-\eta \sqrt{r}}),
\]
and so, at long last,
\[
\frac{1}{rw} - 1 = \lambda_T K_0( 2\lambda_T \sqrt{r}) T + O(\e^{-\eta\sqrt{r}})
\]
for some $\eta > 2\lambda_T$. Substituting $\lambda_T= 1/\sqrt{\mathrm{Im} \, \tau}$, this is precisely the claim.

\medskip

Let us proceed to find (or at least estimate) the metric on the entire end of the moduli space.  For this we need to know
a bit more about the coefficients $a_2$ and $a_3$, which we can learn from the curvature equation $\de \omega = F$. This yields
the three equalities
\[
w_x = -\del_{\har} (w a_2),\ \ \ w_y = -\del_{\har} (w a_3),\ \ \ \ \del_x(w a_2)+\del_y(w a_3)= \del_{\har}( w e^u).
\]
% Since $\del_\har = (\del r/\del \har) \del_r = (1 + O( \e^{-\varepsilon \sqrt{r}})) \del_r$, and vice versa, we can replace $\har$ by $r$ 
% in these identities, at the price of a faster-decaying exponential. 
Now 
\[
w_x = -  \lambda_T \har^{-1} K_2(2 \lambda_T \sqrt{\har}) T_x + \ldots, \ \ \ w_y = -\lambda_T \har^{-1} K_2(2 \lambda_T \sqrt{\har}) T_y + \ldots,
\]
so integrating in from infinity, we obtain that $wa_2$ and $wa_3$ are of the form $\lambda_T \int \har^{-1} K_2(2 \lambda_T \sqrt{\har}) \de \har = -K_1(2 \lambda_T \sqrt{\har}) \har^{-1/2}$ multiplied
by $T_x$ or $T_y$, respectively, plus a faster exponential error.  The third curvature equation does not give additional information.

We can now integrate the equation for $\de \har$ along the $r$-radial rays to obtain that
\begin{align} \label{eq:rvshar}
r  =  \har + \har^{1/2} K_1(2\lambda_T \sqrt{\hat r})T + O(\e^{-\eta \sqrt{\hat r}})
%  \log r &= v + \log \har - \int w a_2 \de x  - \int w a_3 \de y\\ &= 
%  \har^{-1/2}K_1(2 \lambda_T \sqrt{\har}) T + \log \har - 2 \har^{-1/2} K_1(2\lambda_T \sqrt{\har})T + O(\e^{-\eta \sqrt{\har}}),
\end{align}
even away from the Hitchin
section.
With this, we can finally show that $g_{L^2}- g_{\semif}$ has the predicted rate of exponential decay of $g_{L^2}- g_{\semif}$ where 
\begin{align*}
 g_{L^2} &=  \e^u w ( \de x^2 + \de y^2 ) + w \de \har^2 + w^{-1} \omega^2\\
 g_{\semif} &=  (\de x^2 + \de y^2) + r^{-1} \de r^2 + r \de \theta^2.
\end{align*}
Here, $w=u_\har$ and 
\begin{equation*}
 v = u - \log \har = \har^{-1/2} K_1(2 \lambda_T \sqrt{\har}) T + O(\e^{-\eta \sqrt{\har}})
\end{equation*}
From this alone, we see 
\begin{align} \label{eq:diffmet1}
 \e^u u_{\har} = \e^v(1+\har v_\har)&=\left(1 + \har^{1/2} K_1(2 \lambda_T \sqrt{\har}) T\right)\left(1  -\lambda_T K_2(2 \lambda_T \sqrt{\har}) T \right) +  O(\e^{\eta \sqrt{\har}})\\ \nonumber 
 &=1 - \lambda_T K_0(2 \lambda_T 
 \sqrt{\har}) T + O(\e^{\eta \sqrt{\har}})\\ \nonumber
 &=1 - \lambda_T K_0(2 \lambda_T 
 \sqrt{r})T + O(\e^{\eta \sqrt{r}}).\\ \nonumber
 \end{align}
 In the last line we used the relation between $r$ and $\har$ in \eqref{eq:rvshar}.
 Similarly, we see that 
\begin{align*}
(w \de \har^2 + w^{-1} \omega^2) 
&= w \left( \frac{1}{rw} \de r +  a_2 \de x +  a_3 \de y \right)^2 + w^{-1}\left(\de \theta - w a_3 \de x + w a_2 \de y \right)^2 
\\
&=  \frac{1}{\har}\left(1 - \lambda_T  K_2(2 \sqrt{\har} \lambda_T)T  \right)\\
& \cdot \left( \frac{\har}{r}  \left(1 + \lambda_T  K_2(2 \sqrt{\har} \lambda_T) \right)T\de r  - \har^{1/2}K_1(2 \lambda_T \sqrt{\har})T_x \de x -\har^{1/2}K_1(2 \lambda_T \sqrt{\har})T_y \de y \right)^2 \\ 
&+ \har\left(1 + \lambda_T K_2(2 \lambda_T \sqrt{\har})\right)\left(\de \theta 
+ \har^{-1/2}K_1(2 \lambda_T \sqrt{\har})T_y \de x 
- \har^{-1/2}K_1(2 \lambda_T \sqrt{\har})T_x\de y \right)^2 \\
& + O( \e^{- \eta \sqrt{\har}})
\end{align*}
 To make these estimates in $r$ rather than $\har$, we use the relation between $r$ and $\har$ in  \eqref{eq:rvshar}. In particular, note that 
 %\begin{equation*}
  $K_i(2 \lambda_T \sqrt{\har}) =  K_i(2 \lambda_T \sqrt{r}) + O(\e^{- \eta \sqrt{r}})$.
% \end{equation*}
It follows that
%Double check this expression
\begin{align} \label{eq:diffmet2}
(w \de \har^2 + w^{-1} \omega^2)  &= \left(1 +\lambda_T K_0(2 \lambda_T \sqrt{r}) T \right) \frac{ \de r^2 + r^2 \de \theta^2}{r}\\ \nonumber
& + 2 r^{-1/2} K_1(2 \lambda_T \sqrt{r}) \left( -\de r \cdot(T_x \de x + T_y \de y) + r\de \theta \cdot(T_y \de x - T_x \de y) \right)\\ \nonumber
& + O(\e^{-\eta \sqrt{r}}).
\end{align} 
At long last, looking at \eqref{eq:diffmet1} and \eqref{eq:diffmet2}, we see every term appearing in the difference $g_{L^2}- g_{\semif}$ is of order  $O( \e^{-2 \lambda_T \sqrt{r}})$.    
This completes
the proof. 
\end{proof}

\subsection{ALG gravitational instantons}\label{ALG}
Chen-Chen have given a classification of noncompact complete connected hyperk\"ahler manifold $X$ of real 
dimension $4$ with faster than quadratic curvature decay \cite{ChenChen}. They prove that any connected 
complete gravitational instanton with this curvature decay must be asymptotic to one element in a short list 
of standard models which are torus bundles over the flat cone $\C_\beta$ of cone angle $2\pi \beta \in (0, 2\pi]$.
The list of possible torus bundles $E \rightarrow \C_\beta$ is quite restricted.
\begin{thm}
\cite[Theorem 3.11]{ChenChen} \cite[Theorem 3.2]{ChenChenIII}
Suppose $\beta \in (0, 1]$ and $\tau \in \mathbb{H}=\{ \tau | \mathrm{Im}\;\tau >0\}$ are parameters in the following table:
\begin{equation*}\label{eq:table}
\begin{array}{| l | c |  c | c | c | c | c | c | c |}
\hline
D & \mbox{Regular} &  I_0^* & II & II^* & III & III^* & IV & IV^* \\
\beta & 1 & \frac{1}{2} & \frac{1}{6} & \frac{5}{6} & \frac{1}{4} & \frac{3}{4} & \frac{1}{3} & \frac{2}{3} \\
\tau & \in \mathbb{H} & \in \mathbb{H} & \e^{2 \pi \I/3} & \e^{2 \pi \I/3} & \I & \I & \e^{2 \pi \I/3} & \e^{2 \pi \I/3}
\\\hline
\end{array}
\end{equation*}
Suppose $\ell>0$ is some scaling parameter. Let $E$ be the manifold obtained by identifying $(u,v) \simeq 
 (\e^{2 \pi \I \beta}u, \e^{2 \pi \I \beta}v)$ in the space 
\begin{equation*}
\{u \in \C: \mathrm{Arg}(u) \in [0, 2\pi \beta] \; \& \;  |u| \geq R\} \times \C_v/(\Z \ell + \Z \ell \tau).
\end{equation*}
$E$ together with a certain (see \cite[Definition 2.3]{ChenChenIII}) flat hyperk\"ahler metric $g_{\mathrm{mod}}$ is called the \emph{standard ALG model of type $(\beta, \tau)$}.

\medskip

Every ALG gravitational instanton $X$ is asymptotic 
to 
one of these standard models $(E, g_{\mathrm{model}})$.
Moreover, if $\beta = 1$, then $X$ \emph{is} the standard flat gravitational instanton $\C \times T^2_\tau$.
\end{thm}

\bigskip
For this four-dimensional family of Hitchin moduli spaces on $\CP^1$, the semiflat metric is the model metric for 
$\beta = \frac{1}{2}$ and $\tau=\lambda^{-1}(p_0)$. From the bounds on $g_{L^2}$ and all derivatives in 
\eqref{eq:Schauder}, it follows that each component of the Riemann curvature tensor is also exponentially decaying 
as one approaches the ends, hence the curvature decay hypothesis is satisfied. Consequently, these spaces
fit into the Chen-Chen classification of ALG gravitational instantons.

Let $\cN_{\beta, \tau}$ denote the moduli space of ALG gravitational instantons with faster than quadratic curvature decay
which are asymptotic to the $(\beta, \tau)$ standard model. Generic hyperk\"ahler metrics $g \in \cN_{\beta, \tau}$ 
decay at a polynomial rate to the model metric $g_{\mathrm{model}}$.  However, the hyperk\"ahler metrics from the
strongly parabolic Hitchin moduli spaces on the four-punctured sphere decay exponentially to the model metric.  
The four-dimensional family parameterized by parabolic data at $0, 1, p_0, \infty$ lies in a distinguished locus 
in $\cN_{\beta, \tau}$ consisting of those metrics with exponential decay. We shall investigate this further in 
a future paper.

\bibliography{hk-example}{}
\bibliographystyle{fredrickson}

\end{document}